\documentclass[11pt, reqno]{amsart}
\usepackage{amsmath,amsthm,amsfonts,amssymb,mathrsfs,bm,graphicx,stmaryrd}
\usepackage{epstopdf}
\usepackage{mathtools}
\usepackage{dsfont}
\usepackage{multicol}
\usepackage[colorlinks=true,linkcolor=blue,citecolor=blue,breaklinks]{hyperref}
\usepackage{url}
\usepackage{float}

\usepackage{breakurl}
\usepackage{bbm}
\usepackage{subcaption}
\usepackage{enumerate}
\usepackage{longtable}
\newcommand{\wt}{\widetilde}
\newcommand{\wh}{\widehat}

\newcommand{\dist}{\mathrm{dist}}
\newcommand{\scX}{\mathscr{X}}

\newcommand{\bz}{\text{\usefont{U}{bboldx}{m}{n}z}}
\newcommand{\bw}{\text{\usefont{U}{bboldx}{m}{n}w}}

\newcommand{\cX}{\mathcal{X}}
\newcommand{\cZ}{\mathcal{Z}}

\newcommand{\RR}{\mathbb{R}}
\newcommand{\CC}{\mathbb{C}}

\newcommand{\PP}{\mathbb{P}}

\newcommand{\cB}{\mathcal{B}}

\newcommand{\EE}{\mathbb{E}}
\newcommand{\NN}{\mathbb{N}}

\newcommand{\ZZ}{\mathbb{Z}}

\newcommand{\cI}{\mathcal{I}}

\newcommand{\A}{\text{\usefont{U}{bboldx}{m}{n}A}}
\newcommand{\bp}{\mathbbm{p}}
\newcommand{\giv}{\,|\,}
\newcommand{\inte}{\mathrm{int}}
\newcommand{\bphi}{\boldsymbol{\varphi}}
\newcommand{\bpsi}{\boldsymbol{\psi}}

\usepackage[letterpaper,hmargin=1.0in,vmargin=1.0in]{geometry}
\parskip 	\smallskipamount

\newtheorem{theorem}{Theorem}
\newtheorem{lemma}[theorem]{Lemma}
\newtheorem{conjecture}[theorem]{Conjecture}
\newtheorem{problem}[theorem]{Problem}

\newtheorem{definition}[theorem]{Definition}
\newtheorem{corollary}[theorem]{Corollary}
\newtheorem{proposition}[theorem]{Proposition}

\theoremstyle{definition}
\newtheorem{remark}[theorem]{Remark}
\newcommand{\ind}{\mathbbm{1}}

\newcommand{\cA}{\mathcal{A}}
\newcommand{\cC}{\mathcal{C}}
\newcommand{\cD}{\mathcal{D}}
\newcommand{\cE}{\mathcal{E}}

\newcommand{\cM}{\mathcal{M}}

\newcommand{\cG}{\mathcal{G}}
\newcommand{\cS}{\mathcal{S}}

\newcommand{\Cov}{\mathrm{Cov}}

\newcommand{\diam}{\mathrm{diam}}

\newcommand{\stab}{\mathrm{Stab}}

\newcommand{\bx}{\mathbbm{x}}
\newcommand{\by}{\mathbbm{y}}

\newcommand{\cW}{\mathcal{W}}
\newcommand{\cH}{\mathcal{H}}

\usepackage[backend=biber,style=alphabetic,doi=false,maxalphanames=10,maxnames=50]{biblatex}
\AtBeginBibliography{\small}

\begin{document}
\title[]{Strong confluence of geodesics in Liouville quantum gravity}
\author[]{Manan Bhatia}
\address{Manan Bhatia, Department of Mathematics, Massachusetts Institute of Technology, Cambridge, MA, USA}
\email{mananb@mit.edu}
\author[]{Konstantinos Kavvadias}
\address{Konstantinos Kavvadias, Department of Mathematics, Massachusetts Institute of Technology, Cambridge, MA, USA}
\email{kavva941@mit.edu}
\date{}
\maketitle
\begin{abstract}
  $\gamma$-Liouville quantum gravity ($\gamma$-LQG) constitutes a family of planar random geometries whose geodesics exhibit intricate fractal behaviour. As is observed in various planar models of random geometry as part of the phenomenon of geodesic confluence, geodesics in $\gamma$-LQG tend to merge with each other. In particular, in \cite{GM20}, it was established that in $\gamma$-LQG, geodesics targeted to a \emph{fixed} point do coalesce in the sense that any two such geodesics almost surely merge before reaching their common target. However, in view of the randomness inherent to the geometry, it is a priori possible that while geodesics targeted to a fixed point do coalesce, there exists a sequence of geodesics $P_n$ converging to an exceptional geodesic $P$ as $n\rightarrow \infty$ such that $P_n$ does not overlap with $P$ for any $n$. In this paper, we prove that this is not possible, thereby establishing a strong confluence statement for $\gamma$-LQG for all $\gamma\in (0,2)$. This extends the results obtained in \cite{MQ20} for $\gamma=\sqrt{8/3}$ to all subcritical values of $\gamma$. We discuss applications to the study of geodesic stars and geodesic networks and include a list of open questions.

\end{abstract}

\tableofcontents

\section{Introduction}
\label{sec:introduction}
Geodesic confluence, or the tendency of geodesics to merge with each other, is a cornerstone phenomenon observed in planar models of random geometry. This phenomenon was first investigated \cite{LeGal10} for Brownian geometry \cite{LeGal19}, a field which studies uniform planar maps and their scaling limits. Subsequently, versions of geodesic confluence (or coalescence) have also been established \cite{GM20} for the more general\footnote{Indeed, as established by Miller-Sheffield \cite{miller2021liouville}, Brownian geometry is equivalent to $\sqrt{8/3}$-LQG.} $\gamma$-Liouville Quantum gravity ($\gamma$-LQG) metrics \cite{DDG21}, and also for integrable growth models in the Kardar-Parisi-Zhang (KPZ) universality class \cite{KPZ86}.

The most basic form of confluence, by no means easy to prove, concerns the behaviour of geodesics to a \emph{fixed} point. Working with a random metric $D$ on the complex plane admitting geodesics, the question is the following-- fix a point $\bz\in \CC$ and consider a $D$-geodesic $P$ emanating from $\bz$ to some other point $w$. Suppose that we have a family of $D$-geodesics $P_n$ emanating from $\bz$ and converging to $P$ in the Hausdorff sense with respect to the metric $D$. Then, with $\cB_\varepsilon(z)$ denoting a $D$-ball of radius $\varepsilon$ around a point $z$, is it true that for any $\varepsilon>0$, we have $(P_n\setminus P) \cup (P\setminus P_n) \subseteq \cB_\varepsilon(\bz)\cup \cB_\varepsilon(w)$ for all $n$ large enough? In case the answer to the above is affirmative, we say that \emph{one-point} confluence holds, and indeed, the above is known to hold for all the models mentioned above \cite{GM20,BSS19,LeGal10}.%

However, due to the randomness inherent to the above models, the behaviour around fixed points $\bz$ does not necessarily capture the behaviour around all points. Indeed, it is often the case that a property which holds almost surely for a fixed point fails at an infinite collection of exceptional points, and a prominent example of such points are \emph{geodesic stars} (see e.g.\ \cite{LeGal22, Bha22, Dau23+,Gwy21}) -- these are points $z$ for which there exist two distinct geodesics $P',P$ emanating from $z$ which intersect only at $z$ (see Section \ref{sec:geodesic-stars})
In view of the above, one might formulate a stronger version of the confluence phenomenon, known simply as \emph{strong confluence}. For a random metric $D$ on the plane admitting geodesics, strong confluence is said to hold if the following holds. Almost surely, for all points $z,w$ and every geodesic $P$ between them, \emph{every} sequence of geodesics $P_n$ converging to $P$ in the Hausdorff sense satisfies that for any fixed $\varepsilon>0$, $(P_n\setminus P) \cup (P\setminus P_n)\subseteq \cB_\varepsilon(z)\cup \cB_\varepsilon(w)$ as $n\rightarrow \infty$. In practice, strong confluence is a very useful tool since it offers a route to prove results for all geodesics $P$ by instead choosing a sequence of typical geodesics $P_n$ converging to $P$ and only investigating the path properties of these typical geodesics.

It turns out that for both Brownian geometry (or equivalently, $\sqrt{8/3}$-LQG) and for the directed landscape \cite{DOV18,DV21}, which is the putative universal scaling limit of all models in the KPZ universality class, strong confluence has been established-- for the former, this was established in the work by Miller-Qian \cite{MQ20} while the latter is proved in \cite{BGH19}. However, strong confluence has not yet been established for $\gamma$-LQG with general values of $\gamma$, and the goal of the paper is to fill this gap.
\begin{theorem}[informal]
  \label{thm:2}
  For $\gamma$-LQG with $\gamma\in (0,2)$, strong confluence a.s.\ holds.
\end{theorem}
As we discuss in detail later in Section \ref{sec:stars}, the above result has applications to the study of geodesic stars and networks in $\gamma$-LQG. We note that while strong confluence for $\gamma$-LQG was not yet known prior to this paper, a version of confluence stronger than one-point confluence but weaker than strong confluence was established in \cite{GPS20}. Namely, it was shown that in the definition of one-point confluence, one can relax the condition that the geodesics $P_n$ converging to $P$ all emanate from $\bz$. That is, instead of only considering geodesics $P_n$ targeted at a fixed point $\bz$, one can also allow geodesics whose target lies in the vicinity of $\bz$, which we note is the point from which $P$ emanates.

\subsection{The $\gamma$-LQG metric and the main result}
\label{sec:gamma-lqg-metric}

In order to state the formal statement of the main result, we now give a quick introduction to $\gamma$-LQG and the metric associated to it. %
Starting with a Gaussian Free Field (GFF) $h$ on a domain $U\subseteq \CC$, the theory of $\gamma$-LQG studies the random geometry obtained by considering the Riemannian metric tensor ``$e^{\gamma h(x+iy)}(dx^2+dy^2)$''. Such random metrics arise in physics and are in particular expected to describe the scaling limits of a wide class of planar map models when conformally embedded in the plane.

Since the GFF $h$, being log-correlated, is not well-defined pointwise, it is not straightforward to make rigorous sense of the metric tensor ``$e^{\gamma h(x+iy)}(dx^2+dy^2)$''. Indeed, for $\gamma\in (0,2)$, even the volume form ``$e^{\gamma h(z)}d|z|^2$'', an a priori simpler object than the entire Riemannian metric, is defined via a renormalisation procedure resulting in the Gaussian Multiplicative chaos measures \cite{Kah85,DS09} associated to the GFF $h$. The strategy used is to first mollify $h$, for instance, by defining $h^*_\varepsilon(z)$ to be the convolution of $h$ with the heat kernel at time $\varepsilon^2/2$, and then define the volume form associated to $h^*_\varepsilon$ (here $h^*_\varepsilon$ does have well-defined pointwise values) and then argue that this measure, when appropriately renormalised, converges in the limit $\varepsilon\rightarrow 0$.

In the last decade, the $\gamma$-LQG metric $D_h$ associated to $h$ has been successfully defined \cite{DDDF20, GM21} via a renormalisation procedure as a limit of certain first-passage percolation metrics associated to the mollified field $h_{\varepsilon}^*$. Consider the domain $U=\CC$ and let $h$ be an instance of the whole plane GFF (see Section \ref{sec:gaussian-free-field}), and let $d_\gamma$ be the Hausdorff dimension of $\gamma$-LQG as defined in \cite{DG20}. One can then consider the prelimiting metric $D_{h,\varepsilon}$ defined by
\begin{equation}
  \label{eq:11}
  D_{h,\varepsilon}(u,v)= (a_\varepsilon)^{-1}\inf_{P\colon u\rightarrow v}\int_0^1e^{(\gamma/d_\gamma)h^*_\varepsilon(P(t))}|P'(t)|dt,
  \end{equation}
  where the infimum above is over all piecewise continuously differentiable paths $P$ from $u$ to $v$, and $a_\varepsilon$ is an appropriately chosen normalising constant. The construction of the $\gamma$-LQG metric amounts to arguing that the metrics $D_{h,\varepsilon}$ converge in probability, in the appropriate topology, to a limiting metric $D_h$, and this defines the Riemannian metric tensor associated to $\gamma$-LQG.

  In addition to the distances $D_h(u,v)$, one can also consider geodesics, which are shortest paths in the random geometry given by $D_h$. Indeed, given a metric $d$ on $\CC$, the length of a path $\eta\colon [a,b]\rightarrow \CC$ can be defined as
  \begin{equation}
    \label{eq:16}
    \ell(\eta;d)= \sup_{a=t_0<\cdots<t_n=b} \sum_{i=1}^n d(\eta(t_{i-1}), \eta(t_i)),
  \end{equation}
  where the supremum is over all $n\in \NN$ and all partitions $\{t_0,\dots,t_n\}$ of $[a,b]$, and for any $u,v\in \CC$, a geodesic refers to any continuous path $\Gamma_{u,v}$ from $u$ to $v$ for which $\ell(\Gamma_{u,v};D_h)=D_h(u,v)$. It can be shown that a.s.\ such geodesics exist for all $u,v\in \CC$ and in fact, they are a.s.\ unique for fixed $u,v$ \cite[Theorem 1.2]{MQ18} (see also \cite[Lemma 4.2]{DDG21}). Unless otherwise specified, we shall always work with unit speed parametrisations of such geodesics, that is, we shall have $\Gamma_{u,v}\colon [0,D_h(u,v)]\rightarrow \CC$ with $\ell(\Gamma_{u,v}\lvert_{[s,t]};D_h)=t-s$ for every $0\leq s\leq t\leq D_h(u,v)$. For a point $z\in \CC$ and radius $r>0$, $\cB_r(z)$ shall denote the (open) $D_h$-ball of radius $r$ around $z$. We are now ready to state the main result of this paper.%

  \begin{figure}
    \centering
    \includegraphics[width=0.8\linewidth]{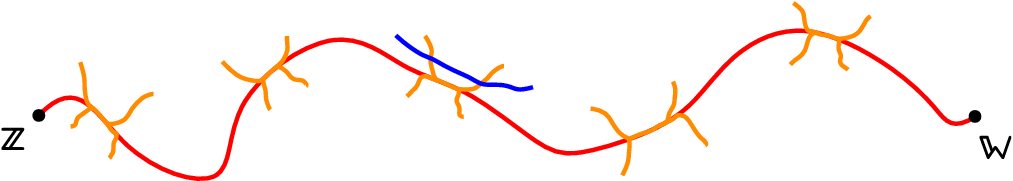}
    \caption{The geodesic $\Gamma_{\bz,\bw}$ is depicted in red and the multiple $\scX$s that $\Gamma_{\bz,\bw}$ passes through are depicted in orange. The blue path is a subsegment of another $D_h$-geodesic which is very close to $\Gamma_{\bz,\bw}$ with respect to the Hausdorff distance with respect to $D_h$.  If the blue segment did not ``coalesce'' with $\Gamma_{\bz,\bw}$, then it would have to cross one of the arms $P_+,P_-$ corresponding to the orange $\scX$ since the size of the $\scX$ is much larger than the Hausdorff distance between $\Gamma_{\bz,\bw}$ and the geodesic containing the blue path. But then, this would violate the uniqueness of the $D_h$-geodesics used to construct the $\scX$s.}
    \label{fig:chis_on_geod}
  \end{figure}
  \begin{theorem}[formal]
    \label{thm:3}
    Fix $\gamma\in (0,2)$ and consider a whole plane GFF $h$ and the associated LQG metric $D_h$. Fix $\varepsilon>0$. Then the following holds almost surely. For any points $u,v\in \CC$, any $D_h$-geodesic $P$ from $u$ to $v$ and any sequence of $D_h$-geodesics $\{P_n\}_{n\in \NN}$ converging to $P$ in the Hausdorff sense with respect to $D_h$\footnote{Or equivalently, with respect to the Euclidean metric and with Euclidean balls on the right side of \eqref{eq:14}. Indeed, since $\gamma\in (0,2)$, the $D_h$-metric is a.s.\ locally bi-H\"older with respect to the Euclidean metric (see Lemma \ref{lem:holder_regularity}).}, for all $n$ large enough, we have
    \begin{equation}
      \label{eq:14}
      (P_n\setminus P)\cup (P\setminus P_n)\subseteq \cB_\varepsilon(u)\cup \cB_\varepsilon(v).
    \end{equation}
  \end{theorem}
  In other words, if a sequence of geodesics eventually gets arbitrarily close to another geodesic, then they must eventually overlap for almost the entirety of their journeys.
  \begin{remark}
    With $d_H$ denoting the Hausdorff distance induced by $D_h$, we expect that it should be possible to carefully adapt the argument in this paper to yield an explicit constant $\beta\in (0,1)$ depending only on the LQG parameter $\gamma\in (0,2)$ for which the following quantitative version of Theorem \ref{thm:3} holds. Fix a bounded set $U\subseteq \CC$ and a constant $\alpha>0$. Then almost surely, for all small enough $\varepsilon>0$, for all points $u,v\in U$ with $D_h(u,v)\geq \alpha$, any $D_h$-geodesic $P$ from $u$ to $v$ and any $D_h$-geodesic $P'$ satisfying $d_H(P,P')\leq \varepsilon$, we have
    \begin{equation}
      \label{eq:105}
      (P\setminus P')\cup (P'\setminus P)\subseteq \cB_{\varepsilon^{\beta}}(u)\cup \cB_{\varepsilon^{\beta}}(v).
    \end{equation}
  We note that, in contrast, the work \cite{MQ20} in the setting of Brownian geometry ($\gamma=\sqrt{8/3}$), proves the above with $\varepsilon^\beta$ replaced by $\varepsilon^{1-o_\varepsilon(1)}$ (see \cite[Theorem 1.1]{MQ20}). We expect this stronger result to be true for $\gamma$-LQG for all $\gamma\in (0,2)$ as well, but this claim does not appear to be tractable by the methods used in this paper.
  \end{remark}
  We now state a particular striking consequence of Theorem \ref{thm:3} which concerns the geodesic frame $\mathcal{W}$ defined as the unions of interiors of \emph{all} possible $D_h$-geodesics; the geodesic frame $\mathcal{W}$ has been studied for Brownian geometry \cite{MQ20,AKM17} as well as the directed landscape \cite{Bha23}, and the following result is the $\gamma$-LQG analogue of \cite[Proposition 23]{Bha23} and \cite[Corollary 1.8]{MQ20}.
  \begin{proposition}
    \label{thm:4}
  Fix $\gamma\in (0,2)$ and consider a whole plane GFF $h$ and the associated LQG metric $D_h$. Almost surely, the Hausdorff dimension of the geodesic frame $\cW$ with respect to the LQG metric $D_h$ is equal to $1$.
\end{proposition}
On the way to establishing strong confluence, we shall also obtain the following interesting result on the non-existence of geodesic bubbles and the approximation of geodesics by typical points.
\begin{proposition}
  \label{prop:11}
  Fix $\gamma\in (0,2)$ and consider a whole plane GFF $h$ and the associated LQG metric $D_h$. %
  Then, almost surely, simultaneously for all $D_h$-geodesics $P: [0,T] \to \mathbb{C}$ and $0<s<t<T$, $P|_{[s,t]}$ is the unique $D_h$-geodesic between $P(s)$ and $P(t)$.
  
\end{proposition}
We note that the above is an analogue of the Brownian geometry result \cite[Theorem 1.3]{MQ20} and the result \cite[Theorem 1]{Bha23} (also independently proved in \cite[Lemma 3.3]{Dau23+}) in the setting of the directed landscape. Finally, we state the following result stating that any $D_h$-geodesic can be well-approximated by geodesics between typical points; this is an analogue of the Brownian geometry result \cite[Theorem 1.7]{MQ20}.
\begin{proposition}
  \label{prop:15}
   Fix $\gamma\in (0,2)$ and consider a whole plane GFF $h$ and the associated LQG metric $D_h$. Then, almost surely, simultaneously for all $D_h$-geodesics $P: [0,T] \to \mathbb{C}$ and $0<s<t<T$, for every $\varepsilon>0$, there exists $\delta>0$ such that for every point $u\in \cB_\delta(P(s)), v\in \cB_\delta(P(t))$, and every geodesic $\Gamma_{u,v}$, we have $P\lvert_{[s+\varepsilon,t-\varepsilon]}\subseteq \Gamma_{u,v}$ and $\Gamma_{u,v}\lvert_{[\varepsilon, D_h(u,v)-\varepsilon]}\subseteq P$.
\end{proposition}
As mentioned earlier, the results of this paper have applications in the study of geodesic stars and networks in $\gamma$-LQG and we refer the reader to Section \ref{sec:stars} for a detailed discussion.

\begin{figure}[]
    \centering
    \begin{subfigure}{0.4\textwidth}
        \centering
        \includegraphics[width=\textwidth]{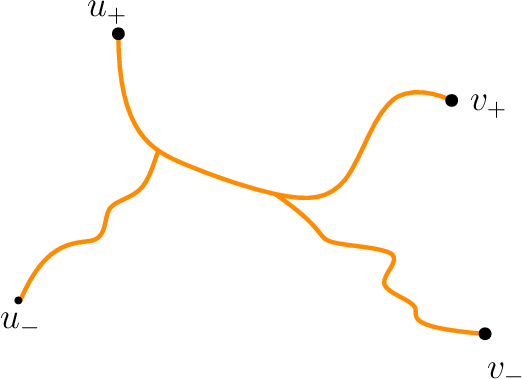}
    \end{subfigure}
    \hfill
    \begin{subfigure}{0.4\textwidth}
        \centering
        \includegraphics[width=\textwidth]{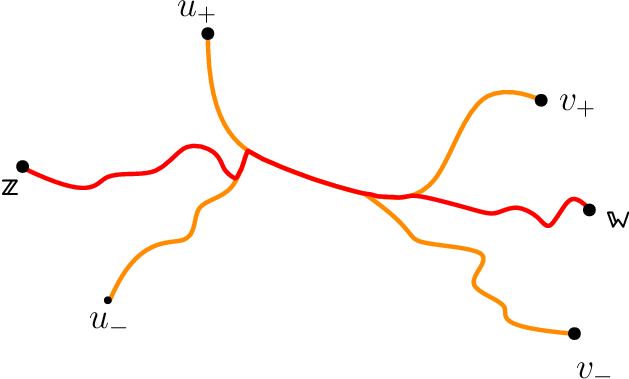}
    \end{subfigure}
    \caption{The left panel depicts the occurrence of the event $\scX=\scX^{u_+,v_-}_{u_-,v_-}$. The unique $D_h$-geodesic from $u_+$ to $v_+$ merges with the unique $D_h$-geodesic from $u_-$ to $v_-$ and after they merge,  the geodesics overlap on a segment up until they separate. After the separation point, the geodesics do not intersect again.  In the right panel, the $D_h$-geodesic $\Gamma_{\bz,\bw}$ (red) passes through the $\scX$ in the sense that $\Gamma_{u_+,v_+}\cap \Gamma_{u_-,v_-}\subseteq \Gamma_{\bz,\bw}$ with $\Gamma_{u_+,v_+}$ and $\Gamma_{u_-,v_-}$ being on opposite sides of $\Gamma_{\bz,\bw}$. %
    }
    \label{fig:two-panels}
\end{figure}

\subsection{Proof outline}
\label{sec:proof-outline}

We now give a broad overview of the strategy of the proof of Theorem \ref{thm:3}. As we shall see, the proof technique draws heavily from the proof of strong confluence in Brownian geometry by Miller-Qian \cite{MQ20}, where they considered the presence of a specific configuration which we denote as a $\scX$. For points $u_+,v_+,u_-,v_-\in \CC$, we say that $\scX^{u_+,v_+}_{u_-,v_-}$ occurs if the scenario described in the first panel of Figure \ref{fig:two-panels} occurs. Further, for points $\bz,\bw\in \CC$, we say that $\scX^{u_+,v_+}_{u_-,v_-}$ lies on the geodesic $\Gamma_{\bz,\bw}$ (or equivalently, $\Gamma_{\bz,\bw}$ passes ``through'' the $\scX$) if the scenario in the second panel holds. In \cite{MQ20}, the argument for strong confluence proceeds by combining two ingredients. First, it is shown that, almost surely, no two points in the space have infinitely many distinct geodesics connecting them. Secondly, it is argued that if two geodesics are very close to each other in the Hausdorff sense, then they must necessarily intersect. In fact, the former statement mentioned above is now also known \cite{Gwy21} for $\gamma$-LQG with general values $\gamma\in (0,2)$ (see Proposition \ref{prop:finitely_many_geodesics}). Thus, our primary task in this paper shall be to obtain a $\gamma$-LQG version (Proposition \ref{prop:12}) of the latter statement above-- namely, that geodesics close by in the Hausdorff sense necessarily intersect non-trivially.

In order to prove this for $\sqrt{8/3}$-LQG, the argument in \cite{MQ20} proceeds by showing that for any geodesic $\Gamma_{\bz,\bw}$ starting from a typical point $\bz$ to any point $\bw$ sufficiently far away from $\bz$, with very high probability, there are a significant number of $\scX$s (corresponding to some points $u_+,v_+,u_-,v_-$) lying on this geodesic (see Figure \ref{fig:chis_on_geod}).
Since the points $u_+,v_+,u_-,v_-$ can be taken such that the geodesics $\Gamma_{u_+,v_+},\Gamma_{u_-,v_-}$ for any $\scX$ as above are unique, the presence of such $\scX$s forces geodesics (see Figure \ref{fig:chis_on_geod}) which stay close to $\Gamma_{\bz,\bw}$ to non-trivially intersect $\Gamma_{\bz,\bw}$. To argue that numerous such $\scX$s are indeed present around $\Gamma_{\bz,\bw}$ with high probability, \cite{MQ20} heavily uses Markov exploration properties which are specific to the Brownian case ($\gamma=\sqrt{8/3}$). Roughly, by using the ``reverse metric exploration'' of a Brownian sphere, it is argued that one can divide $\Gamma_{\bz,\bw}$ into segments surrounded by mutually independent ``metric bands'', and it is then shown that in each such metric band, a $\scX$ as discussed above is present with positive probability.
\begin{figure}
  \centering
  \includegraphics[width=0.9\linewidth]{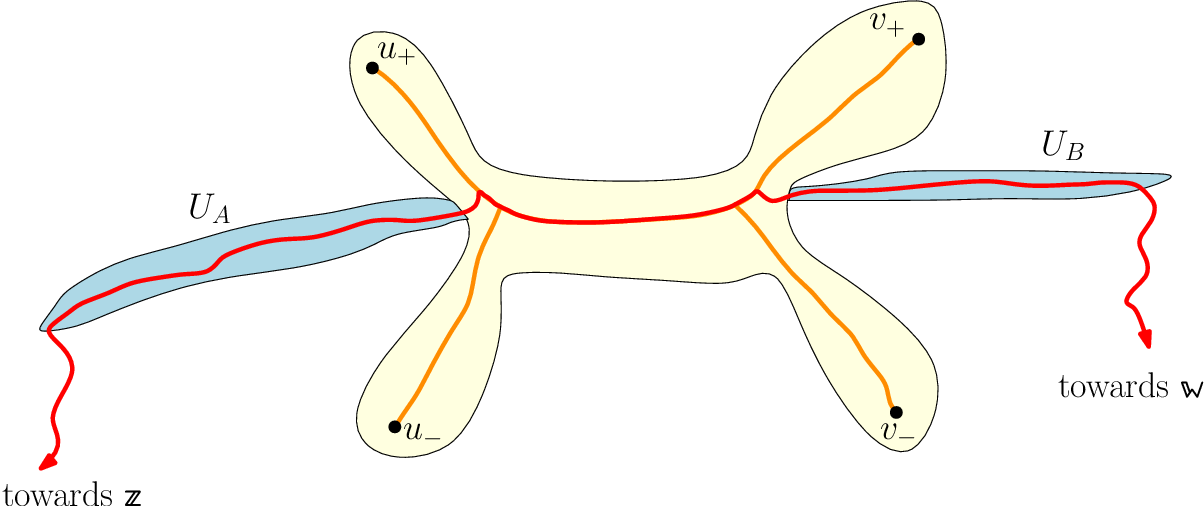}
  \caption{The yellow region is a finite union of small squares which intersect $\scX_{u_-,v_-}^{u_+,v_+}$ (orange colour).  The blue tubes ($U_A$ and $U_B$) come from relatively distant locations and respectively reach very close to the initial and terminal points of $\Gamma_{u_+,v_+}\cap \Gamma_{u_-,v_-}$. Note that $U_A$ and $U_B$ are sandwiched between the arms of the $\scX$. %
  The utility of introducing the bump function $\phi$ which is constant and takes its maximum value on the union of $U_A \cup U_B$ with the yellow region, is that it if the $D_h$ geodesic $\Gamma_{\bz,\bw}$ intersects both $U_A$ and $U_B$, then the $D_{h-\phi}$-geodesic (red colour) from $\bz$ to $\bw$  actually stays in $U_A$ until it enters the yellow region for the first time, then stays in the yellow region, and then finally continues for a while in $U_B$ before departing towards $\bw$.}
    \label{fig:chi}
\end{figure}
However, since the above explicit Markovian exploration is not available for $\gamma$-LQG when $\gamma \neq \sqrt{8/3}$, a different approach is necessary when trying to argue the presence of such $\scX$s for $\gamma$-LQG. Importantly, even in the case of $\gamma$-LQG, one-point confluence of geodesics \cite{GM20} does hold, and as argued in \cite{GM21}, this does lead to a certain degree of independence present along a $\gamma$-LQG geodesic, albeit much weaker than the independence used in \cite{MQ20}. Very roughly, suppose that for $z\in \CC$,  $E_r(z)$ is a locally determined event on which a $\scX$ is present inside the annulus $\mathbb{A}_{r/2,r}(z)=B_r(z)\setminus \overline{B_{r/2}(z)}$ and $\mathfrak{C}_r^{\bz,\bw}(z)$ is the event that the geodesic $\Gamma_{\bz,\bw}$ actually passes through the above $\scX$. %
Now, consider the following conditions--
\begin{enumerate}
\item \label{it:out1} The probabilities $\PP(E_r(z))$ can be made to be large enough so that $E_r(z)$ occur for a dense enough collection of $z$ and radii $r$ spanning many scales.
\item \label{it:out2} For fixed $z,r$, conditional on $\Gamma_{\bz,\bw}$ coming ``sufficiently close'' to $z$ and conditional on the occurrence of $E_r(z)$, $\mathfrak{C}_r^{\bz,\bw}(z)$ has probability bounded away from $0$ independently of $z,r$.
\end{enumerate}
If the above conditions are satisfied, then by an independence result proved in \cite{GM21} (see Proposition \ref{thm:theorem_uniqueness_lqg_metric}), the events $\mathfrak{C}_r^{\bz,\bw}(z)$ hold for a number of points $z$ and radii $r$ around the geodesic $\Gamma_{\bz,\bw}$ as it traverses from $\bz$ to $\bw$. Thus, the primary challenge to overcome in this paper, is to appropriately construct the locally dependent events $E_r(z)$, and the globally dependent events $\mathfrak{C}_r^{\bz,\bw}(z)$ in a way such that the two points above hold. As a prerequisite to obtaining \eqref{it:out1} above, one requires that for some collection of points $u_+,v_+,u_-,v_-\in \CC$, we have $\PP(\scX^{u_+,v_+}_{u_-,v_-})>0$, and this can be obtained by a soft application of the already known confluence result from \cite{GPS20}. The challenging part is to obtain \eqref{it:out2} above and this is done by starting with an instance of $h$ for which $E_r(z)$ holds and carefully modifying the field $h$ using bump functions to be small/large in certain regions so as to incentivise the geodesic $\Gamma_{\bz,\bw}$ to actually pass through the $\scX$ if it gets close enough to $z$, while ensuring that the above modification of $h$ does not disturb the metric $D_h$ enough so that the $\scX$ ceases to exist.

With reference to Figure \ref{fig:chi}, we now give a rough idea of how the above is achieved-- Suppose that for the field $h$, the event $E_r(z)$ which roughly demands the presence of a $\scX$ around $z$ at scale $r$ does occur. We now take a small (yellow) neighbourhood of this $\scX$ along with two thin (blue) tubes $U_A,U_B$ where $U_A$ and $U_B$ are disjoint and connect the ``left'' and ``right'' portions of the $\scX$ to a relatively distant location. The point now is to define an appropriate large non-negative bump function $\phi$ supported on the union of the yellow and blue regions. Further, the bump function is defined so that it is equal to its maximum value on the $\scX$. The effect of this is twofold:
\begin{enumerate}[(a)]
\item \label{it:out3} Since $\phi$ is positive and takes its largest possible value on the $\scX$, the $\scX$ (defined earlier for $D_h$) still remains a $\scX$ for $D_{h-\phi}$. Thus, the event $E_r(z)$ also occurs for the field $h-\phi$ instead of just the field $h$.
\item \label{it:out4} The field $\phi$ being very large and positive on the tubes $U_A, U_B$ makes it so that if a $D_h$-geodesic $\Gamma_{\bz,\bw}$ were to intersect $U_A, U_B$, then the corresponding $D_{h-\phi}$-geodesic (say $\Gamma_{\bz,\bw}^\phi$) would be incentivised to actually stay inside the union of the blue and yellow region in Figure \ref{fig:chi}, thereby passing through the $\scX$.
\end{enumerate}
Having done so, the strategy is to use the mutual absolute continuity between $h$ and $h-\phi$ to obtain item \eqref{it:out2}\footnote{We note that in \cite{GM21}, the above strategy of using bump functions along with absolute continuity was used to lower bound the conditional probability of a \emph{very different} event $\mathfrak{C}_{r}^{\bz,\bw}(z)$ than the one we shall define. In particular, the event $\mathfrak{C}_{r}^{\bz,\bw}(z)$ used in \cite{GM21} does not involve any $\scX$s. We refer the reader to Remark \ref{rem:comparison} for a discussion about this.}.
While the above picture captures the broad intuition, the precise definition of the above bump function and the corresponding event $\mathfrak{C}_r^{\bz,\bw}(z)$ is rather involved-- firstly, the above discussion treated the $\scX$ as if it were simply a deterministic object, whereas in actuality, it is a random set. To get around this issue, we undertake a discretisation of the plane into $\delta$-sized tubes and consider only $\scX$'s passing and staying inside such tubes. However, this approach of defining $E_r(z)$ so as to consider only $\scX$s which stay inside a given thin tube would prevent $\PP(E_r(z))$ from being large enough as is required in item \eqref{it:out1}, and for this, we do an iteration argument by considering multiple copies of such $\scX$s and arguing that with probability close to $1$, at least one of them is present (Section \ref{sec:tubes})-- this portion of the argument is modeled after the corresponding step in \cite{GM21}. Finally, to make \eqref{it:out4} above precise, one needs to ensure that it would be very costly for a $D_{h-\phi}$ geodesic not to utilise the yellow and blue regions and this is done by working on a high-probability regularity event on which the metric $D_h$ satisfies certain natural H\"older regularity conditions. A subtle point is that since the bump function $\phi$ is also large in a small neighbourhood of the $\scX$ itself, it is a priori possible that the $D_{h-\phi}$ geodesic between $\bz,\bw$ does not pass through the $\scX$ but instead approximately traces its boundary-- remaining very close to it at all times so as to take advantage of the bump function $\phi$ without ever touching the $\scX$ itself. Using estimates from \cite{DFGPS20,GM21} showing that $D_h$-geodesics are unlikely to stay restricted to deterministic narrow tubes, the above behaviour will be ruled out in Lemma \ref{lem:geodesic_intersects_chi}.

\textbf{Notational comments.} For a simple curve $\xi\colon [a,b]\rightarrow \CC$ viewed as a directed path from $\xi(a)$ to $\xi(b)$, we shall consider the left $\xi^{\mathrm{L}}$ and right $\xi^{\mathrm{R}}$ sides of $\xi$ and these are defined as collections of prime ends of the boundary of the simply connected domain $\CC\cup \{\infty\}\setminus \xi$. Thus, for every $r\in (a,b)$, we have $\xi^{\mathrm{L}}(r)\neq \xi^{\mathrm{R}}(r)$ while for $r\in \{a,b\}$, we have $\xi^{\mathrm{L}}(r)=\xi^{\mathrm{R}}(r)$. For a curve $\xi$ as above and another curve $\xi'\colon [a',b']\rightarrow \CC$ such that $\xi'(b')\in \xi((a,b))$ but $\xi'( [a',b'))\cap \xi([a,b])=\emptyset$, we have $\xi'(b')\in \xi^{*}$ for precisely one $*\in \{\mathrm{L},\mathrm{R}\}$ and we say that $\xi'$ is to the left (resp.\ right) of $\xi$ if $*=\mathrm{L}$ (resp.\ $*=\mathrm{R}$). For $z \in \mathbb{C}$ and $r>0$,  we shall use $B_r(z)$ to denote the open Euclidean ball of radius $r$ centered at $z$.  Moreover, for $0<r_1<r_2$, we shall often use the open annulus $\mathbb{A}_{r_1,r_2}(z) := B_{r_2}(z) \setminus \overline{B_{r_1}(z)}$. Finally, for a metric $d$ on $\CC$ and sets $U,V\subseteq \CC$, we shall set $d(U,V)=\inf_{u\in U,v\in V}d(u,v)$. Often, we shall use $\dist$ to denote the Euclidean metric on $\CC$, that is, for $z,w\in \CC$, $\dist(z,w)=|z-w|$. For $a<b\in \RR$, we shall use $[\![a,b]\!]$ to denote the discrete interval $[a,b]\cap \ZZ$.

Finally, in Section \ref{sec:construction_of_chi} of the paper, there will be a large number of different parameters and sets simultaneously in play; in order to make it easier for the reader to track their roles, we have provided a summary in Table \ref{table} at the end of the paper.

\textbf{Acknowledgements.} The authors thank Ewain Gwynne, Wei Qian for the discussions and Ewain Gwnyne, Jason Miller for comments on an earlier version of the manuscript. M.B.\ acknowledges the partial support of the NSF grant DMS-2153742 and the MathWorks fellowship.

\section{Preliminaries}
\label{sec:preliminaries}
We now give a quick introduction to the basic objects that we shall use in this paper.
\subsection{Gaussian Free Field}
\label{sec:gaussian-free-field}
Throughout this paper, we shall work with a whole plane GFF $h$ which can be defined as the centered Gaussian process $h$ with covariances given by
\begin{equation}
  \label{eq:15}
  \Cov(h(u),h(v))= G(u,v)=\log \frac{\max\{|u|,1\}\max\{|v|,1\}}{|u-v|}
\end{equation}
  for all $u,v\in \CC$. Due to the explosion of the above covariance kernel $G(u,v)$ along the diagonal, $h$ is not pointwise well-defined but is instead a.s.\ well defined as a Schwartz distribution in the sense that almost surely, for any bump function $\phi$, the average $(h,\phi)=\int h(u)\phi(u)du$ is well-defined. In fact, it is possible \cite[Proposition 3.1]{DS09} to define a version of $h$ such that for all $z\in \CC$ and $r>0$, the average of $h$ on $\partial B_r(z)$ is well-defined, and this quantity is denoted as $h_r(z)$, and is called the circle average-- in fact, we can fix a version of $h$ such that the map $(z,r)\mapsto h_r(z)$ is a.s.\ continuous. Further, we note that the normalisation in \eqref{eq:15} is conveniently done to ensure that $h_1(0)=0$ almost surely. Moreover we note that $h$ is scale and translation invariant when viewed ``modulo an additive constant'' in the sense that for every fixed $z \in \mathbb{C},  r>0$, the laws of the fields $h(\cdot),  h(\cdot + z) - h_1(z)$ and $h(r\cdot ) - h_r(0)$ are the same.

\subsection{The $\gamma$-LQG metric $D_h$}
\label{sec:gamma-lqg-metric-2}
As discussed in Section \ref{sec:gamma-lqg-metric}, the GFF $h$ comes associated with a metric on $\CC$; we now give a statement summarising the properties of the above, phrased for a GFF plus a continuous function $\mathtt{h}$, by which we mean $\mathtt{h}=h+\phi$ for a possibly random continuous function $\phi\colon \CC\rightarrow \RR$.
\begin{proposition}(\cite[Theorem~1.2]{GM21})
  \label{prop:8}
  For a GFF plus a continuous function $\mathtt{h}$, there exists a unique (up to a deterministic multiplicative constant) random metric $D_{\mathtt{h}}$ on $\CC$ which is measurable with respect to $\mathtt{h}$ and satisfies the following almost sure properties.
  \begin{enumerate}
  \item \label{it:length_space}
 \textbf{Length space}: Almost surely, for all $u,v\in \CC$, $D_\mathtt{h}(u,v)=\inf_{\eta}\ell(\eta;D_\mathtt{h})$, where the infimum is over all paths $\eta$ connecting $u,v$.
  \item \label{it:locality}
   \textbf{Locality}: For any deterministic set $U\subseteq \CC$, consider the induced metric $D_{\mathtt{h}}(\cdot,\cdot;U)$ defined by $D_{\mathtt{h}}(u,v;U)= \inf_{\eta\subseteq U} \ell(\eta; D_{\mathtt{h}})$,  where the infimum is taken over all paths contained in $U$. Then, for any fixed open set $U\subseteq \CC$, the induced metric $D_{\mathtt{h}}(\cdot,\cdot;U)$ is measurable with respect to $\mathtt{h}\lvert_U$. %
  \item \label{it:weyl_scaling}
  \textbf{Weyl scaling}: For a continuous function $f\colon \CC\rightarrow \RR$, consider the metric $e^f\cdot D_h$ defined by $e^f\cdot D_h(u,v)=\inf_{\eta}\int_{0}^{\ell(\eta;D_h)}e^{f(\eta(t))}dt$, where the infimum is over all paths $\eta$ from $u$ to $v$. Then, for $\xi=\gamma/d_\gamma$, with $d_\gamma$ denoting the fractal dimension of $\gamma$-LQG as defined in \cite{DG20}, we a.s.\ have the equality $e^{\xi f}\cdot D_{\mathtt{h}}=D_{\mathtt{h}+f}$ simultaneously for every continuous function $f\colon \CC\rightarrow \RR$.
  \item \label{it:coordinate_change}
   \textbf{Coordinate change formula}: Fix a deterministic $z\in \CC$ and $r>0$. With $Q$ defined by $Q=\gamma/2+2/\gamma$ and the field $\mathtt{h}'$ defined by $\mathtt{h}'(\cdot)=\mathtt{h}(r\cdot+z)+Q\log r$, almost surely, for all $u,v\in \CC$, we have $D_\mathtt{h}(ru+z,rv+z)=D_{\mathtt{h}'}(u,v)$.

  \end{enumerate}
\end{proposition}
The notation and the conclusions of the above proposition will be used extensively throughout the paper.

As in any geometric space, for any $z\in \CC$ and $r>0$, one can consider the metric ball $\cB_r(z;D_h)=\{w\in \CC: D_h(z,w)<r\}$. It turns out that such metric balls, though a.s.\ bounded, are not simply-connected and in fact have an infinite number of ``holes''. Thus, it is often useful to consider the filled metric ball $\cB^\bullet_r(z;D_h)$ which is defined as the complement of the unique unbounded connected component of $\CC\setminus \overline{\cB_r(z;D_h)}$. It can be shown that a.s.\ all filled metric balls $\cB^\bullet_r(z;D_h)$ are Jordan domains in the sense that their boundary $\partial\cB^\bullet_r(z;D_h)$ is a Jordan curve (see \cite[Lemma~2.4]{GPS20}). Throughout the paper, we shall abbreviate $\cB_r(z)=\cB_r(z;D_h)$ and $\cB_r^\bullet(z)=\cB_r^\bullet(z;D_h)$.%

\subsection{One-point confluence of geodesics}
\label{sec:geodesic-confluence}
We now discuss certain results from \cite{GM20,GPS20} on the one-point confluence of geodesics in $\gamma$-LQG. Before stating this result, we note that while there can exist multiple geodesics $\Gamma_{z,w}$ for certain exceptional points $z,w\in \CC$ if we start with a fixed $\bz\in \CC$, then almost surely, simultaneously for all $s>0$ and points $w\in \partial \cB_s^\bullet(\bz)$, there is a way to a.s.\ uniquely define a leftmost (resp.\ rightmost) such geodesic $\Gamma_{\bz,w}$  (see \cite[Lemma 2.4]{GM20}). In \cite{GM20} (see also \cite[Theorem 3.1]{GPS20}), the following result is proved regarding the confluence properties of geodesics targeted at a fixed point $\bz$.
    \begin{proposition}
      \label{prop:6}
      Almost surely, for any fixed $\bz\in \CC$ and simultaneously for all $0<t<s$, the following hold.
      \begin{enumerate}
      \item There is a finite set of points $\cX_{t,s}(\bz)\subseteq \partial \cB_t^\bullet(\bz)$ such that every leftmost $D_h$-geodesic from $\bz$ to a point on $\partial \cB_s^\bullet(\bz)$ passes via some $x\in \cX_{t,s}(\bz)$.
      \item \label{it:uniquegeod} There is a unique $D_h$-geodesic from $\bz$ to $x$ for each $x\in \cX_{t,s}(\bz)$.
      \item For $x\in \cX_{t,s}(\bz)$, let $I_x$ denote the set of $y\in \partial \cB_s^\bullet(\bz)$ for which the leftmost $D_h$-geodesic from $0$ to $y$ passes via $x$. Each $I_x$ for $x\in \cX_{t,s}(\bz)$ is a connected arc of the Jordan curve $\partial \cB_s^\bullet(\bz)$ and $\partial \cB_s^\bullet(\bz)$ is a union of the arcs $I_x$ for $x\in \cX_{t,s}(\bz)$, with any two such arcs being disjoint.
      \end{enumerate}
    \end{proposition}
        In fact, if one works with two stopping times instead of considering all $s<t$ simultaneously, the above statement can be upgraded to remove the leftmost condition; we state this result for $\bz=0$ for convenience (recall that $h_1(0)=0$ almost surely). %

    \begin{proposition}[{\cite[Proposition~3.2]{GPS20}}]
      \label{prop:7}
      Let $\sigma_1,\sigma_2$ be stopping times for the filtration $\cG$ defined by $\cG_s=\sigma(\mathcal{B}_s^{\bullet}(0) , h|_{\mathcal{B}_s^{\bullet}(0)})$ such that $\sigma_1 < \sigma_2$ almost surely. Then, almost surely, for every $D_h$-geodesic $P$ from $0$ to a point of $\CC\setminus \cB_{\sigma_2}^\bullet(0)$, there is a point $x\in \cX_{\sigma_1,\sigma_2}(0)$ such that $P(\sigma_1)=x$ and $P(\sigma_2)$ is an interior point of the arc $I_x$. %
    \end{proposition}

\begin{proof}
The claim in the statement of the proposition in the case that $\sigma_1,\sigma_2$ are deterministic times was shown in \cite[Proposition~3.2]{GPS20}. The case for general stopping times follows from the exact same argument since the main ingredient \cite[Lemma~3.5]{GPS20} holds for stopping times and not just deterministic times.   
\end{proof}

The proof of Proposition \ref{prop:6} uses a sophisticated ``barrier'' argument, wherein the filled metric ball is gradually exposed and the Markov property of the GFF is used to argue that there are sufficiently many ``barrier'' events which occur around bottlenecks of the boundary of the filled metric ball. In fact, for defining a regularity event later in this paper, we shall need some notation from the proof of the above, which we now introduce.

The argument in \cite{GM20} (see Section 3.2 therein) crucially uses certain high probability events $H_r(z)$ (therein called $E_r(z)$) for a radius $r>0$ and a point $z\in \CC$. Roughly, the event $H_r(z)\in \sigma( h-h_{3r}(z))\lvert_{\mathbb{A}_{2r,5r}(z)}$ is defined so that the following property holds: for a fixed $\bz\in \CC$ and conditional on $\sigma(\cB_t^\bullet(\bz),h\lvert_{\cB_t^\bullet(\bz)})$, for any $z$ such that $B_r(z)\cap\cB_t^\bullet(\bz)\neq \emptyset$, there is a strictly positive conditional probability that no $D_h$-geodesic from $\bz$ to a point outside of $\cB_t^\bullet(\bz)\cup B_{5r}(z)$ can enter $B_r(z)$ before intersecting with $\cB_t^\bullet(\bz)$.

The precise definition of $H_r(z)$ is complicated and will not be needed for this paper, and we shall refer the reader to \cite[Section 3.2]{GM20} for the details. For $\varepsilon>0$, we shall work with the radii $\rho_\varepsilon^n(z)$ defined by $\rho_\varepsilon^0(z)=\varepsilon$ and
\begin{equation}
  \label{eq:8}
  \rho_\varepsilon^n(z)=\inf\{r\geq 6\rho_{\varepsilon}^{n-1}(z): r=2^k\varepsilon \textrm{ for some } k\in \ZZ \textrm{ and } H_r(z) \textrm{ occurs}\}.
\end{equation}
The following result shows that one has uniform control over the above radii on a lattice of points.
\begin{lemma}[{\cite[Lemma 3.5]{GM20}}]
  \label{lem:5}
  There exists a constant $\eta>0$ for which the following holds. Using $\rho_{\varepsilon}(z)$ to denote $\rho_{\varepsilon}^{\lfloor \eta \log \varepsilon^{-1}\rfloor}(z)$, for each compact set $K\subseteq \CC$ and each $\bz\in \CC$, it holds with probability $1-O(\varepsilon^2)$ at a rate depending only on $K$ that
  \begin{equation}
    \label{eq:9}
    \rho_\varepsilon(z)\leq \varepsilon^{1/2} \textrm{ for all } z\in \left(\frac{\varepsilon}{4}\ZZ^2\right)\cap B_\varepsilon(K+\bz). 
  \end{equation}
\end{lemma}

Another useful property of the $\gamma$-LQG metric that we are going to use throughout the paper is the following.

\begin{lemma}(\cite[Theorem~1.7]{DFGPS20})\label{lem:holder_regularity}
Fix constants $\chi\in (0,\xi(Q-2))$ and $\chi'>\xi(Q+2)$. Then, almost surely, the identity map from $\mathbb{C}$ equipped with the Euclidean metric  to $(\mathbb{C} ,  D_h)$ is locally H\"older continuous with exponent $\chi$. Furthermore,  the inverse of this map is a.s.\ locally H\"older continuous with exponent $\chi'$.
\end{lemma}

Intuitively,  the above result is true for the following reason.  If $z$ is an $\alpha$-thick point for $h$ (see \cite{HMP10}),  i.e.,  the circle average satisfies $h_{\varepsilon}(z) = (\alpha + o_{\varepsilon}(1)) \log \varepsilon^{-1}$ as $\varepsilon \to 0$,  then the $D_h$-distance from $z$ to $\partial B_{\varepsilon}(z)$ behaves like $\varepsilon^{\xi (Q-\alpha) + o_{\varepsilon}(1)}$ as $\varepsilon \to 0$ (see \cite{DFGPS20}).  Further, in the context of a planar GFF, $\alpha$-thick points exist for $\alpha \in [-2,2]$ but not for $|\alpha| > 2$ \cite{HMP10}. 

\textbf{We now fix constants $\chi\in (0,\xi(Q-2))$ and $\chi'>\xi(Q+2)$ as in Lemma \ref{lem:holder_regularity}: these will stay fixed throughout the paper}.

\subsection{Independence along a geodesic and a regularity event}
\label{sec:indep-along-geod}
As outlined in Section \ref{sec:proof-outline}, in order to argue that with very high probability, $\scX$s occur around geodesics between fixed points, we shall use a result from \cite{GM21} obtaining a certain form of independence as one traverses along an LQG geodesic, and we now state the precise result that we shall use. Note that, throughout the paper, for $z\in \CC$ and $r>0$, we shall use $\tau_r(z)$ to denote the smallest $t>0$ for which $\cB_t^\bullet(z)\cap \partial B_r(z)\neq \emptyset$.
\begin{proposition}[{\cite[Theorem~4.2]{GM21}}]
  \label{thm:theorem_uniqueness_lqg_metric}
  Let $h$ be a whole-plane GFF. Fix $\nu>0$ and $0 < \lambda_1 < \lambda_2 \leq \lambda_3 \leq \lambda_4 < \lambda_5$. Then there exists $\mathbbm{p} \in (0,1)$ such that the following is true.  Suppose that for each $z\in \CC,r>0$, we have events $E_r(z) \in \sigma(h)$ and additionally for any points $\bz ,  \bw \in \mathbb{C}$, and any simple path $\eta$ from $\bz$ to $\bw$, we have events $\mathfrak{C}_r^{\bz,\bw,\eta}(z) \in \sigma(h)$, satisfying, for a deterministic constant $\Lambda > 1$, the following properties.

\begin{enumerate}[(1)]
\item \label{it:locality_property}
  For all $z \in \mathbb{C}$ and all $r >0$,  the event $E_r(z)$ is determined by $(h-h_{\lambda_5 r}(z))|_{\A_{\lambda_1 r ,  \lambda_4 r}(z)}$,  and for all $\bz,\bw \in \mathbb{C}$,  the event $\mathfrak{C}_r^{\bz,\bw,\eta}(z)$ is determined by $h|_{B_{\lambda_4 r}(z)}$ and the path $\eta$ stopped at the last time that it exits $B_{\lambda_4 r}(z)$.
\item \label{it:high_prob}
For all $z \in \mathbb{C}$ and all $r >0$, we have $\mathbb{P}(E_r(z)) \geq \mathbbm{p}$.
\item \label{it:inclusion_of_events}
For all points $z \in \mathbb{C},r >0$ and $\bz,\bw$ satisfying $\bz,\bw\notin B_{\lambda_4 r}(z)$, we have
\begin{align*}
\Lambda^{-1} \mathbb{P}(E_r(z) \cap \{\Gamma_{\bz,\bw} \cap B_{\lambda_2 r}(z) \neq \emptyset\} \lvert h|_{\mathbb{C} \setminus B_{\lambda_3r}(z)}) \leq \mathbb{P}(\mathfrak{C}_r^{\bz,\bw,\Gamma_{\bz,\bw}}(z) \cap \{\Gamma_{\bz,\bw} \cap B_{\lambda_2 r}(z) \neq \emptyset\} \lvert h|_{\mathbb{C} \setminus B_{\lambda_3r}(z)}).
\end{align*}
\end{enumerate}

Under the above hypotheses, for all $q>0,  \ell \in (0,1)$,  and for each bounded open set $U \subseteq \mathbb{C}$,  it holds with probability tending to $1$ as $\varepsilon \to 0$,  at a rate depending only on $U,q,\ell,\nu,  \{\lambda_i\}_{i=1,\cdots,5},  \Lambda$,  that for all $\bz,\bw \in (\varepsilon^q \mathbb{Z}^2) \cap U$ with $\bw\notin \cB_{\tau_{4\ell}(\bz)}^\bullet(\bz)$,  there exist $z \in \mathbb{C},  r \in [\varepsilon^{1+\nu} ,  \varepsilon]$ such that $\Gamma_{\bz,\bw} \cap B_{\lambda_2 r}(z) \neq \emptyset$ and $\mathfrak{C}_r^{\bz,\bw,\Gamma_{\bz,\bw}}(z)$ occurs. %
\end{proposition}
To obtain the above, we apply \cite[Theorem~4.2]{GM21} with $\ell$ replaced by $4\ell$. Also, we note that applying \cite[Theorem~4.2]{GM21} directly would yield the above statement with the condition $\bw\notin \cB^\bullet_{\tau_{4\ell}(\bz)}(\bz)$ replaced by the stronger condition $|\bw-\bz|>4\ell$; however, it can be checked that the proof therein only utilises the weaker condition $\bw\notin \cB^\bullet_{\tau_{4\ell}(\bz)}(\bz)$ (see in particular the proofs of Propositions 4.12 and 4.17 in \cite{GM21}). In our application, as indicated in Section \ref{sec:proof-outline}, $E_r(z)$ will be the event that there exist a $\scX$ of scale $r$ around $z$, while $\mathfrak{C}_r^{\bz,\bw,\eta}(z)$ will be the event that the path $\eta$ passes through such a $\scX$. %

For most of the paper, we will work with the event $\mathfrak{C}_r^{\bz,\bw,\Gamma_{\bz,\bw}}(z)$ for points $\bz,\bw$ admitting a unique $D_h$-geodesic $\Gamma_{\bz,\bw}$. In such a setting, we shall simply write $\mathfrak{C}_r^{\bz,\bw}(z)=\mathfrak{C}_r^{\bz,\bw,\Gamma_{\bz,\bw}}(z)$. In case there are multiple geodesics $\Gamma_{\bz,\bw}$, we shall use the more general notation, with a path appearing as superscript, to indicate which geodesic the event is being considered for.

In fact, apart from Proposition \ref{thm:theorem_uniqueness_lqg_metric} itself, we shall also need to introduce a certain regularity event introduced in its proof. This definition is taken from \cite[Section 4.3]{GM21} and will be stated in terms of certain parameters-- consider $a\in (0,1)$, $\nu>0$, open sets $U\subset V\subset \CC$ and $\ell>0$. Let the regularity event $\cE=\cE(a,\nu,U,V,\ell)$ depending on the above parameters $a,\nu,U,V,\ell$ and the event $E_r(z)$ (from Proposition \ref{thm:theorem_uniqueness_lqg_metric}) be as follows.
\begin{enumerate}
\item We have $\sup_{z,w\in U}D_h(z,w)\leq D_h(U,\partial V)$.
\item \label{it:ell2ell} For each $z\in B_{16\ell}(V)$, we have $B_a(z)\subseteq \cB^\bullet_{\tau_{\ell}(z)}(z)$ and
  \begin{align}
    \label{eq:4}
    &\min\{\tau_{2\ell}(z)-\tau_{\ell}(z), \tau_{3\ell}(z)-\tau_{2\ell}(z), \tau_{8\ell}(z)-\tau_{4\ell}(z), \tau_{12\ell}(z)-\tau_{8\ell}(z)\}\notag \\
    & \geq a \max\{e^{\xi h_1(z)},\ell^{\xi Q}e^{\xi h_\ell(z)}, (4\ell)^{\xi Q} e^{\xi h_{4\ell}(z)}\}.
  \end{align}
\item For each $z,w\in B_{16\ell}(V)$ with $|z-w|\leq a$, we have
  \begin{equation}
    \label{eq:5}
    D_h(z,w)\geq |z-w|^{\chi'}, D_h(z,w;B_{2|z-w|}(z))\leq |z-w|^\chi.    
  \end{equation}
\item \label{it:field_bounded} We have $\sup_{z\in V}|h_1(z)|\leq a^{-1}$.
  \item For each $\varepsilon\in (0,a]\cap \{2^{-n}\}_{n\in \NN}$ and each $z\in (\frac{\lambda_1\varepsilon^{1+\nu}}{4}\ZZ^2)\cap B_{16\ell}(V)$, there exists at least one $r\in [\varepsilon^{1+\nu},\varepsilon]$ for which $E_r(z)$ occurs.
    \item We have $\rho_{\varepsilon}(z)\leq \varepsilon^{1/2}$ for each $z\in \frac{\varepsilon^{1+\nu}}{4}\ZZ^2\cap B_{16\ell }(V)$ and each dyadic $\varepsilon \in (0,a]$.
    \end{enumerate}

    The above conditions are chosen such that the corresponding event $\cE$ from \cite[Section 4.3]{GM21} holds for both the parameters $\ell$ and $4\ell$. We now state a slightly different version of Proposition~\ref{thm:theorem_uniqueness_lqg_metric} where we condition on the regularity event $\mathcal{E}$ and demand that the output point $z$ is not too close or far from the point $\bz$-- this will be useful later.

\begin{proposition}\label{prop:chi_proposition}
  Suppose that we have the same setup as in Proposition~\ref{thm:theorem_uniqueness_lqg_metric}. Fix bounded open sets $U\subseteq V\subseteq \CC$ and $a\in (0,1), \ell > 0,\nu>0$. Then, for any fixed $q>0$, conditional on the event $\mathcal{E}$, the following holds.
  With probability tending to $1$ as $\varepsilon \to 0$ faster than any positive power of $\varepsilon$ (at a rate depending only on $a,\ell,\nu,q,U$ and $V$), for all $\bz,\bw \in (\varepsilon^q \mathbb{Z}^2) \cap U$ with $\bw\notin \cB_{\tau_{4\ell}(\bz)}^\bullet(\bz)$, there exist $(z,r) \in \mathbb{C} \times [\varepsilon^{1+\nu} ,  \varepsilon]$ such that 
\begin{align*}
B_{\lambda_4 r}(z) \subseteq \mathcal{B}_{\tau_{2 \ell}(\bz)}^{\bullet}(\bz) \setminus \mathcal{B}_{\tau_{\ell}(\bz)}^{\bullet}(\bz), \quad  \Gamma_{\bz ,  \bw} \cap B_{\lambda_2 r}(z) \neq \emptyset \quad \text{and} \quad \mathfrak{C}_r^{\bz ,  \bw}(z)\quad  \text{occurs}.
\end{align*}
\end{proposition}

\begin{proof}
  This directly follows from the proof of \cite[Theorem 4.2]{GM21}. Indeed, on the event $\cE$, all the candidate balls $B_{\lambda_4 r}(z)$ considered therein are contained in the region $\mathcal{B}_{\tau_{2 \ell}(\bz)}^{\bullet}(\bz) \setminus \mathcal{B}_{\tau_{\ell}(\bz)}^{\bullet}(\bz)$; this is precisely stated in \cite[Lemma 4.19]{GM21}.
\end{proof}

As formalised in the following lemma, as long as the event $E_r(z)$ has a sufficiently high probability, the probability of the regularity event $\cE$ can be made as high as desired.
    \begin{lemma}[{\cite[Lemma 4.11]{GM21}}]
      \label{lem:4}
    There is a constant $\mathbbm{p}$ depending only on $\nu$ and $\{\lambda_i\}_{i=1}^5$ such that under the hypotheses of Proposition \ref{thm:theorem_uniqueness_lqg_metric}, the following holds.  For each bounded open set $U\subseteq \CC,\ell > 0,  p\in (0,1)$, there exists a bounded open set $V\supseteq U$ and an $a\in (0,1)$ such that the event $\cE=\cE(a,\nu,U,V,\ell)$ satisfies $\PP(\cE)\geq p$.
  \end{lemma}

    \section{An upgraded version of Proposition \ref{prop:chi_proposition}}
    The goal now is to prove an upgraded version of Proposition~\ref{prop:chi_proposition} where $\bw$ is allowed to vary freely instead of being restricted to lie on the lattice $\varepsilon^q\ZZ^2$.  For the rest of the section,  we shall work in the same setup as Proposition~\ref{thm:theorem_uniqueness_lqg_metric}.

        \begin{proposition}
      \label{prop:5}
   Fix bounded open sets $U\subseteq V\subseteq \CC$, $a\in (0,1), \ell>0$ and $\nu>0$. %
   Then, for any fixed $q>0$, conditional on the event $\cE$, the following holds. With probability tending to $1$ as $\varepsilon\rightarrow 0$ faster than any positive power of $\varepsilon$ (at a rate depending only on $a,\ell,\nu,q,U$ and $V$), for all $\bz\in (\varepsilon^q\ZZ^2)\cap U$ and all $w\in U$ satisfying $w\notin \cB_{\tau_{16\ell}(\bz)}^\bullet(\bz)$, and for every $D_h$-geodesic $\Gamma_{\bz,w}$ from $\bz$ to $w$, there exists $(z,r)\in \CC\times [\varepsilon^{1+\nu},\varepsilon]$ such that 
\begin{align*}
B_{\lambda_4 r}(z) \subseteq \mathcal{B}_{\tau_{2\ell}(\bz)}^{\bullet}(\bz) \setminus \mathcal{B}_{\tau_{\ell}(\bz)}^{\bullet}(\bz), \quad  \Gamma_{\bz ,  w} \cap B_{\lambda_2 r}(z) \neq \emptyset \quad \text{and} \quad \mathfrak{C}_r^{\bz ,  w,\Gamma_{\bz,w}}(z)\quad  \text{occurs}.
\end{align*}
    \end{proposition}

In the remainder of this section, we shall discuss the proof of the above proposition.   For this, we shall need to use certain ideas from the proof technique used in \cite{GM21}-- it will be convenient to use a setup very similar to \cite[Section 4.1]{GM21}, and we now introduce this. %

\label{sec:reg-event}

    Let $\beta\in (0,1)$ be a parameter to be specified later. Also, for now we fix a choice of the parameters $U,V,\nu,\ell,a$ from the definition of $\cE$; later, we shall choose them appropriately using Lemma \ref{lem:4}.  We define $K:=\lfloor a \varepsilon^{-\beta} \rfloor - 1$,  where $a$ is the constant in the definition of the regularity event $\mathcal{E}$. Fix $\bz\in U$ and, for $k\in \NN$, consider the following set of times which constitute a gradual metric exploration of the LQG surface starting from the point $\bz$.
    \begin{align}
      \label{eq:6}
      s_k&=\tau_{4\ell}(\bz)+ k\varepsilon^\beta e^{\xi h_1(\bz)},\nonumber\\
      t_k&=s_k+\varepsilon^{2\beta}e^{\xi h_1(\bz)}.
    \end{align}
    \begin{figure}
      \centering
      \includegraphics[width=0.7\linewidth]{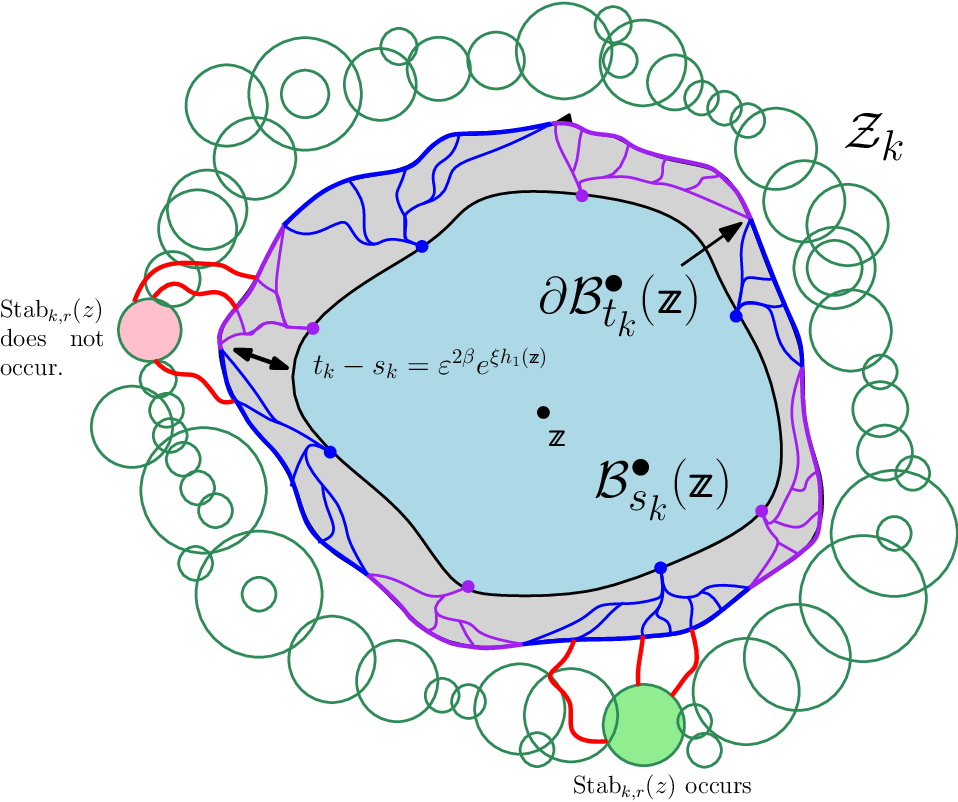}
      \caption{Illustration of the setup of the proof of Proposition~\ref{prop:5}; this figure is taken directly from \cite[Figure 3]{GM21}. The two filled LQG metric balls $\mathcal{B}_{s_k}^{\bullet}(\bz) \subseteq \mathcal{B}_{t_k}^{\bullet}(\bz)$ centered at $\bz$ are shown,  together with the set of points $\text{Conf}_k \subseteq \partial \mathcal{B}_{s_k}^{\bullet}(\bz)$ hit by leftmost $D_h$-geodesics from $\bz$ to $\partial \mathcal{B}_{t_k}^{\bullet}(\bz)$ (blue and purple) and the sets of arcs $\mathcal{I}_k$ of $\partial \mathcal{B}_{s_k}^{\bullet}(\bz)$ consisting of points whose leftmost $D_h$-geodesics hit the same point of $\text{Conf}_k$. Several balls $B_{\lambda_3 r}(z)$ for $(z,r) \in \mathcal{Z}_k$ are shown in green. One such ball for which $\text{Stab}_{k,r}(z)$ occurs is also depicted, that is, all the geodesics (red) from $\bz$ to points of $\partial B_{\lambda_3 r}(z)$ hit the same arc of $\mathcal{I}_k$. Finally, on the left, there is a ball (pink) for which $\text{Stab}_{k,r}(z)$ does not occur.}
    \end{figure}
Compared with the setup from \cite[Section 4]{GM21}, we have replaced $\ell$ by $4\ell$ as this will be convenient later. Note that $t_k\in [s_k,s_{k+1}]$.  Also, Proposition~\ref{prop:6} implies that it is a.s.\  the case that for each $k \in \mathbb{N}$ there are only finitely many points of $\partial \mathcal{B}_{s_k}^{\bullet}(\bz)$ which are hit by leftmost $D_h$-geodesics from $\bz$ to $\partial \mathcal{B}_{t_k}^{\bullet}(\bz)$.  Let $\text{Conf}_k \subseteq \partial \mathcal{B}_{s_k}^{\bullet}(\bz)$ be the set of such points and let $\mathcal{I}_k$ be the family of subsets of $\partial \mathcal{B}_{t_k}^{\bullet}(\bz)$ of the form
\begin{align*}
\{y \in \partial \mathcal{B}_{t_k}^{\bullet}(\bz)\,\,: \,\text{the leftmost} \,\,\,D_h\text{-geodesic} \,\,\,\text{from} \,\,\,  \bz \,\,\,\text{to} \,\,\,  y \,\,\,\,  \text{passes through} \,\,\,x\} \quad \text{for} \quad x \in \text{Conf}_k.
\end{align*}

Recall that Proposition~\ref{prop:6} implies that the sets in $\mathcal{I}_k$ are connected arcs of the Jordan curve $\partial \mathcal{B}_{t_k}^{\bullet}(\bz)$. We now wish to consider a family of small balls around the boundary $\cB_{t_k}^\bullet(\bz)$. Using $\dist(\cdot,\cdot)$ to denote the Euclidean metric, let $\cZ_k$ for $k \in \NN$ be the set of pairs $(z,r)$ such that %
    \begin{equation}
      \label{eq:7}
      z\in (\lambda_1\varepsilon^{1+\nu}/4)\ZZ^2\setminus \cB_{t_k}^\bullet(\bz), r\in [\varepsilon^{1+\nu},\varepsilon], \dist(z,\partial \cB_{t_k}^\bullet(\bz))\in [\lambda_4\varepsilon,2\lambda_4\varepsilon].
    \end{equation}
From now on,  we will work with events $E_r(z),\mathfrak{C}_r^{\bz,\bw,\eta}(z)$ defined so as to satisfy the hypothesis of Proposition \ref{thm:theorem_uniqueness_lqg_metric}.%

    The following result is an easy consequence of the definitions and the choices of the constants. Recall from the notational comments in the introduction that for $x<y\in \RR$, we shall use $[\![x,y]\!]$ to denote $[x,y]\cap \ZZ$.
        \begin{proposition}
      \label{prop:4}
      On the regularity event $\cE$, there exists a deterministic $\varepsilon_0>0$ such that for all $\varepsilon\leq \varepsilon_0$ and any $k\in [\![1,K]\!]$, simultaneously for all points $w\notin \mathcal{B}^\bullet_{\tau_{12\ell}(\bz)}(\bz)$ and every $D_h$-geodesic $P$ from $\bz$ to $w$, there must exist a pair $(z,r)\in \cZ_k$ such that $P\cap B_{\lambda_2 r}(z) \neq \emptyset$.
    \end{proposition}
    \begin{proof}
      First, by condition \eqref{it:ell2ell} from the definition of $\cE$, it follows that for all $k\in [\![1,K]\!]$, we have $t_k< s_{k+1}\leq \tau_{8\ell}(\bz)$; indeed, this is because $\tau_{8\ell}(z)-\tau_{4\ell}(z)\geq ae^{\xi h_1(\bz)}$ and as $k+1\leq K+1=\lfloor a\varepsilon^{-\beta}\rfloor$. Now, note that for adjacent points in the lattice $\lambda_1\varepsilon^{1+\nu}\ZZ^2/4$, the balls of radius $\lambda_1 \varepsilon^{1+\nu}$ ``overlap'' with each other. %
      Thus, since for $(z,r) \in \mathcal{Z}_k$, we have $\lambda_2 \varepsilon \geq \lambda_2 r \geq \lambda_1 \varepsilon^{1+\nu}$, the union $\bigcup_{(z,r) \in \mathcal{Z}_k} B_{\lambda_2 r}(z)$ disconnects $\partial \mathcal{B}_{t_k}^{\bullet}(\bz)$ from $\infty$. %
  
Moreover, \eqref{eq:7} implies that 
      \begin{align*}
          B_{\lambda_2 r}(z) \subseteq B_{3\lambda_4 r}(\mathcal{B}_{t_k}^{\bullet}(\bz)), \quad \text{for all} \quad (z,r) \in \mathcal{Z}_k,
      \end{align*}
       and this when combined with \eqref{eq:5} and H\"older continuity implies that if we take $\varepsilon_0$ to be small enough, then
      \begin{align*}
          B_{\lambda_2 r}(z) \subseteq \mathcal{B}_{t_k + (3\lambda_4 \varepsilon)^{\chi}}^{\bullet}(\bz) \quad \text{for all} \quad (z,r) \in \mathcal{Z}_k
      \end{align*}
for all $\varepsilon\leq \varepsilon_0$, where we note that $\cZ_k$ implicitly depends on $\varepsilon$. By combining items \eqref{it:field_bounded}, \eqref{eq:4} in the definition of $\cE$, we have
      \begin{align*}
          \tau_{12\ell}(\bz) - \tau_{8\ell}(\bz) \geq a e^{\xi h_1(\bz)} \geq a e^{-\xi a^{-1}}.
      \end{align*}
      As discussed in the first paragraph of the proof, we also have $t_k \leq \tau_{8\ell}(\bz)$. As a result, by possibly taking $\varepsilon_0 \in (0,1)$ to be smaller, we can ensure that for all $\varepsilon\leq \varepsilon_0$,
      \begin{align*}
          B_{\lambda_2 r}(z) \subseteq \mathcal{B}_{\tau_{12\ell}(\bz)}^{\bullet}(\bz) \quad \text{for all} \quad (z,r) \in \mathcal{Z}_k.
      \end{align*}
      Therefore, the union 
      \begin{align*}
          \bigcup_{(z,r) \in \mathcal{Z}_k} B_{\lambda_2 r}(z) \subseteq \mathcal{B}_{\tau_{12\ell}(\bz)}^{\bullet}(\bz)
      \end{align*}
      disconnects $\partial \mathcal{B}_{t_k}^{\bullet}(\bz)$ from $\infty$, and thus any $D_h$-geodesic from $\bz$ to some point $w \notin \mathcal{B}^{\bullet}_{\tau_{12\ell}(\bz)}(\bz)$ must intersect $B_{\lambda_2 r}(z)$ for some $(z,r) \in \mathcal{Z}_k$. This completes the proof of the proposition.
    \end{proof}
    
    Now, for $(z,r)\in \cZ_k$, we say that $\stab_{k,r}(z)$ occurs if all $D_h(\cdot,\cdot; \CC\setminus \overline{B_{\lambda_3r}(z)})$-geodesics $P$ from $\bz$ to points of $\partial B_{\lambda_3 r}(z)$ hit $\partial \cB_{t_k}^\bullet(\bz)$ in the same arc of $\cI_k$. Note that since $\lambda_4 \geq  \lambda_3$, for every $(z,r)\in \cZ_k$, we must (see \eqref{eq:7}) have $B_{\lambda_3 r}(z)\subseteq \CC\setminus \cB_{t_k}^\bullet(\bz)$. Thus, for $(z,r)\in \cZ_k$, any $D_h$-geodesic $P$ from $\bz$ to a point in $B_{\lambda_3r}(z)$ must pass through $\partial B_{\lambda_3 r}(z)$ and thus there must be a smallest time $t_P$ satisfying $P(t_P)\in \partial B_{\lambda_3 r}(z)$ and $P\lvert_{[0,t_P)}\subseteq \CC\setminus\overline{B_{\lambda_3r}(z)}$. Note that since $P$ is a $D_h$-geodesic, $P\lvert_{[0,t_P)}$ is necessarily a $D_h(\cdot,\cdot; \CC\setminus \overline{B_{\lambda_3r}(z)})$-geodesic. As a result of the above discussion, on the event $\stab_{k,r}(z)$, all geodesics $P$ from $\bz$ to points in $B_{\lambda_3 r}(z)$ must pass via the same arc of $\cI_k$.

By combining the above with Proposition \ref{prop:7}, we have the following.
    \begin{lemma}
      \label{lem:6}
      By translation invariance, we know that $h(\bz+\cdot)-h_1(\bz)\stackrel{d}{=}h$. Now, consider the filtration
      For $(z,r) \in \cZ_k$, on the event $\stab_{k,r}(z)$, for any geodesics $P_1,P_2$ from $\bz$ to some points of $B_{\lambda_3 r}(z)$, we must necessarily have $P_1\lvert_{[0,s_k]}=P_2\lvert_{[0,s_k]}$.  
    \end{lemma}
    
\begin{proof}
  Since Proposition \ref{prop:7} is stated with $0$ instead of $\bz$, we shift our origin by considering the field $h_{\bz}(\cdot)=h(\bz+\cdot)-h_1(\bz)$ which we know satisfies $h_{\bz}\stackrel{d}{=} h$. Note that $P\colon [0,T]\rightarrow \CC$ is a $D_h$-geodesic if and only if $P_\bz\colon[0,e^{-\xi h_1(\bz)} T]\rightarrow \CC$ defined by $P_\bz(x)= P(e^{\xi h_1(\bz)} x)-\bz$ is a $D_{h_{\bz}}$-geodesic. Now, it can be checked that for all $k\in \NN$, $e^{-\xi h_1(\bz)} s_k$ and $ e^{-\xi h_1(\bz)}t_k$ are stopping times with respect to the filtration $\cG$ defined by
  \begin{equation}
    \label{eq:106}
    \cG_s=\sigma(\mathcal{B}_s^{\bullet}(0;D_{h_\bz}) , h_\bz|_{\mathcal{B}_s^{\bullet}(0;D_{h_\bz})}).
  \end{equation}
 The statement of the lemma now follows by combining item \eqref{it:uniquegeod} in Proposition~\ref{prop:6} with Proposition~\ref{prop:7}.
\end{proof}

    The following result, implicitly shown in Section 4.4 in \cite{GM21}, is the reason behind the utility of the $\stab_{k,r}(z)$ events discussed above and also fixes the choice of $\beta\in (0,1)$ used in the exploration \eqref{eq:6}.
    \begin{proposition}
      \label{prop:3}
      There exist constants $\beta,\theta\in (0,1)$ depending only on the LQG parameter $\gamma\in (0,2)$ such that for any choice of the parameters $U,V,a,\nu,\ell$, conditional on the regularity event $\cE$, the following holds. With probability tending to $1$ as $\varepsilon \to 0$ faster than any positive power of $\varepsilon$, there are at least $(1-\varepsilon^\theta)K$ many values of $k\in [\![1,K]\!]$ for which the following holds-- simultaneously for all points %
      $w\notin \cB_{\tau_{16\ell}(\bz)}^\bullet(\bz)$ and every $D_h$-geodesic $\Gamma_{\bz,w}$ from $\bz$ to $w$, $\stab_{k,r}(z)$ occurs for every $(z,r)\in \cZ_k$ for which $\Gamma_{\bz,w}\cap B_{\lambda_2 r}(z)\neq \emptyset$.
    \end{proposition}
    
\begin{proof}
  This is implicitly shown in the proof of \cite[Proposition 4.12]{GM21}. Indeed, while the proof from \cite{GM21} would yield the statement for a \emph{fixed} $w\notin B_{16\ell}(\bz)$, the only relevant property of $w$ used to construct the ``barrier-events'' therein is that $w\notin \cB_{\tau_{16\ell}(\bz)}^\bullet(\bz)$, and the desired result holds for all such $w$.
\end{proof}

We are now ready to complete the proof of Proposition \ref{prop:5}. To do so, we shall undertake a two step procedure where first, we use Proposition \ref{prop:3} to argue that simultaneously for all $w\in U$ with $w\notin \cB_{\tau_{16\ell}(\bz)}^\bullet(\bz)$, the geodesic $\Gamma_{\bz,w}$ must come close to a lattice point $z_w$ for which $\stab_{k,r}(z_w)$ occurs for appropriate $k,r$. Having done so, we will invoke Proposition \ref{prop:chi_proposition} with $\bw$ replaced by the above lattice point $z_w$.

    \begin{proof}[Proof of Proposition \ref{prop:5}]

      First by applying Proposition~\ref{prop:chi_proposition} and a union bound,  we obtain that there is an event $E_{\varepsilon}^1$ such that $\PP((E_{\varepsilon}^1)^c \cap \cE)$ converges to $0$ as $\varepsilon \to 0$ faster than any positive power of $\varepsilon$ and such that the following holds on $E_{\varepsilon}^1 \cap \mathcal{E}$.  For all $\bw \in (\frac{\lambda_1 \varepsilon^{1+\nu}}{4}) \mathbb{Z}^2 \cap V$ additionally satisfying $\bw\notin \cB_{\tau_{4\ell}(\bz)}^\bullet(\bz)$ and for each geodesic $\Gamma_{\bz,\bw}$, there exists $(z,r) \in \mathbb{C} \times [\varepsilon^{1+\nu} ,  \varepsilon]$ such that 
\begin{align*}
B_{\lambda_4 r}(z) \subseteq \mathcal{B}_{\tau_{2 \ell}(\bz)}^{\bullet}(\bz) \setminus \mathcal{B}_{\tau_{\ell}(\bz)}^{\bullet}(\bz), \quad  \Gamma_{\bz ,  \bw} \cap B_{\lambda_2 r}(z) \neq \emptyset \quad \text{and} \quad \mathfrak{C}_r^{\bz ,  \bw , \Gamma_{\bz,\bw}}(z)\quad  \text{occurs}.
\end{align*}

   By combining Propositions~\ref{prop:4} and ~\ref{prop:3},  we obtain that there is an event $E_{\varepsilon}^2$ with the property that $\PP((E_{\varepsilon}^2)^c \cap \mathcal{E})$ converges to $0$ as $\varepsilon \to 0$ faster than any positive power of $\varepsilon$ such that on $E_{\varepsilon}^2 \cap \mathcal{E}$, the following holds. For any point $w \in U$ satisfying $w\notin \cB_{\tau_{16\ell}(\bz)}^\bullet(\bz)$, and any geodesic $\Gamma_{\bz,w}$, there must exist at least one $k_w \in [\![K(1-\varepsilon^{\theta}) ,  K]\!]$ for which there exists $(z_w,r_w) \in \mathcal{Z}_{k_w}$ additionally satisfying $\Gamma_{\bz ,  w} \cap B_{\lambda_2r_w}(z_w) \neq \emptyset$,  so that $\stab_{k_w,r_w}(z_w)$ also occurs. Note that (see \eqref{eq:6}) since $s_{k_w} \geq \tau_{4\ell}(\bz)$, this $z_w$ must satisfy $z_w\notin \cB_{\tau_{4\ell}(\bz)}^\bullet(\bz)$.  %

By combining the above two paragraphs, on the event 
\begin{align*}
\mathcal{E} \cap E_{\varepsilon}^1 \cap E_{\varepsilon}^2,
\end{align*}
we can first find for any $w \in U\setminus \cB_{\tau_{16\ell}(\bz)}^\bullet(\bz)$ ,  a corresponding point $z_w \notin B_{\tau_{4\ell}(\bz)}(\bz)$ as described above. Further, since $z_w\in (\lambda_1\varepsilon^{1+\nu}/4)\ZZ^2 \cap V$, by the first paragraph of the proof,  there exists $(z,r) \in \mathbb{C} \times [\varepsilon^{1+\nu} ,  \varepsilon]$ such that
\begin{align*}     
B_{\lambda_4 r}(z) \subseteq \mathcal{B}_{\tau_{2 \ell}(\bz)}^{\bullet}(\bz) \setminus \mathcal{B}_{\tau_{\ell}(\bz)}^{\bullet}(\bz)     
\end{align*}
and for which both
\begin{align*}
\{\Gamma_{\bz ,  z_w} \cap B_{\lambda_2 r}(z) \neq \emptyset\} \quad \text{and} \quad \mathfrak{C}_r^{\bz ,  z_w , \Gamma_{\bz,z_w}}(z)
\end{align*}     
occur.   Finally,  we note by Lemma~\ref{lem:6} that
\begin{align*}
\Gamma_{\bz,w}|_{[0,s_{k_w}]} = \Gamma_{\bz,z_w}|_{[0,s_{k_w}]}.
\end{align*}
Since $s_{k_w} \geq \tau_{4\ell}(\bz)$ and both $\Gamma_{\bz,w} , \Gamma_{\bz , z_w}$ are $D_h$-geodesics and $B_{\lambda_4 r}(z) \subseteq \mathcal{B}^{\bullet}_{\tau_{2\ell}(\bz)}(\bz)$, we obtain that the part of $\Gamma_{\bz,w}$ (resp.\ $\Gamma_{\bz,z_w}$) stopped at the last time that it exits $B_{\lambda_4 r}(z)$ is contained in $\Gamma_{\bz,w}|_{[0,s_{k_w}]}$ (resp.\ $\Gamma_{\bz,z_w}|_{[0,s_{k_w}]}$). Moreover, the definition of $\mathfrak{C}_r^{\bz,w,\Gamma_{\bz,w}}(z)$ (resp.\ $\mathfrak{C}_r^{\bz,z_w,\Gamma_{\bz,z_w}}(z)$) implies that it is a.s.\ determined by $h|_{B_{\lambda_4 r}(z)}$ and the part of $\Gamma_{\bz,w}$ (resp.\ $\Gamma_{\bz,z_w}$) stopped at the last time that it exits $B_{\lambda_4 r}(z)$. It follows that
\begin{align*}
\{\Gamma_{\bz,w} \cap B_{\lambda_2 r}(z) \neq \emptyset\} \cap \mathfrak{C}_r^{\bz,w,\Gamma_{\bz,w}}(z) = \{\Gamma_{\bz,z_w} \cap B_{\lambda_2 r}(z) \neq \emptyset\} \cap \mathfrak{C}_r^{\bz,z_w,\Gamma_{\bz,z_w}}(z).
\end{align*}     
 By using the discussion in the first two paragraphs of the proof, we obtain that $\PP((E_{\varepsilon}^1 \cap E_{\varepsilon}^2)^c \cap \mathcal{E})$
converges to $0$ as $\varepsilon \to 0$ faster than any positive power of $\varepsilon$.  This completes the proof of the proposition.     
\end{proof}

Finally, we now state an immediate corollary of Proposition \ref{prop:5}, and this the only result from this section that will be used later in the paper. We emphasize that the family of events $\mathfrak{C}_r^{\bz,\bw,\eta}(z)$ is defined such that the hypothesis in Proposition \ref{thm:theorem_uniqueness_lqg_metric} is satisfied.

    \begin{corollary}\label{cor:main_corollary}
Fix a bounded open set $U \subseteq \mathbb{C}$, $\ell>0$ and $\nu>0$. Set $\varepsilon_n = 2^{-n}$ for all $n \in \mathbb{N}$. Fix also $q>0$. Then almost surely, the following holds for all $n$ large enough. For all $\bz \in (\varepsilon_n^{q} \mathbb{Z}^2) \cap U$ and all $w \in U$ satisfying $w\notin \cB_{\tau_{16\ell}(\bz)}^\bullet(\bz)$, and every geodesic $\Gamma_{\bz,w}$, there exists $(z,r) \in \mathbb{C} \times [\varepsilon_n^{1+\nu} , \varepsilon_n]$ such that 
\begin{align}
  \label{eq:hitt}
B_{\lambda_4 r}(z) \subseteq \mathcal{B}_{\tau_{2 \ell}(\bz)}^{\bullet}(\bz) \setminus \mathcal{B}_{\tau_{\ell}(\bz)}^{\bullet}(\bz), \quad  \Gamma_{\bz ,  w} \cap B_{\lambda_2 r}(z) \neq \emptyset \quad \text{and} \quad \mathfrak{C}_r^{\bz ,  w,\Gamma_{\bz,w}}(z)\quad  \text{occurs}.
\end{align}
\end{corollary}
\begin{proof}
  First obtain by using Lemma \ref{lem:4} that by properly adjusting the parameters $a$ and $V$, $\PP(\cE)$ can be made arbitrarily high. Now, we use Proposition \ref{prop:5} along with the Borel-Cantelli lemma.
\end{proof}

In the next section, we will take up the task of carefully defining the events $E_r(z)$ and $\mathfrak{C}_r^{\bz,\bw,\eta}(z)$ satisfying the desired conditions. At the end, in Section \ref{sec:geodesic_goes_through_chi}, we will apply Corollary \ref{cor:main_corollary} with the above-mentioned choice of $\mathfrak{C}_r^{\bz,\bw,\eta}(z)$.

\section{Definition of $\scX$ and showing that the geodesic goes through it in many different places}\label{sec:construction_of_chi}

\subsection{Definitions and setup.}\label{sec:setup}

As defined in Section \ref{sec:gaussian-free-field}, we shall always work with a fixed $\gamma \in (0,2)$ and a whole plane GFF $h$ normalised to have $h_1(0)=0$. The goal of this section is to show that, with strictly positive probability uniformly in $z\in \CC$ and $r>0$, there is a ``locally dependent'' event $F_r(z)$, on which there exists a $\scX$ of Euclidean scale $r$ around the point $z$ as described in Section~\ref{sec:proof-outline}. For later use, we shall require that this $\scX$ satisfies a few properties depending on various parameters-- for instance, we shall need that the $\scX$ is a subset of a deterministic family of small squares around $z$ and that the four distinct ``arms'' of the $\scX$ do not venture close to each other. In this section, we shall formalise the above in Proposition~\ref{prop:definition_of_F_r(z)}, which is where the desired event $F_r(z)$ shall be finally defined. As a matter of notation, for a family of events $\{H_r(z)\}_{z\in \CC,r>0}$ measurable with respect to $\sigma(h)$, we say that $H_r(z)$ is translation and scale invariant if the occurrence of $H_r(z)$ for the field $h$ is the same as the occurrence of the event $H_1(0)$ for the field $z'\mapsto h(z+rz')-h_r(z)$. The events $F_r(z)$ that we shall define in this section will be \textbf{translation and scale invariant} in the above sense.

\textbf{There shall be many different parameters and sets in play throughout this section and we refer the reader to Table \ref{table} for a summary of their roles.}

At its core, the $\scX$ is produced via the following confluence result from \cite{GPS20} that we alluded to just after Theorem \ref{thm:2}.
\begin{proposition}
  \label{prop:1*} (\cite[Theorem~1.2]{GPS20})
  Fix $\bz\in \CC$. Almost surely, for every neighbourhood $U$ of $\bz$, there is a neighbourhood $U'\subseteq U$ of $\bz$ and a point $z\in U\setminus U'$ such that every $D_h$-geodesic from a point in $U'$ to a point in $\CC\setminus U$ passes through $z$.
\end{proposition}

In order to proceed, we shall now need a precise definition of a $\scX$, which we now provide.  %
Note that for any $z,w\in \CC$, all geodesics $\Gamma_{z,w}$ are necessarily simple curves and we shall consider the left and right sides of $\Gamma_{z,w}$ defined as collections of prime ends (recall the notational comments from the introduction). %

\begin{definition}
  \label{def:1*}
  For distinct points $u_+,v_+,u_-,v_-\in \CC$, we say that $\scX^{u_+,v_+}_{u_-,v_-}$ occurs if the following hold.
  \begin{enumerate}
  \item \label{it:uniquegeod1} The $D_h$-geodesics $P_+,P_-$ from $u_+$ to $v_+$ and $u_-$ to $v_-$ are unique. Further, $P_+\cap P_-$ is non-empty and is a non-trivial simple path, whose (distinct) starting and ending points we call $u$ and $v$ respectively.
  \item \label{it:concatenation} There is a unique $D_h$-geodesic $\Gamma_{u_+,v_-}$ (resp.\ $\Gamma_{u_-,v_+}$) and this is equal to the concatenation of $\Gamma_{u_+,u}, \Gamma_{u,v}, \Gamma_{v,v_-}$ (resp.\ $\Gamma_{u_-,u}, \Gamma_{u,v}, \Gamma_{v,v_+}$).
  \item \label{it:sides} The geodesics $\Gamma_{u_-,u}, \Gamma_{v,v_-}$ lie to the right of $P_+$ and the geodesics $\Gamma_{u_+,u},\Gamma_{v,v_+}$ lie to the left of $P_-$.
  \end{enumerate}
\end{definition}

In order to make the notation simpler, for the reminder of the paper, we shall linearly reparametrise $P_+$ and $P_-$ to be maps from $[0,1]$ to $\CC$\footnote{Note that this is an exception to our usual convention of parametrising geodesics to cover unit $D_h$-length in unit time.}. That is, we shall have $P_+(0)=u_+,P_+(1)=v_+$ and $P_-(0)=u_-,P_-(1)=v_-$. %
We now have the following immediate consequence of Proposition \ref{prop:1*}, and we refer the reader to Figure \ref{fig:setup1} for a depiction of the setting.

\begin{figure}
  \centering
  \includegraphics[width=\linewidth]{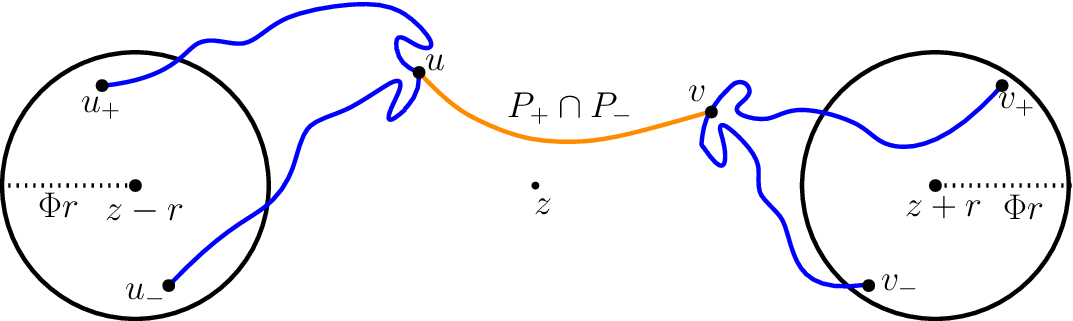}
  \caption{The setup in Lemma \ref{lem:1*}: $b$ is chosen to be small enough (independently of $r$) so that the Euclidean distance between $\{u_+ ,  u_-\}$ (resp.\ $\{v_+ ,  v_-\}$) and the union of the orange curve with the two blue curves intersecting the ball of radius $\Phi r$ centered at $z+r$ (resp.\  $z-r$) is at least $2br$.  Moreover the Euclidean distance between the union of the two blue curves intersecting the ball centered at $z-r$ and the union of the two blue curves intersecting the ball centered at $z+r$ is at least $2br$.}
  \label{fig:setup1}
\end{figure}

\begin{lemma}
  \label{lem:1*}
   Fix $q\in (0,1/5)$ and $\Phi\in (0,1/10)$. %
   Then for all $b \in (0,1)$ sufficiently small (depending only on $q,\Phi$), the following is true. For all $z\in \CC,r>0$, there exist deterministic (depending only on $z,r,q,\Phi,b$) points $u_+,u_-,v_+,v_-$ satisfying $u_+,u_-\in B_{\Phi r}(z-r)$ and $v_+,v_-\in B_{\Phi  r}(z+r)$ such that the balls of Euclidean radius $2br$ around $u_+,u_-,v_+,v_-,u,v$ are all pairwise disjoint and the intersection of the following events, which we call $G_r(z)$, occurs with probability at least $q$.
   \begin{enumerate}
   \item  $\scX_{u_-,v_-}^{u_+,v_+}$
   \item \label{it:Phi} $\Gamma_{u_+,u}\cup \Gamma_{u_-,u}\subseteq B_{\Phi r}(z-r)$ and $\Gamma_{v,v_+}\cup \Gamma_{v,v_-}\subseteq B_{\Phi r}(z+r)$.
   \item \label{it:distuv} $\dist(\Gamma_{u,v_+}\cup \Gamma_{u,v_-}, \{u_+,u_-\})\geq 2br$ and $\dist(\Gamma_{u_+,v}\cup \Gamma_{u_-,v}, \{v_+,v_-\})\geq 2br$

   \item $\dist(\Gamma_{u_-,u}\cup \Gamma_{u_+,u}, \Gamma_{v,v_-}\cup \Gamma_{v,v_+})\geq 2br$
   \end{enumerate}
   Further, the points $u_+,u_-,v_+,v_-$ can be chosen such that the family of events $\{G_r(z)\}_{z\in \CC, r>0}$ is translation and scale invariant.
\end{lemma}

\begin{proof}
By the translational and scaling symmetry of the whole plane GFF (see Section~\ref{sec:gaussian-free-field}), it suffices to work with $z=0,r=1$ and just define $G_1(0)$. Thus, for $(z,r)\neq (0,1)$, we can define the corresponding $u_+,u_-,v_+,v_-$ be translating and scaling the corresponding points for $(z,r)=(0,1)$, and as a result, the translation and scale invariance of $\{G_r(z)\}_{z\in \CC,r>0}$ shall hold automatically.

  Now, we consider the points $-1,1\in \CC$ and invoke Proposition \ref{prop:1*} twice with $\bz$ taking the above values. As a result, we obtain neighbourhoods $U'_{-1}\subseteq B_{\Phi }(-1),U'_{1}\subseteq B_{\Phi}(1)$ of $-1,1$ along with points $z_{-1}\in B_{\Phi}(-1)\setminus U_{-1}', z_{1}\in B_{\Phi}(1)\setminus U_{1}'$ such that all geodesics from points in $U_{-1}'$ to $U_{1}'$ pass via both $z_{-1}$ and $z_{1}$. As a consequence, condition \eqref{it:Phi} would be satisfied as long as we ensure that $u_+,u_-\in U_{-1}'$ and $v_+,v_-\in U_{1}'$.

  To do so, note that by choosing $b'$ sufficiently small, we can assure that with probability at least $5q$, we have $B_{b'}(-1)\subseteq U_{-1}', B_{b'}(1)\subseteq U_{1}'$. Thus, we can now just fix $u_+,u_-\in B_{b'}(-1)$ and $v_+,v_-\in B_{b'}(1)$ deterministically and note that for $b$ small enough with respect to $b'$, we have $|u_+-u_-|,|v_+-v_-|\geq 2b$ and thus the condition \eqref{it:Phi} holds as well.

Moreover since all the intersections $\{u_+,u_-\} \cap (\Gamma_{u,v_+} \cup \Gamma_{u,v_-}) $, $\{v_+,v_-\} \cap (\Gamma_{u_+,v} \cup \Gamma_{u_-,v}) $, $(\Gamma_{u_-,u}\cup \Gamma_{u_+,u})\cap (\Gamma_{v,v_+}\cup \Gamma_{v,v_-})$ are almost surely empty and $D_h$-geodesics between fixed and distinct deterministic points are almost surely unique, possibly by taking $b \in (0,1)$ to be even smaller, the event in the lemma statement, except for \eqref{it:sides} in the definition of $\scX_{u_-,v_-}^{u_+,v_+}$ in Definition \ref{def:1*}, occurs with probability at least $4q$. Finally, we note that by possibly swapping $u_+$ with $u_-$ and $v_+$ with $v_-$, we can ensure that condition \eqref{it:sides} in the definition of $\scX_{u_-,v_-}^{u_+,v_+}$ holds as well. Thus, by the pigeonhole principle, the event in the lemma occurs with probability at least $q$, and this completes the proof.
\end{proof}

\textbf{We now just fix $\Phi\in (0,1/10)$ that we shall use when applying the above lemma; this choice will not be altered throughout the paper.}

Since the geodesics $P_+,P_-$ are unique on the event $\scX_{u_-,v_-}^{u_+,v_+}$, we now define $\tau_u^+$ to be the a.s.\ unique time for which $P_+(\tau_u^+)=u$; we define $\tau_u^-,\tau_v^+,\tau_v^-$ analogously. Later, it will be convenient for us to have that the geodesics $P_+,P_-$ discussed above all lie inside a grid of squares of lattice spacing $\delta_1 r/2$ for a small $\delta_1>0$. For this purpose, for any set $X\subseteq \CC$ and $z\in \CC,\varepsilon>0$, we use $\cS^z_\varepsilon(X)$ %
to denote the set of closed Euclidean squares of size length $\varepsilon$ with corners in $z+\varepsilon \ZZ^2$ which additionally intersect $X$. %

As we shall see soon, for a small constant $\delta_1$, we shall cover the points $u,v$ by squares of size $\delta_1 r$. Further, for a constant $\delta_3$ much smaller than $\delta_1$ above, we cover the geodesics $P_+,P_-$ by squares of size $\delta_3r$. Indeed, %
we shall construct five families of squares of size $\delta_3 r$ covering each of the four arms of the $\scX$ and the central portion connecting $u$ and $v$. We shall also have another intermediate parameter $\delta_2$-- intuitively, the above parameters shall be chosen such that ``$b\gg \delta_1\gg \delta_2\gg \delta_3$''.

For us, it will be important that the squares described above be deterministic. However, the points $u,v$ are themselves random, and thus, in order to achieve this, we shall work on a positive probability event on which the above families of squares can just be fixed deterministically, and the goal now is to formalise this. Recall the constants $\chi\in (0,\xi(Q-2))$ and $\chi' > \xi (Q + 2)$ depending only on $\gamma$ that we fixed earlier (see Lemma \ref{lem:holder_regularity}). Also, recall the $\gamma$-dependent constants $Q = \frac{2}{\gamma}+\frac{\gamma}{2}$, $d_{\gamma}$ and $\xi = \frac{\gamma}{d_{\gamma}}$ from Proposition \ref{prop:8}. The following proposition defines the positive probability event $F_r(z)$ that we shall use repeatedly throughout the paper; a visual depiction of the relevant geometric objects is present in Figure \ref{fig:setup2}. Note that for a set $V\subseteq \CC$, we use $\inte(V)$ to denote its topological interior $V \setminus \partial V$.

\begin{figure}
  \centering
  \includegraphics[width=0.7\linewidth]{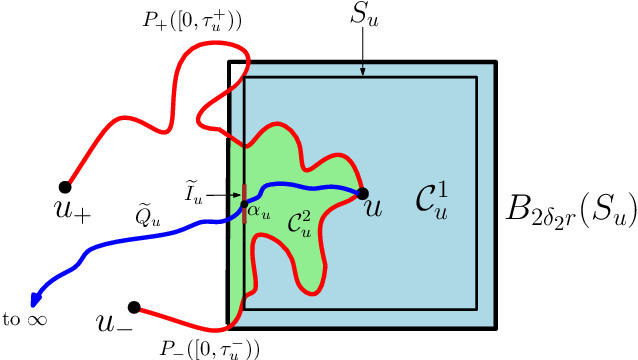}
  \caption{Illustration of the setup of conditions~\eqref{it:positive_distance_from_marked_points} and \eqref{it:boxdist} in Proposition~\ref{prop:definition_of_F_r(z)}.  The red paths represent the segments of the paths $P_+$ and $P_-$ stopped at the time that they hit $u$.  If we consider the segments of the red paths between the last time that they hit $\partial B_{2\delta_2 r}(S_u)$ and the time that they hit $u$,  then the aforementioned segments divide $B_{2\delta_2 r}(S_u)$ into two regions (green and blue).  We can take $\delta_2$ and $\delta_1$ to be sufficiently small so that the parts of $P_+$ and $P_-$ do not enter the green region after they hit $u$. The objects $\wt{Q}_u,\wt{I}_u,\alpha_u$ appear in the proof of condition \eqref{it:boxdist}. $\wt{I}_u\subseteq \partial S_u$ is a \emph{random} interval of length $100\delta r$ around the \emph{random} point $\alpha_u$. The objects $x_u,I_u$ are deterministic approximations of $\alpha_u,\wt{I}_u$ while $Q_u$ is a deterministic approximation of the portion of $\wt{Q}_u$ from $\alpha_u$ to $\infty$.}
  \label{fig:conncomp}
\end{figure}
\begin{proposition}\label{prop:definition_of_F_r(z)}  
  There exists a choice of the parameters $0<\delta_3<\delta_2<\delta_1<b<1$, $A>2$ and $\delta_1<\min\{(b/6)^{\chi'/\chi},b/8\}$ such that for any $z\in \CC, r>0$, there exist
  \begin{enumerate}
  \item \label{it:obj1} Deterministic subsets $\cA_{u_+}, \cA_{u_-},\cM,\cA_{v_+},\cA_{v_-}\subseteq B_{Ar}(z)$ which are unions of certain squares in $\cS^z_{\delta_3 r}(\mathbb{C})$,
  \item \label{it:obj2} Deterministic squares $S_u,S_v\subseteq B_{Ar}(z)$ of side length $\delta_1 r$ which belong to the family $2\cS^z_{\delta_1 r/2}(u)$\footnote{For a square $S$ and $\alpha>0$, $\alpha S$ refers to the square with the same center and $\alpha$ times the side-length as $S$.},
  \item \label{it:obj3} Deterministic segments $I_u\subseteq \partial S_u, I_v\subseteq \partial S_v$ both of Euclidean length equal to $50\delta_2 r$,
  \item \label{it:obj4} A deterministic point $x_u\in I_u$ (resp.\ $x_v\in I_v$), a deterministic path $Q_u$ (resp.\ $Q_{v}$) from $z-A^2r$ (resp.\ $z+A^2 r$) to $x_u$ (resp.\ $x_v$) satisfying $Q_{u}\setminus \{z-A^2r,x_u\}\subseteq B_{A^2 r}(z)\setminus S_u$ (resp.\ $Q_{v}\setminus \{z+A^2r,x_v\}\subseteq B_{A^2 r}(z)\setminus S_v$), and the fattened tube $\cD_u=\mathrm{int}(\bigcup \cS_{\delta_3 r}^z(Q_u))$ (resp.\ $\cD_v=\mathrm{int}(\cS_{\delta_3 r}^z(Q_v))$,
  \end{enumerate}
  satisfying the properties
  \begin{enumerate}[(a)]
    \item \label{it:geodesics_separated}
    The Euclidean balls of radius $100\delta_2r$ around the sets $\cA_{u_+}\setminus S_u, \cA_{u_-}\setminus S_u, \cM \setminus (S_u \cup S_v),  \cA_{v_+}\setminus S_v,$ and $\cA_{v_-}\setminus S_v$ are mutually disjoint.
  \item \label{it:geodesics_separated2} The Euclidean balls of radius $br$ around the sets $\{u_+\},\{u_-\}$ and $S_u\cup \mathcal{M} \cup S_v \cup \mathcal{A}_{v_+} \cup \mathcal{A}_{v_-}$ are mutually disjoint and the Euclidean balls of radius $br$ around the sets $\{v_+\},\{v_-\}$ and $\mathcal{A}_{u_+} \cup \mathcal{A}_{u_-}\cup S_u\cup \mathcal{M} \cup S_v$ are mutually disjoint.
  \item \label{it:Mdisconn}%
    The balls $B_{100\delta_2 r}( \cA_{u_+}\cup S_u\cup \cA_{u_-})$ and $B_{100\delta_2 r}( \cA_{v_+}\cup S_v\cup \cA_{v_-})$ are disjoint.
\item \label{it:I1} $B_{\delta_3 r}(\cD_u) \cap B_{\delta_3 r}(S_u) \subseteq B_{\delta_2 r}(I_u)$ and $B_{\delta_3 r}(\cD_v) \cap B_{\delta_3 r}(S_v) \subseteq B_{\delta_2 r}(I_v)$.
\item \label{it:I2} $\dist\left( \cD_u\cup I_u ,  \mathcal{A}_{u_+} \cup \mathcal{A}_{u_-} \cup \mathcal{M} \cup S_v\cup \mathcal{A}_{v_+} \cup \mathcal{A}_{v_-}  \cup \cD_v\right) \geq 50 \delta_2 r$.
\item \label{it:I3} $\dist\left( \cD_v \cup I_v , \cD_u\cup S_u \cup \mathcal{A}_{u_+} \cup \mathcal{A}_{u_-} \cup \mathcal{M} \cup \mathcal{A}_{v_+} \cup \mathcal{A}_{v_-} \right) \geq 50 \delta_2 r$.
\item \label{it:I4} $\dist\left(\cD_u\cup \cD_v, \{u_+,u_-,v_+,v_-\}\right) \geq br$.
  \end{enumerate}
such that the following is true. Let the event $F_r(z)$ be defined as the intersection of the event $G_r(z)$ from Lemma \ref{lem:1*} along with the following list of conditions.
    \begin{enumerate} [(i)]
          \item \label{it:across} We have $\sup_{w_1 ,  w_2 \in B_{Ar}(z)} D_h(w_1 ,  w_2;B_{A^{3/2}r}(z)) < D_h(\partial B_{A^{3/2}r}(z) ,  \partial B_{A^2 r}(z))$ and as a result,  every $D_h$-geodesic between any two points in the ball $B_{Ar}(z)$ is contained inside $B_{A^2r}(z)$ and is thus determined by $h\lvert_{B_{A^2r}(z)}$.
              \item \label{it:geodloc} We have  $P_+\cup P_-\subseteq B_{Ar/4}(z)$.
              \item \label{it:upper_holder} $\sup_{x,y \in S_u} D_h(x,y ; S_u) \leq \delta_1^\chi r^{\xi Q} e^{\xi h_r(z)}$ and the same with $u$ replaced by $v$.
      \item \label{it:lower_holder}
        $D_h(\partial B_{br/3}(u), \partial B_{b r/2}(u))\geq (b/6)^{\chi'} r^{\xi Q} e^{\xi h_{r}(z)}$ and the same is true when $u$ is replaced by $v$.

              \item \label{it:positive_distance_from_marked_points}
                The set $B_{2\delta_2 r}(S_u) \setminus (P_+([0,\tau_u^+]) \cup P_-([0,\tau_u^-]))$ has exactly (see Figure \ref{fig:conncomp}) two connected components $\cC^1_u$ and $\cC^2_u$ whose boundary contains $u$ and the set $P_+((\tau_u^+,1]) \cup P_-((\tau_u^-,1])$ intersects only one of $\cC^1_u$ and $\cC^2_u$,  say $\cC^1_u$.  Similarly there are exactly two connected components $\cC^1_v$ and $\cC^2_v$ of the set $B_{2\delta_2 r}(S_v) \setminus (P_+([\tau_v^+,1]) \cup P_-([\tau_v^-,1]))$ such that exactly one of them, namely $\cC_v^1$, intersects $P_+([0,\tau_v^+)) \cup P_-([0,\tau_v^-))$.
  \item \label{it:SuSvsquares} We have $u\in (1/2)S_u$ and $v\in (1/2)S_v$.\footnote{We emphasize here that $S_u,S_v$ are \emph{deterministic} squares while the points $u,v$ are the confluence points of $P_+$ and $P_-$ and are thus \emph{random}.}
  \item \label{it:boxes}  We have the equalities (see Figure \ref{fig:setup2})
    \begin{align}
            \label{eq:1*}
  &\cA_{u_+}=\inte(\bigcup\cS^z_{\delta_3 r}(P_+([0,\tau_u^+]))), \cA_{v_+}=  \inte(\bigcup\cS^z_{\delta_3 r}(P_+([\tau_v^+,1]))).\nonumber\\
  &\cA_{u_-}=\inte(\bigcup \cS^z_{\delta_3 r}(P_-([0,\tau_u^-]))), \cA_{v_-}=  \inte(\bigcup\cS^z_{\delta_3 r}(P_-([\tau_v^-,1]))),\nonumber\\
  &\cM=\inte(\bigcup\cS^z_{\delta_3r}(P_+([\tau_u^+,\tau_v^+])))=\inte(\bigcup(\cS^z_{\delta_3r}(P_-([\tau_u^-,\tau_v^-]))).
    \end{align}

  \item \label{it:Mcomp}We have $B_{2\delta_2 r}(\cM\setminus S_u)\cap \cC_u^2 = \emptyset$. Similarly, $B_{2\delta_2 r}(\cM\setminus S_v)\cap \cC_v^2 = \emptyset$,
    
  \item \label{it:boxdist} The deterministic segment $I_u\subseteq \partial S_u$ (resp.\ $I_v\subseteq \partial S_v$) satisfies $B_{\delta_2 r}(I_u)\subseteq \cC_u^2$ (resp.\ $B_{\delta_2 r}(I_v)\subseteq\cC_v^2)$.%

  \end{enumerate}
  Then the event $F_r(z)$ is measurable with respect to $h\lvert_{B_{A^2 r}(z)}$ viewed modulo an additive constant\footnote{Consider the equivalence relation $\sim$ wherein for two generalised functions $\mathtt{h}_1,\mathtt{h}_2$ on a domain $U$, we say $\mathtt{h}_1\sim \mathtt{h}_2$ if $\mathtt{h}_1-\mathtt{h}_2$ is constant. Let $\pi$ be the map projecting a generalised function $\mathtt{h}$ into its $\sim$ equivalence class. Then  $\mathtt{h}$ ``viewed modulo an additive constant'' refers to $\pi(\mathtt{h})$.}, and there exists a constant $p_1 \in (0,1)$ which is independent of $z$ and $r$ such that $\mathbb{P}(F_r(z)) \geq p_1$. Further, the $z,r$ dependent deterministic sets in items \eqref{it:obj1}, \eqref{it:obj2}, \eqref{it:obj3}, and \eqref{it:obj4} can be chosen such that the family of events $\{F_r(z)\}_{z\in \CC,r>0}$ is translation and scale invariant.
  \end{proposition}
  \begin{proof}
    Begin with choosing $b$ to be small enough such that for all $z,r$, the event $G_r(z)$ from Lemma \ref{lem:1*} satisfies
    \begin{equation}
      \label{eq:107}
      \PP(G_r(z))\geq 1/10.
    \end{equation}
We shall now go through all the conditions in Proposition \ref{prop:definition_of_F_r(z)} step by step.
  \item \paragraph{\textbf{Local dependence of geodesics-- items \eqref{it:across}, \eqref{it:geodloc}}} That
    \begin{equation}
      \label{eq:89}
      \sup_{w_1 ,  w_2 \in B_{Ar}(z)} D_h(w_1 ,  w_2;B_{A^{3/2}r}(z)) < D_h(\partial B_{A^{3/2}r}(z) ,  \partial B_{A^2 r}(z))
    \end{equation}
    in item \eqref{it:across}  holds with arbitrarily high probability as $A\rightarrow \infty$ follows by using the scaling properties of the LQG metric (see Proposition \ref{prop:8}) and the whole plane GFF combined with the fact that, almost surely,
\begin{align*}
  D_h(\partial B_1(0) ,  \partial B_{A^{1/2}}(0)) \to \infty \quad \text{as} \quad A \to \infty.
\end{align*}
An easy consequence of this is the fact that every $D_h$-geodesic between points of $B_{Ar}(z)$ stays within $B_{A^2r}(z)$ (see \cite[Lemma 4.8]{GPS20}). For item \eqref{it:geodloc}, since $P_+,P_-$ are bounded sets, we simply note that by taking $A$ to be large enough, we can ensure that $P_+\cup P_-\subseteq B_{Ar/4}(z)$ with arbitrarily high probability.
\item \paragraph{\textbf{H\"older continuity-- items \eqref{it:upper_holder}, \eqref{it:lower_holder}}}
 By basic H\"older continuity estimates for the LQG metric (see \cite[Lemmas~3.19,3.22]{DFGPS20} and Lemma~\ref{lem:holder_regularity}), it is straightforward to see that items \eqref{it:upper_holder} and \eqref{it:lower_holder} can be made to hold with arbitrarily high probability by taking $b,\delta_1$ to be small enough. \emph{The parameter $b$ is now fixed and will not be altered now.}
\item \paragraph{\textbf{The connected components $\cC_u^1,\cC_u^2,\cC_v^1,\cC_v^2$-- item \eqref{it:positive_distance_from_marked_points}}} We will only prove the claim for the point $u$ since the exact same argument works for $v$. Let $\sigma_u^+$ (resp.\ $\sigma_u^-$) denote the last time before $\tau_u^+$ (resp.\ $\tau_u^-$) that $P_+$ (resp.\ $P_-$) intersects $\partial B_{2\delta_2 r}(S_u)$.
  Since $P_+|_{[0,\tau_u^+]}$ and $P_-|_{[0,\tau_u^-]}$ are both simple curves which intersect only at $u$, we obtain that there are exactly two connected components $\mathcal{C}_u^1$ and $\mathcal{C}_u^2$ of 
\begin{align*}
    B_{2\delta_2 r}(S_u) \setminus (P_+([0,\tau_u^+]) \cup P_-([0,\tau_u^-]))
\end{align*}
whose boundaries contain $u$ and we choose these such that $\partial \mathcal{C}_u^2$ contains both the right side of $P_+([\sigma_u^+ , \tau_u^+])$ and the left side of $P_-([\sigma_u^- , \tau_u^-])$, and $\partial \mathcal{C}_u^1$ contains both the right side of $P_-([\sigma_u^- , \tau_u^-])$ and the left side of $P_+([\sigma_u^+ , \tau_u^+])$. Now, since we are working on $F_r(z)$, $\scX^{u_+,v_+}_{u_-,v_-}$ does occur and as a result (see item \eqref{it:sides} in Definition \ref{def:1*}), the geodesic $\Gamma_{u_-,u}$ lies to the right of $P_+$. As a consequence of this, for some small enough $\varepsilon>0$, we must have $P_+((\tau_u^+ , \tau_u^++\varepsilon]),P_-((\tau_u^- , \tau_u^-+\varepsilon])\subseteq \cC_u^1$.

Recall that $|u_+-u|,|u_--u|>4br$ (see Lemma \ref{lem:1*}). Since geodesics are simple curves, $P_+((\tau_u^+ , 1]) \cup P_-((\tau_u^- , 1])$ cannot intersect $P_+([0,\tau_u^+]) \cup P_-([0,\tau_u^-])$. Thus, by planarity, in order for there to exist a $w\in (P_+((\tau_u^+ , 1]) \cup P_-((\tau_u^- , 1]))\cap \cC_u^2$, we must have $D_h(u,w)\geq D_h(u,\partial B_{4br}(u))$. However, by choosing $\delta_1,\delta_2$ to be small enough relative to $b$, we can ensure that $\diam_{D_h}(B_{2\delta_2 r}(S_u)) < D_h(u,\partial B_{4br}(u))$ with high probability, and when this occurs, we must have $(P_+((\tau_u^+ , 1]) \cup P_-((\tau_u^- , 1]))\cap \cC_u^2=\emptyset$.
\item \paragraph{\textbf{Covering $P_+,P_-$ by tiny squares-- items \eqref{it:SuSvsquares},  \eqref{it:boxes}, \eqref{it:Mcomp} and items \eqref{it:geodesics_separated} to \eqref{it:Mdisconn}}} %
  By the pigeonhole principle and \eqref{eq:107}, we can fix squares $S_u,S_v\in 2\cS^z_{\delta_1r/2}(u)$ such that, conditional on the event $G_r(z)$, we additionally have $u\in S_u, v\in S_v$ with strictly positive probability. Similarly, by the pigeonhole principle, we can choose $\cA_{u_+},\cA_{u_-},\cM,\cA_{v_+},\cA_{v_-}$ such that \eqref{eq:1*} holds with positive probability conditional on the event $G_r(z)$. Note that we can ensure that all the above sets are subsets of $B_{Ar}(z)$ by working with only $\delta_3<A/4$; indeed, this is true because item \eqref{it:across} guarantees that $P_+\cup P_-\subseteq B_{Ar/4}(z)$. Items \eqref{it:geodesics_separated} and \eqref{it:geodesics_separated2} can be ensured by simply choosing a small enough $\delta_1<(b/6)^{\chi'/\chi}$ and then $\delta_2$ to be much smaller than $\delta_1$ and combining with Lemma~\ref{lem:1*}. Similarly, since $P_+([0,\tau_u^+]) \cup P_-([0,\tau_u^-])$ and $P_+([\tau_v^+,1]) \cup P_-([\tau_v^-,1])$ are disjoint almost surely, by taking $\delta_1$ to be smaller and $\delta_2 \in (0,1)$ to be much smaller, we can ensure that item \eqref{it:Mdisconn} holds with high probability. \emph{The parameter $\delta_1$ is now fixed and will not be altered now.}

  As for item \eqref{it:Mcomp},  we note that by item \eqref{it:positive_distance_from_marked_points}, $P_+((\tau_u^+,1]) \cup P_-((\tau_u^-,1])$ does not intersect $\mathcal{C}_u^2$.  Hence by taking $\delta_2 \in (0,1)$ to be even smaller, we can arrange so that $B_{2\delta_2 r}(\mathcal{M} \setminus S_u)$ does not intersect 
$\mathcal{C}_u^2 \cup P_+([0,\tau_u^+]) \cup P_-([0,\tau_u^-])$.  Therefore it follows that 
\begin{align*}
B_{2\delta_2 r}(\mathcal{M} \setminus S_u) \cap \mathcal{C}_u^2 = \emptyset.
\end{align*}
Similarly we obtain that 
\begin{align*}
B_{2\delta_2 r}(\mathcal{M} \setminus S_v) \cap \mathcal{C}_v^2 = \emptyset.
\end{align*}
\begin{figure}
  \centering
  \includegraphics[width=0.9\linewidth]{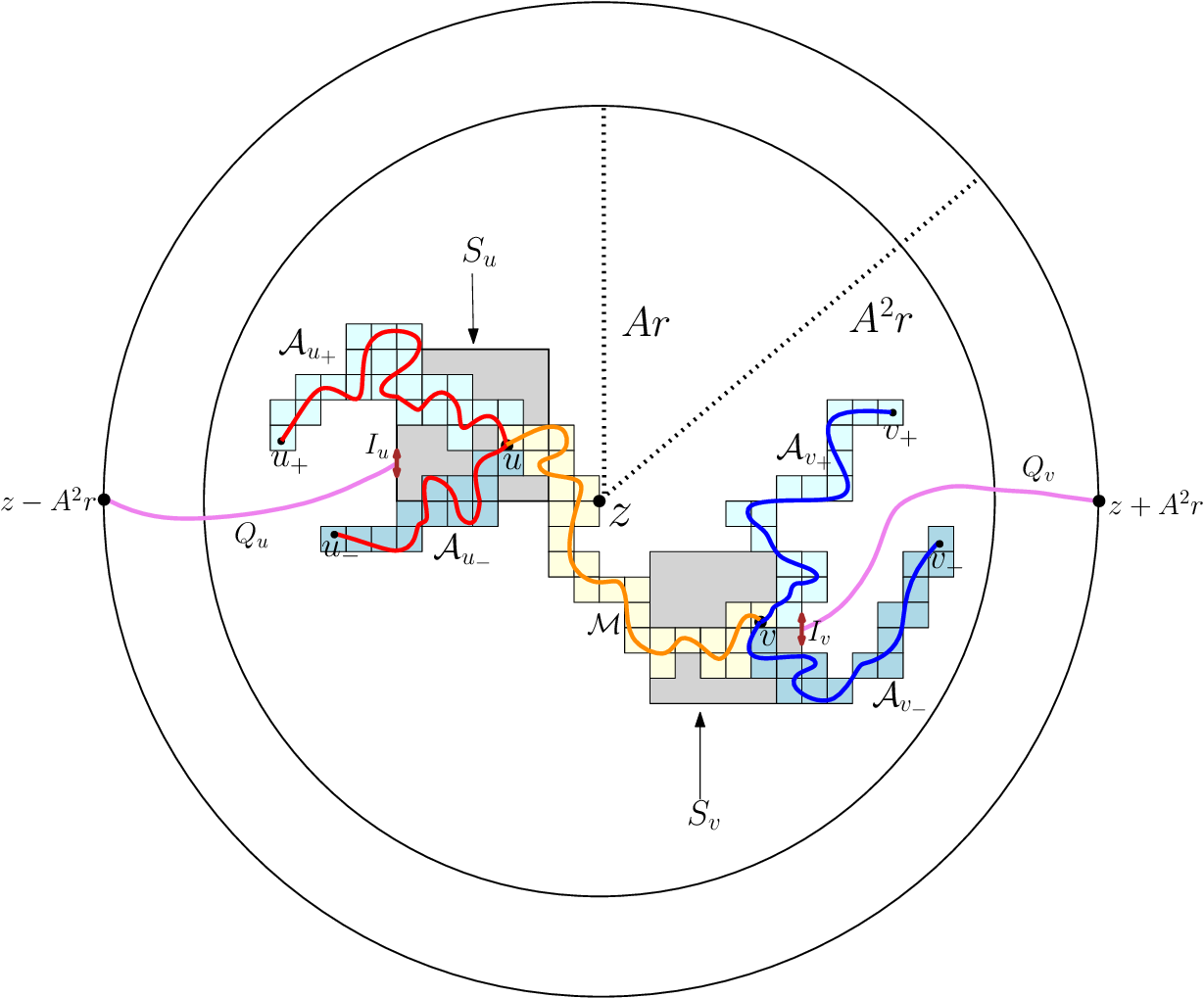}
  \caption{The setup in Proposition \ref{prop:definition_of_F_r(z)}: $S_u$ (resp.\ $S_v$) is a square of size $\delta_1 r$ containing $u$ (resp.\ $v$) depicted in grey colour.  The red (resp.\ blue) paths represent the segments of the paths $P_+,P_-$ from $u_+$ to $u$ and $u_-$ to $u$ (resp.\ $v$ to $v_+$ and $v$ to $v_-$) while the orange curve represents the common segment between $P_+$ and $P_-$.  The union of the red curves is at least $100 \delta_2 r$ away from the union of the blue curves and the two red (resp.\  blue) curves minus $S_u$ (resp.\ $S_v$) are at least $100 \delta_2 r$ apart.  $\mathcal{A}_{u_+}$ (cyan),  $\mathcal{A}_{u_-}$ (light blue),  $\mathcal{M}$ (yellow),  $\mathcal{A}_{v_+}$ (cyan),  $\mathcal{A}_{v_-}$ (light blue) are families of squares covering the red path from $u_+$ to $u$,  the red path from $u_-$ to $u$,  the orange path,  the blue path from $v$ to $v_+$,  and the blue path from $v$ to $v_-$ respectively.  Further, the violet path $Q_u$ (resp.\ $Q_v$) intersects $\mathcal{A}_{u_+} \cup \mathcal{A}_{u_-} \cup \mathcal{M} \cup \mathcal{A}_{v_+} \cup \mathcal{A}_{v_-} \cup S_u$ (resp.\ $\mathcal{A}_{u_+} \cup \mathcal{A}_{u_-} \cup \mathcal{M} \cup \mathcal{A}_{v_+} \cup \mathcal{A}_{v_-} \cup S_v$ ) precisely at the arc $I_u$ (resp.\ $I_v$). The set $\cD_u$ (resp.\ $\cD_v$) is a tube consisting of (not shown) $\delta_3 r$ sized boxes around $Q_u$ (resp.\ $Q_v$).}
  \label{fig:setup2}
\end{figure}
\item \paragraph{\textbf{The intervals $I_u$ and $I_v$: item \eqref{it:boxdist} and items \eqref{it:I1} to \eqref{it:I4}}}

  Since $P_+ \cup P_-\cup B_{br}(\{u_+,u_-,v,v_+,v_-\})$ is a bounded connected set, the unique unbounded component of
  \begin{equation}
    \label{eq:90}
    \CC\cup \{\infty\} \setminus (P_+ \cup P_-\cup B_{br}(\{u_+,u_-,v,v_+,v_-\}))
  \end{equation}
  is a simply connected domain. Now, as a consequence of item \eqref{it:geodesics_separated2}, the point $u$ corresponds to exactly two prime ends on the boundary of this domain; one corresponds to the right side of the Jordan curve $P_+([0,\tau_u^+]) \cup P_-([0,\tau_u^-])$ viewed as being traversed from $u_+$ to $u_-$, while the other prime end corresponds to its left side. Therefore, there exists (see Figure \ref{fig:conncomp}) a random path $\wt{Q}_{u}\colon [0,\infty)\rightarrow \CC$ in the unbounded component of $\mathbb{C} \cup \{\infty\} \setminus (P_+ \cup P_-\cup B_{br}(\{u_+,u_-,v,v_+,v_-\}))$ from $u$ (seen as prime end to the right side of $P_+([0,\tau_u^+]) \cup P_-([0,\tau_u^-])$) to $\infty$. %
  Note that $\wt{Q}_{u}$ can be chosen measurably as a function of the set $P_+ \cup P_-$ and that $\wt{Q}_{u}\setminus \{u\}$ is disjoint from $(P_+ \cup P_-\cup B_{br}(\{u_+,u_-,v_+,v_-, v\})$ and for a small enough (random) $\varepsilon>0$, we have
  \begin{equation}
    \label{eq:91}
    \wt{Q}_{u}((0,\varepsilon))\subseteq \cC_u^2.
  \end{equation}
  One can analogously define a random path $\wt{Q}_{v}$ joining $v$ to $\infty$ such that $\wt{Q}_v\setminus \{v\}$ is disjoint from $(P_+ \cup P_-\cup B_{br}(\{u_+,u_-,v_+,v_-,u\})$, and we additionally have $\wt{Q}_{v}((0,\varepsilon))\subseteq \cC_v^2$.

  Now, consider the last point $\alpha_u$ on $\partial S_u$ visited by $\wt{Q}_{u}$, let $\wt{Q}'_{u}\subseteq \wt{Q}_{u}$ be the subpath of $\wt{Q}_{u}$ from $\alpha_u$ onwards. Note that $\alpha_u\in \cC_u^2$ for any choice of $\delta_2>0$. Note that the Euclidean distance between $\alpha_u$ and $P_+\cup P_-$ is almost surely positive. Therefore, by taking $\delta_2$ to be sufficiently small, we obtain that with arbitrarily high probability, the Euclidean distance between $\wt{Q}_u'$ and $P_+ \cup P_-\cup \wt{Q}_v$ is at least $200\delta_2 r$. %
  As a result, on the above event, we can choose a random interval $\wt{I}_u$ of length equal to $100\delta_2 r$ around $\alpha_u$ such that
  \begin{equation}
    \label{eq:108}
    \dist( B_{100\delta_2 r}(P_+\cup P_-\cup \wt{Q}_v), \wt{I}_u\cup \wt{Q}_u')\geq 100\delta_2 r
  \end{equation}
Now, since $\alpha_u$ is the last point on $\partial S_u$ visited by $\wt{Q}_u$, by choosing a $\delta_3$ much smaller than $\delta_2$, we can also ensure that $\wt{Q}'_{u}$ does not intersect
\begin{displaymath}
    B_{100\delta_2 r}(P_+\cup P_-\cup \wt{Q}_{v}) \cup B_{br}(\{u_+,u_-,v,v_+,v_-\})\cup (B_{\delta_3 r}(S_u)\setminus B_{\delta_2 r}(\wt{I}_u)).
\end{displaymath}
By symmetry, the above steps can be repeated with $\wt{I}_u$ swapped with an analogously defined random interval $\wt{I}_v\subseteq \partial S_v$ and the path $\wt{Q}_u$ swapped with $\wt{Q}_{v}$. Finally, by a pigeonhole argument, conditional on the event $G_r(z)$, with positive probability, we can fix a deterministic interval $I_u\subseteq \wt{I}_u$ (resp.\ $I_v\subseteq \wt{I}_v$), a deterministic point $x_u\in I_u$ (resp.\ $x_v\in I_v$) and a deterministic path $Q_u$ (resp.\ $Q_{v}$) such that the Euclidean length of $I_u$ (resp.\ $I_v$) is equal to $50\delta_2 r$ and with the definition $\cD_u=\mathrm{int}(\bigcup \cS^z_{\delta_3 r}(Q_u))$ (resp.\ $\cD_v=\mathrm{int}(\cS^z_{\delta_3 r}(Q_v))$), all the conditions \eqref{it:obj3} to \eqref{it:obj4}, \eqref{it:boxdist} along with items \eqref{it:I1} to \eqref{it:I4} are satisfied.

\paragraph{\textbf{Probability lower bound, measurability, translation and scale invariance}} Having gone through all the conditions, on combining with \eqref{eq:107}, we obtain that there exists a constant $p_1 \in (0,1/10)$ which is independent of $z$ and $r$ such that $\mathbb{P}(F_r(z)) \geq p_1$.

Finally, that $F_{r}(z)$ is measurable with respect to $h\lvert_{B_{A^2r}(z)}$ viewed modulo an additive constant is true as the condition in \eqref{it:across} is measurable with respect to $h\lvert_{B_{A^2r}(z)}$ modulo additive constant. Further, note that the points $u_-,u_+,v_-,v_+$ and all the sets $\cA_{u_+},\cA_{u_-},\cM,\cA_{v_+},\cA_{v_-}$ lie in $ B_{Ar}(z)\subseteq B_{A^2r}(z)$. It can be verified that when \eqref{it:across} holds, all the conditions in the lemma can be verified by only revealing $h\lvert_{B_{A^2r}(z)}$ viewed modulo an additive constant (see the proof of \cite[Lemma 4.8]{GPS20}). The translation and scale invariance of the family $\{F_r(z)\}_{z\in \CC,r>0}$ follows immediately from the translation and scale invariance of $\{G_r(z)\}_{z\in \CC,r>0}$. Indeed, to define the objects in items \eqref{it:obj1}, \eqref{it:obj2}, \eqref{it:obj3}, we can first work with $(z,r)=(0,1)$ and then just subsequently just translate and scale the obtained objects for general $(z,r)$.
  \end{proof}

  \textbf{From now on, we shall always work with fixed parameters $b,A,\delta_1,\delta_2,\delta_3$ such that the result holds-- these parameters will not be changed for the rest of the paper.}

  \begin{remark}
    \label{rem:slightmod}
    We now describe a slight modification of Proposition \ref{prop:definition_of_F_r(z)} that can be obtained via the same argument. Starting with a $\Phi\in (0,1/10)$ and a $p_2\in (0,1)$ possibly depending on $\Phi$, one can choose $A>1$ depending only on $p_2$ such that the following holds. Using $\pi(\mathtt{h})$ to denote a field $\mathtt{h}$ viewed modulo an additive constant, let  $\{H_r(z)\}_{z\in \CC,r>0}$ be an additional translation and scale invariant family of events $\{H_r(z)\}_{z\in \CC,r>0}$ measurable with respect to $\sigma(P_1\cup P_2, \bigcap_{\varepsilon>0}\sigma(\pi(h)\lvert_{B_\varepsilon(P_1\cup P_2)}))$ such that $\PP(H_r(z))>p_2$. Then we can choose the parameters $\delta_3,\delta_2,\delta_1,b$ in Proposition \ref{prop:definition_of_F_r(z)} to be small enough such that the event $F_r(z)\cap H_r(z)$ is measurable with respect to $h\lvert_{B_{A^2r}(z)}$ and has strictly positive probability (as opposed to just demanding that $F_r(z)$ has positive probability, as is shown in Proposition \ref{prop:definition_of_F_r(z)}). As a result, we can legitimately replace $F_r(z)$ by the more restrictive event $F_r(z)\cap H_r(z)$ for defining the events $E_r(z)$ later (see Section \ref{sec:bump_functions}) and thereby for constructing the $\scX$s. This will not be required in this paper but can, in principle, be used to ensure that the $\scX$s that are constructed (see Section \ref{sec:proof-outline}) along a geodesic $\Gamma_{\bz,\bw}$ satisfy additional properties (encoded via the $H_r(z)$ events) depending on the desired application.
  \end{remark}
  Proposition \ref{prop:definition_of_F_r(z)} yields deterministic points $u_+,v_+,u_-,v_-$ and deterministic families of squares $\cA_{u_+},\cA_{u_-},\cM,\cA_{v_+},\cA_{v_-}$ which depend on the choice of $z\in \CC$ and the scale parameter $r>0$. For simplicity, we shall define $V_r(z)$ as the interior of the set
\begin{equation}
  \label{eq:26}
  \cA_{u_+}\cup\cA_{u_-}\cup\cM\cup\cA_{v_+}\cup\cA_{v_-} \cup S_u \cup S_v\subseteq B_{Ar}(z),
\end{equation}
where we emphasize that the set on the left hand side above depends on the choice of $z,r$. To summarise, $V_r(z)$ is a union of certain squares from the family $\cS^z_{\delta_3 r}(\mathbb{C})$ along with the two $\delta_1$-sized squares $S_u,S_v$. Importantly, on the positive probability event $F_r(z)$ that we just defined, we have $P_-\cup P_+\subseteq V_r(z)$. For convenience, we also define
\begin{equation}
  \label{eq:95}
  O_r(z)= \cD_u\cup V_r(z)\cup \cD_v.
\end{equation}
We now state a basic geometric property of the set $V_r(z)$ that will be useful for us shortly.
\begin{lemma}
  \label{lem:11}
Fix $z\in \CC, r>0$. On the event $F_r(z)$, the following hold.
  \begin{enumerate}
  \item  Any path $\xi\subseteq B_{\delta_2 r}(V_r(z))$ from $B_{\delta_2 r}(I_u)$ to $B_{\delta_2 r}( (\cM\setminus S_u)\cup S_v\cup \cA_{v_+}\cup \cA_{v_-})$ must necessarily intersect $B_{\delta_2 r}(\cA_{u_-}\cup \cA_{u_+})$.
  \item Any path $\xi\subseteq B_{\delta_2 r}(V_r(z))$ from $B_{\delta_2 r}(I_v)$ to $B_{\delta_2 r}( (\cM\setminus S_v)\cup S_u\cup \cA_{u_+}\cup \cA_{u_-})$ must necessarily intersect $B_{\delta_2 r}(\cA_{v_-}\cup \cA_{v_+})$.
  \end{enumerate}
\end{lemma}
\begin{proof}
By symmetry, it suffices to simply prove that the first point holds. With the aim of eventually obtaining a contradiction, assume that there exists a path $\xi\colon [a,b]\rightarrow \CC$ such that $\xi\subseteq B_{\delta_2 r}(V_r(z))$ from $B_{\delta_2 r}(I_u)$ to $B_{\delta_2 r}((\cM\setminus S_u)\cup S_v\cup \cA_{v_+}\cup \cA_{v_-})$ which does not intersect $B_{\delta_2 r}(\cA_{u_-}\cup \cA_{u_+})$ and thus satisfies $\xi\subseteq B_{\delta_2 r}(S_u\cup \cM\cup \cA_{v_+}\cup \cA_{v_-})$. Now, recall that by \eqref{it:Mdisconn} in Proposition \ref{prop:definition_of_F_r(z)}, we have
  \begin{equation}
    \label{eq:58}
    \mathrm{dist}(B_{\delta_2 r}(S_u), B_{\delta_2 r}(S_v\cup \cA_{v_-}\cup \cA_{v_+})) \geq 98\delta_2 r>0,
  \end{equation}
  and as a result, by replacing $\xi$ by a shorter subsegment, we can assume that $\xi$ starts at $B_{\delta_2 r}(I_u)$, ends in $B_{\delta_2 r}(\cM\setminus S_u)$ and satisfies $\xi\subseteq B_{\delta_2 r}(S_u\cup \cM)$. 

  Now, recall the connected components $\cC_u^1$ and $\cC_u^2$ from \eqref{it:positive_distance_from_marked_points} in Proposition \ref{prop:definition_of_F_r(z)} and also the conditions \eqref{it:Mcomp}, \eqref{it:boxdist} in the definition of $F_r(z)$. The goal now is to show that there exists a subpath $\xi'\subseteq \xi$ from a point in $\cC_u^2$ to a point in some connected component $G$ of $B_{2\delta_2 r}(S_u) \setminus (P_+([0,\tau_u^+]) \cup P_-([0,\tau_u^-]))$ that is distinct from $\mathcal{C}_u^2$ such that $\xi'\subseteq B_{\delta_2 r}(S_u)$. %
We now define $b_*=\inf_{s\in [a,b]}\{\xi(s)\in B_{\delta_2 r}(\cM\setminus S_u)\}$ and set $\xi'=\xi\lvert_{[a,b_*]}$; note that we necessarily have $\xi'\subseteq B_{\delta_2 r}(S_u)$. Further, since $\xi$ is assumed to not intersect $B_{\delta_2 r}(\cA_{u_-}\cup \cA_{u_+})$, $\xi'$ cannot intersect $P_+([0,\tau_u^+])\cup P_-([0,\tau_u^-])$. As a result, $\xi'\subseteq B_{2\delta_2 r}(S_u)\setminus (P_+([0,\tau_u^+])\cup P_-([0,\tau_u^-]))$. Now, by the definition of $b_*$ and $\cM$, we must have $\xi(b_*)\in B_{2\delta_2 r}(\cM\setminus S_u)\cap B_{2\delta_2 r}(S_u)$ and thus by \eqref{it:Mcomp}, we must have $\xi(b_*)\in G$ for some connected component $G$ of $B_{2\delta_2 r}(S_u) \setminus (P_+([0,\tau_u^+]) \cup P_-([0,\tau_u^-]))$ that is distinct from $\mathcal{C}_u^2$. Therefore $\xi'$ is a path from $B_{\delta_2 r}(I_u)\subseteq \cC_u^2$ (see \eqref{it:boxdist}) to a point in $G$. However, this contradicts that $G$ and $\cC_u^2$ are distinct connected components of $B_{2\delta_2 r}(S_u)\setminus (P_+([0,\tau_u^+])\cup P_-([0,\tau_u^-]))$. Thus our initial assumption that a path $\xi$ exists must be incorrect, and this yields a contradiction, thereby completing the proof.
\end{proof}

\subsection{Constructing deterministic tubes and proving that $F_r(z)$ occurs at many places with high probability.}\label{sec:tubes}
The event $F_r(z)$ that we just defined only occurs with positive probability. However, the independence result Proposition~\ref{thm:theorem_uniqueness_lqg_metric} requires that we work with an event $E_r(z)$ of probability sufficiently close to $1$. For this reason, we will need to boost the probability of $F_r(z)$ and this will be done by dividing the space into distinct regions and looking at the event that $F_r(z)$ occurs at least somewhere. Such an approach was also used in the work \cite{GM21}, and thus, to improve readability for readers already familiar with \cite{GM21}, we use a very similar set up.
\subsubsection{Proving that $F_r(z)$ occurs at many places with high probability.}
  Let $p_1 \in (0,1)$ be such that $\mathbb{P}(F_r(z)) \geq p_1$ for all $z \in \mathbb{C} ,  r>0$,  where $p_1$ is independent of $z$ and $r$ (see Proposition~\ref{prop:definition_of_F_r(z)}). Fix $r>0$ and $\bp ,  \delta \in (0,1)$ (to be chosen in a way which does not depend on $z,r$).  We shall now use the following basic result from \cite{GM21}.

\begin{lemma}
(\cite[Lemma~2.7]{GM21})
Let $h$ be a whole plane GFF and fix $s>0$.  Let $n \in \mathbb{N}$ and let $\mathcal{Z}$ be a collection of $n$ points in $\mathbb{C}$ such that $|z-w| \geq 2(1+s)$ for all distinct points $z,w \in \mathcal{Z}$.  For $z \in \mathcal{Z}$, let $E_z$ be an event measurable with respect to $\sigma(h|_{B_1(z)}-h_{1+s}(z))$.  Then, for all $p,q \in (0,1)$,  there exists $n_* = n_*(s,p,q) \in \mathbb{N}$ such that if $\mathbb{P}(E_z) \geq p$ for all $z \in \mathcal{Z}$,  then $\mathbb{P}\left(\cup_{z \in \mathcal{Z}} E_z \right) \geq q$ for all $n \geq n_*$.
\end{lemma}
We define $\rho :=  \delta (n_*)^{-1}/24$ and we choose $n_* \in \mathbb{N}$ to be large enough so that the conclusion of the above lemma is satisfied with $s = \frac{1}{2}, p_1$ in place of $p$,  and $1-\frac{\delta (1-\mathbbm{p})}{16\pi}$ in place of $q$. In addition, we assume that $n_*$ is large enough so that $|e^{6i\rho}-1|=2\sin(3\rho)> 3\rho$. With $i=\sqrt{-1}$, we now define the set of points
\begin{align*}
\mathcal{Z}:= \left\{r \exp\left(-i k(6\rho)\right) : k \in [\![1,\pi (3\rho)^{-1}]\!] \right\} \subseteq \partial B_r(0).
\end{align*}
Note that the distance between any two distinct points of $\mathcal{Z}$ is at least $|e^{6i\rho}-1| r >3\rho r$. We now have the following.

\begin{lemma}\label{lem:event_occurs_at_many_places}
With probability at least $1 - \frac{1-\bp}{4}$,  each arc of $\partial B_r(0)$ with Euclidean length at least $\frac{\delta r}{2}$ contains a point $z \in \mathcal{Z}$ such that $F_{\frac{\rho r}{A^2}}(z)$ occurs.
\end{lemma}

\begin{proof}
To begin, note by Proposition \ref{prop:definition_of_F_r(z)} that the events $F_{\frac{\rho r}{A^2}}(z)$ for $z \in \mathcal{Z}$ all have probability at least $p_1$ and are determined by the field $(h-h_{3\rho r / 2}(z))|_{B_{\rho r}(z)}$ by the locality of the LQG metric (Property \ref{it:locality}).  Moreover, by the definition of $\rho$, each arc $I \subseteq \partial B_r(0)$ with Euclidean length at least $\frac{\delta r}{4}$ satisfies $|\mathcal{Z} \cap I| \geq n_*$.  Therefore it follows by applying the previous lemma with the whole plane GFF $h(\frac{\cdot}{\rho r})-h_{\rho r}(0)$ in place of $h$ (see Section~\ref{sec:gaussian-free-field}) that for each such arc $I$, it holds that
\begin{align}\label{eqn:event_occurs_with_high_probability}
\mathbb{P}\left[\text{there exists} \,\,  z \in \mathcal{Z}\cap I \,\,\text{such that} \,\,F_{\frac{\rho r}{A^2}}(z) \,\,\text{occurs}\right] \geq 1 - \frac{\delta (1-\bp)}{16\pi}.
\end{align}
Furthermore we can choose at most $4 \pi \delta^{-1}$ arcs of $\partial B_r(0)$ with Euclidean length $\frac{\delta r}{4}$ in such a way that each arc of $\partial B_r(0)$ with Euclidean length at least $\frac{\delta r}{2}$ contains one of these arcs.  Hence,  the claim of the lemma follows by combining ~\eqref{eqn:event_occurs_with_high_probability} with a union bound.
\end{proof}

\subsubsection{Constructing deterministic tubes.}
\label{sec:manytubes}
Via Lemma \ref{lem:event_occurs_at_many_places}, we have ensured that with high probability, the events $F_{\rho r/A^2}(z)$ occur for many points $z\in \mathcal{Z}$. As discussed in Section~\ref{sec:gamma-lqg-metric}, recall that our goal later would be to argue that with high probability,  the $\scX$s in fact exist along a geodesic $\Gamma_{\bz,\bw}$ for fixed points $\bz,\bw$ which are not too close. In order to do so, we will precisely use the set up above and will want to route a geodesic to go through the $\scX$s guaranteed by the events $F_{\rho r/A^2}(z)$ for $z\in \mathcal{Z}$. To do so, we will need to define a set of deterministic tubes linking together different such $\scX$s through which we will finally guide the geodesic $\Gamma_{\bz,\bw}$ to go through. The plan now is to introduce this family of tubes, and to do so, we will need to introduce various deterministic tubes.

\begin{figure}
  \centering
  \includegraphics[width=0.9\linewidth]{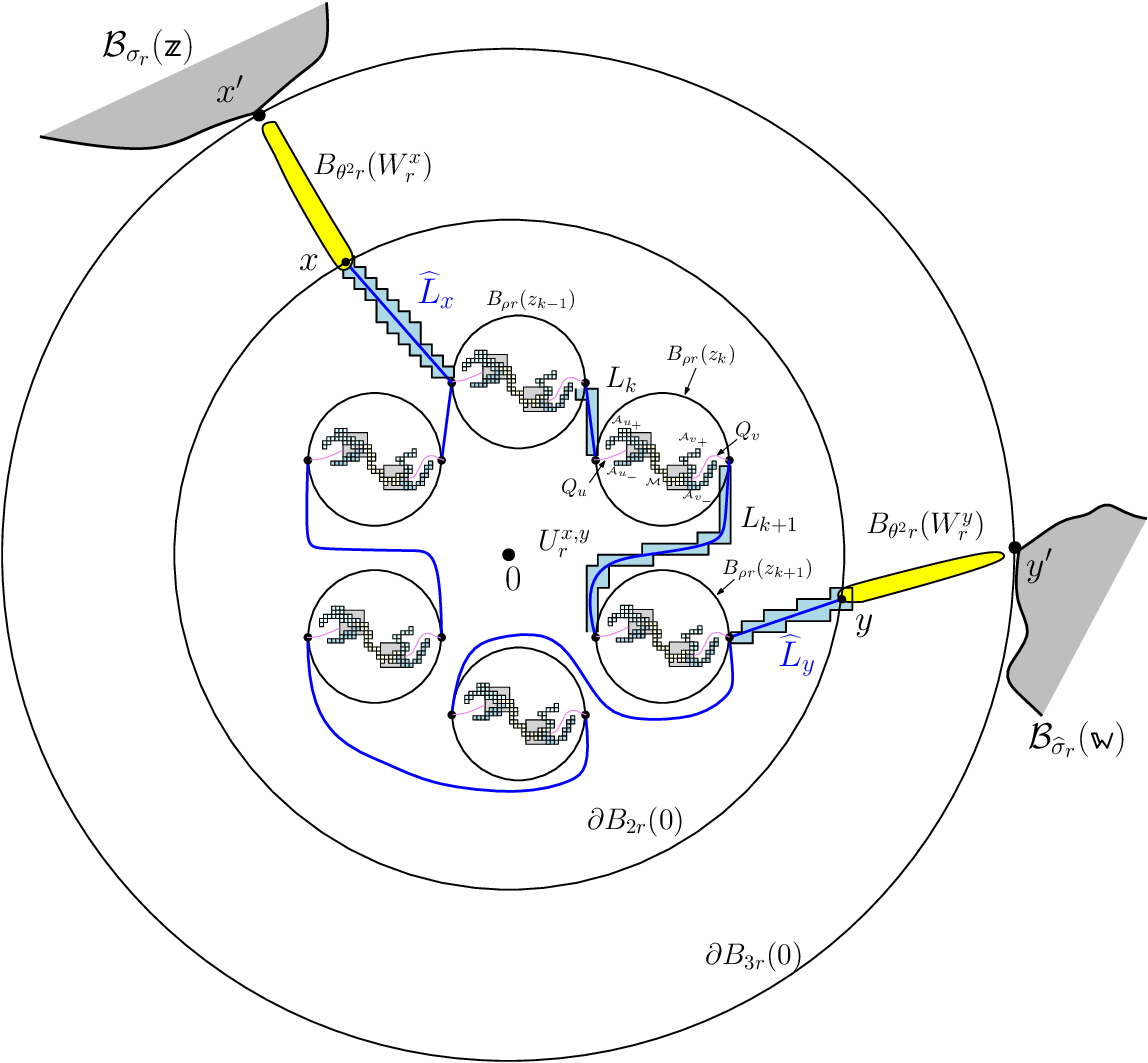}
  \caption{The setup in Section~\ref{sec:tubes}, this figure should be compared with \cite[Figures 11, 12]{GM21}: We grow LQG metric balls centered at $\bz$ and $\bw$ respectively and stop them at the first time that they intersect $\partial B_{3r}(0)$ (grey regions).  Then we construct thin tubes (yellow regions) which connect the aforementioned metric balls with $\partial B_{2r}(0)$.  Further,  we use pairwise disjoint tubes (light blue regions) which connect pairwise disjoint Euclidean balls such that each one of these balls contain a $\scX$ configuration (light grey).  Within each ball containing a $\scX$ configuration,  we construct disjoint paths (red colour) which connect the $\scX$ configuration with the light blue tubes.}
  \label{fig:detailedchi}
\end{figure}
\subsubsection*{The tubes $L_k$ connecting $\partial B_{\rho z}(z)$ for $z\in \mathcal{Z}$}
\label{sec:paths-l_k}
We enumerate $\mathcal{Z} = \{z_1,\cdots,z_N\}$,  where $N:=\lfloor \pi (3 \rho)^{-1} \rfloor$ and $z_k :=\exp\left(-ik(6 \rho)\right)$.   Also, for $n\in \ZZ\setminus [\![1,N]\!]$, set $z_n=z_k$, where $k\in [\![1,N]\!]$ and $n-k$ is divisible by $N$.  For $k \in [\![1,N]\!]$ we choose in a deterministic manner a piecewise linear path $L_k'$ from $z_{k-1}+\rho r$ to $z_k - \rho r$ which is contained in $\A_{(1-4\rho)r ,  (1+4\rho)r}(0)$ and such that the following are satisfied.  %
\begin{enumerate}
\item \label{it:tubel1}The paths $L_k'$s do not intersect (see Figure \ref{fig:detailedchi}) any of the (open) balls $B_{\rho r}(z)$ for $z \in \mathcal{Z}$.%
\item \label{it:tubel2} The tubes $L_k$ defined by $L_k = \mathrm{int}(\bigcup\cS^z_{\delta_3 \rho r / A^2}(L_k'))$ do not intersect the balls $B_{3\rho r/4}(z)$ for any $z\in \mathcal{Z}$.
\item \label{it:tubel3} For $k_1\neq k_2\in [\![1,N]\!]$, we have $\dist(L_{k_1},L_{k_2})>\rho r$.
\end{enumerate}

For $n\in \ZZ\setminus [\![1,N]\!]$, set $L_n=L_k$ where $k\in [\![1,N]\!]$ is such that $n-k$ is divisible by $N$. The tubes $L_k$ defined above shall be used throughout the paper. While these tubes provide a mechanism to go from $z_{k-1}+\rho r$ to $z_k-\rho r$, we will in fact need a way to go much closer to the points $z_{k-1},z_{k}$ since we are working with the events $F_{\rho r/A^2}(z)$ for $z\in \mathcal{Z}$ and thus the $\scX$s occur at a distance of order $\rho r/A^2$ around $z$ which is much smaller than $\rho r$.

\subsubsection*{The tubes $\cD_u, \cD_v$ connecting $\partial B_{\rho r}(z)$ to the $\scX$s}
\label{sec:paths-p_z-rho}
For points $z\in \cZ$, we shall heavily use the deterministic tubes $\cD_u,\cD_v$ that take us from the points $z-\rho r, z+\rho r$ to the coarse-grained $\scX$ coming from the event $F_{\rho r/A^2}(z)$ (recall Proposition \ref{prop:definition_of_F_r(z)}). %

\subsubsection*{The tubes $\widehat{L}_x$, $\widehat{L}_y$ going outward to $x,y\in \partial B_{2r}(0)$ and tubes from there to $3x/2,3y/2$. }
\label{sec:paths-hatl_x-hatl_y}
We shall, in fact, need certain extra paths in order to be able to attract $\Gamma_{\bz,\bw}$ to reach $B_{r}(0)$ in the first place, where we recall that all the balls $B_{\rho r/A^2}(z_k)$ are subsets of $B_{r}(0)$. For this, as in \cite[Section 5.4]{GM21}, we shall grow metric balls from $\bz,\bw$ until they hit $\partial B_{3r}(0)$ at certain random points $x',y'$ for $x',y'\in \partial B_{2r}(0)$ and then attempt to force the geodesic $\Gamma_{x',y'}$ to go via the $\scX$s corresponding to the points $z_k$. To do so, since the points $2x'/3,2y'/3\in \partial B_{2r}(0)$ above are random, we shall need to build highways corresponding to every deterministic pair $x,y\in \partial B_{2r}(0)$, and this is what we do now. This is done in two steps-- first, we construct tubes going from a point $3x/2$ to $x$ and then subsequently connect $x$ further inwards.

For the former, working with a parameter $\theta$, for any $x\in \partial B_{2r}(0)$, locally using $[x,(3/2-\theta)x]$ to denote the line segment joining $x$ and $(3/2-\theta)x$, we define
\begin{align*}
W_r^x = W_r^x(\theta) = \inte\left(\bigcup\cS^z_{\theta r}([x,(3/2-\theta) x])\right) \subseteq \mathbbm{A}_{r,3r}(0).
\end{align*}
Now, starting with a pair of points $x,y \in \partial B_{2r}(0)$ with $|x-y| \geq \delta r$, we define deterministic tubes $\widehat{L}_x',\widehat{L}_y'$ as follows. By possibly relabelling, we can assume that the counterclockwise arc of $\partial B_{2r}(0)$ from $x$ to $y$ is shorter than the clockwise arc.  Let $J_{x,y} \subseteq \partial B_r(0)$ be the clockwise arc from $\frac{x}{2}$ to $\frac{y}{2}$, so that $J_{x,y}$ has length at least $\frac{\delta r}{2}$.  Let $k_x,  k_y \in \ZZ$ be chosen so that $J_{x,y} \cap \mathcal{Z} = \{z_{k_x},z_{k_x+1},\cdots,z_{k_y-1},z_{k_y}\}$.  Let $\widehat{L}_x'$ (resp.\ $\widehat{L}_y'$) be a deterministic smooth path from $x$ to $z_{k_x}-\rho r$ (resp.\ from $z_{k_y} + \rho r$ to $y$ satisfying the following conditions:%
\begin{enumerate}[(a)]
\item \label{it:tubeht1} $\widehat{L}_x'$ (resp.\ $\widehat{L}_y'$) does not intersect any of the open balls $B_{\rho r}(z)$ for $z \in \mathcal{Z}$.
\item \label{it:tubeht2}The tube $\widehat{L}_x= \mathrm{int}(\bigcup\cS^z_{\delta_3 \rho r / A^2}(\widehat{L}_x'))$ (resp.\ $\widehat{L}_y= \mathrm{int}(\bigcup\cS^z_{\delta_3 \rho r / A^2}(\widehat{L}_y'))$) does not intersect any of the open balls $B_{3\rho r/4}(z)$ for $z\in \mathcal{Z}$. 
\item \label{it:tubeht3} Both $\widehat{L}_x$ and $\widehat{L}_y$ lie at Euclidean distance at least $\rho r$ from each other and from each $L_k$ for $k \in [\![k_x + 1 ,  k_y]\!]$.
\end{enumerate}

\subsubsection*{Defining the set $U_r^{x,y}$}
\label{sec:fatt-above-paths}
Recall the sets $O_r(z)$ from \eqref{eq:95}. We now define
\begin{align*}
  U_r^{x,y}:=\bigcup_{k = k_x}^{k_y} O_{\frac{\rho r}{A^2}}(z_k) \cup \widehat{L}_x\cup \widehat{L}_y\cup \bigcup_{k=k_x+1}^{k_y} L_k.
\end{align*}
Also,  since the $O_{\frac{\rho r}{A^2}}(z_k)$'s are connected,  we obtain that $U_r^{x,y}$ is connected and contains $x,y$.  Moreover, given $x,y$, the set $U_r^{x,y}$ is deterministic.  Furthermore combining with Lemma~\ref{lem:event_occurs_at_many_places},  we obtain that with probability at least $1-\frac{1-\bp}{4}$, the arc $J_{x,y}$ contains a point $z \in \mathcal{Z}$ such that $F_{\frac{\rho r}{A^2}}(z)$ occurs and $V_{\frac{\rho r}{A^2}}(z) \subseteq \overline{U_r^{x,y}}$.

\subsection{Constructing bump functions and the event $E_r(z)$}
\label{sec:bump_functions}
Recall from Section \ref{sec:introduction} that we want to construct an event $E_r(z)$ which occurs with high probability, admits $\scX$s of scale $\rho r/A^2$ near $z$, and is regular enough such that if the geodesic $\Gamma_{\bz,\bw,}$ were to venture $O(r)$ close to $z$, then with positive conditional probability, it would actually pass through one of the above $\scX$s-- that is, the event $\mathfrak{C}_{r}^{\bz,\bw}(z)$ would occur. We shall in fact just work with $z=0$-- and for other $z$, the event $E_r(z)$ is defined so that the family $\{E_{r}(z)\}_{z\in \CC,r>0}$ is translation and scale invariant.

The strategy is to proceed by an absolute continuity argument. Indeed, defining the Dirichlet inner product $(\phi,\psi)_\nabla =\int \nabla \phi \nabla \psi$, recall that for a zero boundary GFF $\mathtt{h}$ (see e.g.\ \cite[Theorem 1.29]{BP25} for a definition) on a domain $U\subseteq \CC$ and for a bump function $\phi$ with support inside $U$, the Radon-Nikodym derivative of the field $\mathtt{h}-\phi$ with respect to $\mathtt{h}$ is simply $\exp( -(\mathtt{h},\phi)_\nabla-(\phi,\phi)_\nabla/2)$ (see e.g.\ \cite[Proposition 1.51]{BP25})

Thus, the plan is to define an appropriate bump function $\phi_r^{x,y}$ for $x,y\in \partial B_{2r}(0)$ taking very large values and supported in a small neighbourhood of the set $U_r^{x,y}\cup W_r^x\cup W_r^y$ defined earlier, such that if we let $\bx',\by'$ be the (random) points at which growing metric balls around $\bz,\bw$ first hit $B_{3r}(0)$, then on a high probability event $E_r(0)$ on which abundant $\scX$s exist inside $B_{2r}(0)$ for the metric $D_h$, the following occurs:
\begin{itemize}
\item The above-mentioned $\scX$s for the metric $D_h$ are also $\scX$s for the field $D_{h-\phi}$.
\item The $D_{h-\phi}$ geodesic from $\bx'$ to $\by'$ stays close to the tube $U_r^{2\bx'/3,2\by'/3}\cup W_r^{2\bx'/3}\cup W_r^{2\by'/3}$ and in fact also passes through one of the $\scX$s above.
\end{itemize}
The above would provide a lower bound for the probability that the event $\mathfrak{C}_{r}^{\bz,\bw}(0)$ holds for the field $h-\phi$ given that the event $E_r(0)$ and the event $\{\Gamma_{\bz,\bw}\cap B_{2r}(0)\neq \emptyset\}$ occur. By additionally assuming certain regularity on the Dirichlet energy $(\phi,\phi)_{\nabla}$ along with the absolutely continuity discussed above, this can be translated into obtaining a corresponding lower bound for the conditional probability of the desired event for the field $h$ itself.
\begin{remark}
  \label{rem:comparison}
Both in this paper and in \cite{GM21}, bump functions are used to route $D_{h-\phi}$-geodesics through certain tubes, but the events $\mathfrak{C}_{r}^{\bz,\bw}(z)$ that these are used to force are very different in the two works. Here, bump functions are used to force geodesics to pass through a $\scX$, whereas in \cite{GM21} the task is to route the $D_{h-\phi}$-geodesic $\Gamma_{\bz,\bw}^\phi$ to go very close to two specific points $u,v$ satisfying a certain metric comparison inequality (see \cite[equation (5.1)]{GM21}). However, proving the latter \emph{does not} guarantee that the $D_{h-\phi}$-geodesic $\Gamma^\phi_{u,v}$ merges with $\Gamma_{\bz,\bw}^\phi$ (see \cite[Remark 5.1]{GM21} for a discussion). Indeed, guaranteeing merging statements of the above sort is the goal of the $\scX$-based construction in this paper.
\end{remark}
In view of the above remark, in our definitions of the bump functions, we have tried to use notation similar to \cite[Section 5.5.1]{GM21} to highlight the differences for a reader comparing the two works. We now give formal definitions of the bump functions that we shall use.

\subsubsection{Constructing bump functions.}

Consider small parameters $\delta ,\Delta, \zeta,\alpha \in (0,\frac{\delta_3\rho}{4 A^2}),  \theta \in (0,\zeta/100)$ and large parameters $B,M,\Lambda_0 >1$. The relevant bump functions and the event $E_r(0)$ will be defined in terms of these parameters; shortly, the parameters will be appropriately chosen to ensure that $E_r(0)$ has high probability and that certain natural geometric conditions hold (see Lemma \ref{lem:main_event_high_prob}).

For each $x,y \in \partial B_{2r}(0)$ with $|x-y| \geq \delta r$,  we choose in a deterministic manner depending only on $U_r^{x,y}$ but not on the particular values of $x$ and $y$,  a smooth,  compactly supported bump function $f_r^{x,y} : \mathbb{C} \to [0,1]$ which is identically equal to $1$ on the neighbourhood $U_r^{x,y}$ and is equal to $0$ outside of the neighbourhood $B_{\zeta r }(U_r^{x,y})$. Since each $U_r^{x,y}$ is the interior of a finite union of discrete squares,  there are at most a finite number of choices for $U_r^{x,y}$ as $x$ and $y$ vary which does not depend on $r$.  Therefore combining with the scale invariance of the Dirichlet energy\footnote{That is, if for functions $f,g\colon \CC\rightarrow \RR$, we have $g(x)=f(rx)$ for all $x$ and some $r>0$, then $(g,g)_{\nabla}=(f,f)_{\nabla}$.}, we obtain that the Dirichlet energy $(f_r^{x,y} ,  f_r^{x,y})_{\nabla}$ is bounded above by a constant depending only on $\delta_1 ,\delta_2, \delta_3, \rho ,  A ,  \delta$ and $\zeta$.  %

For all $x \in \partial B_{2r}(0)$, we let $g_r^{x} : \mathbb{C} \to [0,1]$ be a smooth compactly supported function which is identically equal to $1$ on $W_r^x$ and is identically equal to $0$ outside of the neighbourhood $B_{\theta^2 r}(W_r^x) \subseteq B_{3r}(0)$.  We now define two constants $K_f,K_g$ by
\begin{equation}
  \label{eq:1}
  K_f := \frac{1}{\xi} \log\left(\frac{100B}{\alpha \Delta}\right), K_g := K_f + \frac{1}{\xi} \log(M).
\end{equation}
Here, $K_f$ should be thought of as large with $K_g$ being even larger than $K_f$. For all $x,y \in \partial B_{2r}(0)$ with $|x-y| \geq \delta r$,  we define $\phi_r^{x,y}:= K_f f_r^{x,y} + K_g (g_r^x + g_r^y)$.  Since $f_r^{x,y},g_r^x,g_r^y$ are all supported on the annulus $\mathbb{A}_{r / 4 , 3r}(0)$,  the same is true for $\phi_r^{x,y}$.  We then set
\begin{align*}
\mathcal{G}_r:=\{\phi_r^{x,y} : x,y  \in \partial B_{2r}(0),  |x-y| \geq \delta r\} \cup \{\text{zero function}\}.
\end{align*}

\subsubsection{Defining the main event $E_r(0)$.}
\label{sec:maineventforce}
With the parameters $\delta,  \Delta,  B,  \zeta,\alpha,\theta,M,\Lambda_0$ as above,  we now define $E_r(0)$ to be the event on which the following are true.

\begin{enumerate}[(1)]
\item \label{it:abundance} Each arc of $\partial B_r(0)$ with Euclidean length at least $\frac{\delta r}{2}$ contains a point $z \in \mathcal{Z}$ such that $F_{\frac{\rho r}{A^2}}(z)$ occurs.
\item \label{it:lqg_euclidean_distance_small} 
For all $x,y \in \partial B_{2r}(0)$ with $|x-y| < \delta r$,  we have
\begin{align*}
D_h(x',y' ;  \A_{r,4r}(0)) \leq \Delta r^{\xi Q} e^{\xi h_r(0)} \leq D_h(\partial B_{2r}(0),\partial B_{3r}(0)),
\end{align*}
where $x' :=\frac{3}{2}x$ and $y':=\frac{3}{2}y$.
\item \label{it:lqg_distance_upper_bound}
For all $x,y \in \partial B_{2r}(0)$ with $|x-y| \geq \delta r$, we have
\begin{align*}
\sup_{w_1 ,  w_2 \in U_r^{x,y}} D_h\left(w_1 ,  w_2 ; U_r^{x,y}\right) \leq B r^{\xi Q} e^{\xi h_r(0)}.
\end{align*}
\item \label{it:lqg_distance_thin_ngbd}
For all $x,y \in \partial B_{2r}(0)$ with $|x-y| \geq \delta r$,  the $D_h$-length of every continuous path of Euclidean diameter at least $\frac{\delta_3 \rho r}{10^4 A^2}$ which is contained in $B_{2\zeta r}(\partial U_r^{x,y})$ is at least $100 B r^{\xi Q} e^{\xi h_r(0)}$.
\item \label{it:lqg_distance_lower_bound}
For all $z_1,z_2 \in \A_{\frac{r}{4},4r}(0)$ such that $|z_1-z_2| \geq \zeta r$,  we have that $D_h(z_1,z_2 ;  \A_{\frac{r}{4},4r}(0)) \geq \alpha r^{\xi Q} e^{\xi h_r(0)}$.
\item \label{it:lqg_distance_thin_line}
$D_h(\frac{3x}{2},(\frac{3}{2}-\theta)x ; \A_{r,4r}(0)) \leq e^{-\xi K_f} r^{\xi Q} e^{\xi h_r(0)}$ for all $x \in \partial B_{2r}(0)$.
\item \label{it:lqg_distance_thin_tube}
$\sup_{w_1 ,  w_2 \in W_r^x} D_h(w_1 ,  w_2 ; W_r^x) \leq M r^{\xi Q} e^{\xi h_r(0)}$.
\item \label{it:bump_function_dirichlet_energy}
  $|(h,\phi)_{\nabla}| + \frac{1}{2} (\phi,\phi)_{\nabla} \leq \Lambda_0$ for all $\phi \in \mathcal{G}_r$.
\end{enumerate}
As we shall see now, the probability of $E_r(0)$ can be made as high as desired by carefully tuning the parameters. We shall also need some basic geometric properties of the $\scX$s obtained when the event $E_r(0)$ occurs and these are also formalised in the following lemma. %
\begin{lemma}\label{lem:main_event_high_prob}
  We can choose the parameters $\delta,\Delta,\zeta,\alpha \in (0,\frac{\delta_3\rho}{4A^2}),\theta \in (0,\zeta/100),M,\Lambda_0 , B>1$ in a manner depending only on $b,\bp,\delta_1,\delta_2,\delta_3,A$ and $\rho$ in such a way that the following hold.
  \begin{enumerate}
  \item \label{it:lower_bound_on_prob_of_E_r(z)}
    We have $\mathbb{P}(E_r(0)) \geq \mathbbm{p}$ for all $r>0$.
  \item \label{it:measurable}The event $E_r(0)$ is measurable with respect to $\sigma( (h - h_{5r}(0))\lvert_{\A_{r/4,4r}(0)})$.
  \item \label{it:disconnection_of_tubes}
   For all $x,y\in \partial B_{2r}(0)$ with $|x-y|\geq \delta r$, and any $z\in J_{x,y}\cap \cZ$, the corresponding sets $\cA_{u_+},\cA_{u_-},\cM,\cA_{v_-},\cA_{v_+}$ and tubes $\cD_u,\cD_v$  have the following properties.  %
Set $\mathcal{C} = B_{2 \zeta r}(U_r^{x,y} \cup W_r^x \cup W_r^y)$ and suppose that $z = z_k$ for some $k \in [\![k_x ,  k_y]\!]$.  Then the following conditions hold.
\begin{enumerate}[(i)]
\item \label{it:intersecting_branches}
Any path in $\mathcal{C}$ connecting $B_{2 \zeta r}(L_k)$ to $B_{2\zeta r}(L_{k+1})$ has to intersect all of the sets $B_{2\zeta r}(\mathcal{D}_u),  B_{2\zeta r}(\mathcal{D}_v)$ and $B_{2\zeta r}(\mathcal{M} \setminus (S_u \cup S_v))$. 
\item \label{it:intersecting_square}
Any path in $\mathcal{C}$ from $B_{2\zeta r}(\mathcal{D}_u)$ to $B_{2\zeta  r}(\mathcal{A}_{u_+} \cup \mathcal{A}_{u_-})$ (resp.\  $B_{2\zeta r}(\mathcal{D}_v)$ to $B_{2\zeta r}(\mathcal{A}_{v_+} \cup \mathcal{A}_{v_-})$) has to intersect $B_{2\zeta r}(S_u)$ (resp.\  $B_{2\zeta r}(S_v)$).
\item \label{it:intersecting_arms}
Any path in $\mathcal{C}$ from $B_{2\zeta r}(\mathcal{D}_u)$ to $B_{2\zeta r}(V_{\rho r / A^2}(z) \setminus (\mathcal{A}_{u_+} \cup S_u \cup \mathcal{A}_{u_-}))$ (resp.\  $B_{2\zeta r}(\mathcal{D}_v)$ to $B_{2\zeta r}(V_{\rho r / A^2}(z) \setminus (\mathcal{A}_{v_+} \cup S_v \cup \mathcal{A}_{v_-}))$) has to intersect $B_{2\zeta r}(\mathcal{A}_{u_+} \cup \mathcal{A}_{u_-})$ (resp.\  $B_{2\zeta r}(\mathcal{A}_{v_+} \cup \mathcal{A}_{v_-})$). 
\item \label{it:intersecting_square_again}
  Any path in $\mathcal{C}$ from $B_{2\zeta r}(\mathcal{A}_{u_+} \cup \mathcal{A}_{u_-})$ to $B_{2\zeta r}(\mathcal{M} \setminus S_u)$ (resp.\  $B_{2\zeta r}(\mathcal{A}_{v_+} \cup \mathcal{A}_{v_-})$ to $B_{2\zeta r}(\mathcal{M} \setminus S_v)$) has to intersect $B_{2\zeta r}(S_u)$ (resp.\  $B_{2\zeta r}(S_v)$).
    \item \label{it:intersectmid} 
    Any path in $\mathcal{C}$ from $B_{2\zeta r}(\cA_{u_+}\cup \cA_{u_-})$ (resp.\ $B_{2\zeta r}(\cA_{v_+}\cup \cA_{v_-}))$ to $B_{2\zeta r}(\cA_{v_+}\cup S_v\cup \cA_{v_-})$ (resp.\ $B_{2\zeta r}(\cA_{u_+}\cup S_u\cup \cA_{u_-})$) necessarily intersects $B_{2\zeta r}(\cM\setminus S_u)$ (resp.\ $B_{2\zeta r}(\cM\setminus S_v)$).
    \item \label{it:intersect_endpoint}
    Any path in $\mathcal{C}$ from $B_{2\zeta r}(L_k)$ (resp.\ $B_{2\zeta r}(L_{k+1})$) to $B_{2\zeta r}(\mathcal{A}_{u_+} \cup \mathcal{A}_{u_-})$ (resp.\ $B_{2\zeta r}(\mathcal{A}_{v_+} \cup \mathcal{A}_{v_-})$) intersects $B_{2\zeta r}(\mathcal{D}_u)$ (resp.\ $B_{2\zeta r}(\mathcal{D}_v)$).
\end{enumerate}     
  \item \label{it:onlyway}
    We have $B_{\frac{b\rho r}{2A^2}}(\{u_-,u_+\})\cap B_{2\zeta r}(U_r^{x,y}\cup W_r^x \cup W_r^y) = B_{2\zeta r}(\cA_{u_-}\cup\cA_{u_+}) \cap B_{\frac{b\rho r}{2A^2}}(\{u_-,u_+\})$.
  \item \label{it:onlyway1}
    We have $B_{\frac{b\rho r}{2A^2}}(\{v_-,v_+\})\cap B_{2\zeta r}(U_r^{x,y}\cup W_r^x \cup W_r^y) = B_{2\zeta r}(\cA_{v_-}\cup\cA_{v_+}) \cap B_{\frac{b\rho r}{2A^2}}(\{v_-,v_+\})$. %
  \end{enumerate}
\end{lemma}
We note that in the above, the geometric condition \eqref{it:disconnection_of_tubes} is important as we require the relevant geodesics to pass through all the constructed tubes one-by-one without any ``short-circuiting''. The reason why this is true is that the parameters have been chosen earlier to ensure that different parts of Figure \ref{fig:detailedchi} do not venture close to each other. As we shall see, proving this formally just involves technical bookkeeping and is not particularly illuminating; however, we have chosen to include the proof for completeness.

\begin{proof}[Proof of Lemma \ref{lem:main_event_high_prob}]
  Item ~\eqref{it:lower_bound_on_prob_of_E_r(z)} was proved in \cite[Lemma~5.10]{GM21}, and we now give a brief overview. By just choosing the individual parameters to be appropriately large/small and utilising the basic properties of the LQG metric, one can obtain all the conditions except \eqref{it:lqg_distance_thin_ngbd}. Obtaining that condition \eqref{it:lqg_distance_thin_ngbd} also holds with high probability for $\zeta$ small involves an estimate arguing that paths restricted to pass through deterministic narrow tubes have rather large lengths (see \cite[Proposition 4.1]{DFGPS20}). We note that we have chosen the parameters $\delta,\Delta,\zeta,\alpha,\theta,M,\Lambda_0$, and $B$, in the order induced by items \eqref{it:abundance}-\eqref{it:bump_function_dirichlet_energy} in the definition of $E_r(0)$. Possibly by taking the parameters to be smaller, we can assume without loss of generality that $\delta,\Delta,\zeta,\alpha \in (0,\frac{\delta_3 \rho}{4A^2})$ and $\theta<\zeta/100$. Throughout the rest of the proof of the lemma, the choice of the above parameters will be fixed and it will not change.

  Regarding item \eqref{it:measurable}, that $E_r(0)\in \sigma( h\lvert_{\A_{r/4,4r}(0)})$ is easy to see by using the locality of the LQG metric along with the fact that the involved sets $U_r^{x,y}, W_r^x,W_r^y$ are all subsets of $\A_{r/4,4r}(0)$. To obtain the measurability of $E_r(0)$ with respect to the smaller $\sigma$-algebra $\sigma( (h - h_{5r}(0))\lvert_{\A_{r/4,4r}(0)})$, simply note by Weyl-scaling that $E_r(0)$ stays invariant on adding a global constant to the field. We note that the above discussion appears as \cite[Lemma 5.9]{GM21}.

Next we focus on proving that item \eqref{it:disconnection_of_tubes} holds.

\emph{Proof of item \eqref{it:intersecting_branches}.}
 Let $P$ be a path in $\mathcal{C}$ connecting $B_{2\zeta r}(L_k)$ to $B_{2\zeta r}(L_{k+1})$ and consider the sets
\begin{align*}
&\mathcal{B} = B_{2\zeta r}\left(W_r^x \cup \wh{L}_x \cup L_k \cup\bigcup_{m = k_x}^{k-1} (V_{\frac{\rho r}{A^2}}(z_m) \cup L_m \cup \cD_u\cup\cD_v) \right),\\
&\mathcal{D} = B_{2\zeta r}\left(W_r^y \cup \wh{L}_y  \cup \bigcup_{m = k+1}^{k_y} (V_{\frac{\rho r}{A^2}}(z_m) \cup L_m \cup \cD_u\cup \cD_v)  \right).\\
\end{align*}
Clearly, both $\mathcal{B}$ and $\mathcal{D}$ are connected; we now show that $\mathcal{B}\cap \mathcal{D}=\emptyset$. First, note that
\begin{align*}
V_{\frac{\rho r}{A^2}}(z_m) \cup \mathcal{D}_u \cup \mathcal{D}_v \subseteq B_{\rho r}(z_m) \text{ for all } m \in [\![k_x,k_y]\!],
\end{align*}
and the Euclidean distance between any two disjoint balls of the form $B_{\rho r}(z_m)$ for $m \in [\![k_x ,  k_y]\!]$ is at least $\rho r$. Further, recall from Section \ref{sec:manytubes} that
\begin{align*}
\dist(L_n ,  L_m) \geq \rho r
\end{align*}
for all $m,n \in [\![k_x ,  k_y]\!]$ such that $n \leq k < m$.  Moreover, recall that (item \eqref{it:tubeht2} in the definition of $\wh{L}_x$) the tubes $\wh{L}_x ,  \wh{L}_y$ do not intersect any of the balls $B_{3\rho r/4}(w)$ for $w \in \mathcal{Z}$ and both $\wh{L}_x$ and $\wh{L}_y$ lie at (item \eqref{it:tubeht3}) Euclidean distance at least $\rho r$ from each $L_m$ for $m \in [\![k_x +1 ,  k_y]\!]$.  Furthermore, note that 
\begin{align*}
\dist(W_r^x ,  W_r^y) \geq (\delta - 4\theta) r \geq \left(\delta - \frac{\delta_3 \rho}{A^2}\right) r \geq \frac{\delta r}{2} > \rho r,
\end{align*} 
where to obtain the last inequality, we note that $\rho\leq \delta/24$. Now, we also have
\begin{align*}
W_r^x \cup W_r^y \subseteq \mathbb{C} \setminus B_{(2 - 2\theta)r}(0)
\end{align*}
which implies that the Euclidean distance between $W_r^x$ (resp.\  $W_r^y$) and $\wh{L}_y \cup \left(\bigcup_{m=k_x}^{k_y} L_m \cup B_{\rho r}(z_m) \right)$ (resp.\  $\wh{L}_x \cup \left(\bigcup_{m=k_x}^{k_y} L_m \cup B_{\rho r}(z_m) \right)$) is at least $\rho r$.

It follows that $\mathcal{B} \cap \mathcal{D} = \emptyset$ and that
\begin{align*}
\mathcal{C} \cap B_{\rho r}(z_k) = B_{2\zeta r}\left(V_{\frac{\rho r}{A^2}}(z_k) \cup \mathcal{D}_u \cup \mathcal{D}_v \right) \cap B_{\rho r}(z_k)
\end{align*}
since $\zeta < \frac{\delta_3\rho}{4A^2}<\frac{\rho}{2}$ (recall from Proposition \ref{prop:definition_of_F_r(z)} that $\delta_3\in (0,1)$ and $A>2$). Note also that $V_{\frac{\rho r}{A^2}}(z_k) \subseteq B_{\frac{\rho r}{A}}(z_k)$ and by item \eqref{it:I3} in Proposition \ref{prop:definition_of_F_r(z)}, we have
\begin{align}
\dist\left(\mathcal{D}_u , \mathcal{D}_v \right) \geq \frac{50 \delta_2 \rho r}{A^2}.
\end{align}
Thus, we obtain that
\begin{align*}
B_{2\zeta r}(L_k) \cap B_{2\zeta r}\left(V_{\frac{\rho r}{A^2}}(z_k) \cup \mathcal{D}_v \right) = \emptyset
\end{align*}
which implies that 
\begin{align*}
\mathcal{B} \cap \left(\mathcal{D} \cup B_{2\zeta r}\left(V_{\frac{\rho r}{A^2}}(z_k) \cup \mathcal{D}_v \right)\right) = \emptyset.
\end{align*}
Therefore, if $P$ does not intersect $B_{2\zeta r}(\mathcal{D}_u)$,  we must have that 
\begin{align*}
P \subseteq \mathcal{B} \cup \mathcal{D} \cup B_{2\zeta r}\left(V_{\frac{\rho r}{A^2}}(z_k) \cup \mathcal{D}_v \right)
\end{align*}
and 
\begin{align*}
P \cap \mathcal{B} \neq \emptyset,  P \cap  \left(\mathcal{D} \cup B_{2\zeta r}\left(V_{\frac{\rho r}{A^2}}(z_k) \cup \mathcal{D}_v \right)\right) \neq \emptyset,
\end{align*}
which contradicts the connectedness of $P$.  It follows that $P$ must intersect $B_{2\zeta r}(\mathcal{D}_u)$ and a similar argument shows that $P$ must intersect $B_{2\zeta r}(\mathcal{D}_v)$ as well.

Finally, in order to complete the proof that item \eqref{it:intersecting_branches} holds, we show that $P$ intersects $B_{2\zeta r}\left(\mathcal{M} \setminus (S_u \cup S_v) \right)$.  Indeed, recall that condition \eqref{it:Mdisconn} in Proposition~\ref{prop:definition_of_F_r(z)} combined with items \eqref{it:I2}, \eqref{it:I3} imply that the sets
\begin{align*}
B_{2\zeta r}\left(\mathcal{D}_u \cup \mathcal{A}_{u_+} \cup \mathcal{A}_{u_-} \cup S_u\right) ,  B_{2\zeta r}\left(\mathcal{D}_v \cup \mathcal{A}_{v_+} \cup \mathcal{A}_{v_-} \cup S_v\right)
\end{align*}
are disjoint.  Suppose that $P$ does not intersect $B_{2\zeta r}\left(\mathcal{M} \setminus (S_u \cup S_v)\right)$.  Then we have that
\begin{align*}
P \subseteq \left(\mathcal{B} \cup B_{2\zeta r}\left(\mathcal{D}_u \cup \mathcal{A}_{u_+} \cup \mathcal{A}_{u_-} \cup S_u\right) \right) \cup \left(\mathcal{D} \cup B_{2\zeta r}\left(\mathcal{D}_v \cup \mathcal{A}_{v_+} \cup \mathcal{A}_{v_-} \cup S_v\right)\right)
\end{align*}
and
\begin{align*}
&P \cap \left(\mathcal{B} \cup B_{2\zeta r}\left(\mathcal{D}_u \cup \mathcal{A}_{u_+} \cup \mathcal{A}_{u_-} \cup S_u\right) \right) \neq \emptyset\\
&P \cap \left(\mathcal{D} \cup B_{2\zeta r}\left(\mathcal{D}_v \cup \mathcal{A}_{v_+} \cup \mathcal{A}_{v_-} \cup S_v\right)\right) \neq \emptyset.
\end{align*}
But this contradicts the fact that $P$ is a connected set.  Therefore, $P$ must intersect $B_{2\zeta r}\left(\mathcal{M} \setminus (S_u \cup S_v)\right)$.  This completes the proof that item \eqref{it:intersecting_branches} holds.

\emph{Proof of item \eqref{it:intersecting_square}.}
Next we focus on proving that item \eqref{it:intersecting_square} holds.  Let $P$ be a path in $\mathcal{C}$ from $B_{2\zeta r}(\mathcal{D}_u)$ to $B_{2\zeta r}(\mathcal{A}_{u_+} \cup \mathcal{A}_{u_-})$.  Note that we have already shown in the proof of item \eqref{it:intersecting_branches} that 
\begin{align*}
    \mathcal{B} \cap (\mathcal{D} \cup B_{2\zeta r}(V_{\rho r / A^2}(z_k) \cup \mathcal{D}_v)) = \emptyset
\end{align*}
and a similar argument shows that
\begin{align*}
    \mathcal{D} \cap (\mathcal{B} \cup B_{2\zeta r}(V_{\rho r / A^2}(z_k) \cup \mathcal{D}_u)) = \emptyset.
\end{align*}
In particular, we have that
\begin{align*}
B_{2\zeta r}(\mathcal{D}_u) \cap \mathcal{D} = \emptyset.
\end{align*}
Also item \eqref{it:I2} in Proposition~\ref{prop:definition_of_F_r(z)} implies that
\begin{align*}
B_{2\zeta r}(\mathcal{D}_u) \cap B_{2\zeta r}\left(\left(\mathcal{A}_{u_+} \cup \mathcal{A}_{u_-} \cup \mathcal{A}_{v_+} \cup \mathcal{A}_{v_-} \cup \mathcal{M} \cup S_v \cup \mathcal{D}_v \right) \setminus S_u \right) = \emptyset
\end{align*}
and hence 
\begin{align*}
B_{2\zeta r}(\mathcal{D}_u) \cap \left(\mathcal{D} \cup B_{2\zeta r}\left(\left(\mathcal{A}_{u_+} \cup \mathcal{A}_{u_-} \cup \mathcal{A}_{v_+} \cup \mathcal{A}_{v_-} \cup \mathcal{M} \cup S_v \cup \mathcal{D}_v \right) \setminus S_u \right)\right) = \emptyset.
\end{align*}
This implies that
\begin{align*}
\left(\mathcal{B} \cup B_{2\zeta r}(\mathcal{D}_u) \right) \cap  \left(\mathcal{D} \cup B_{2\zeta r}\left(\left(\mathcal{A}_{u_+} \cup \mathcal{A}_{u_-} \cup \mathcal{A}_{v_+} \cup \mathcal{A}_{v_-} \cup \mathcal{M} \cup S_v \cup \mathcal{D}_v \right) \setminus S_u \right)\right) = \emptyset.
\end{align*}
It follows that the sets 
\begin{align*}
\mathcal{B} \cup B_{2\zeta r}(\mathcal{D}_u),  B_{2\zeta r}(\mathcal{A}_{u_+} \setminus S_u),B_{2\zeta r}(\mathcal{A}_{u_-} \setminus S_u),  B_{2\zeta r}\left(\left(\mathcal{M} \cup \mathcal{D}_v \cup \mathcal{A}_{v_+} \cup \mathcal{A}_{v_-} \cup S_v \right) \setminus S_u\right), \mathcal{D} \cup B_{2\zeta r}(\mathcal{D}_v)
\end{align*}
are open,  connected and pairwise disjoint,  and they contain $P$ if $P \cap B_{2\zeta r}(S_u) = \emptyset$.  But this contradicts the connectedness of $P$ and so we must have that $P \cap B_{2\zeta r}(S_u) \neq \emptyset$.  Similarly we can show that if $Q$ is a path in $\mathcal{C}$ from $B_{2\zeta r}(S_v)$ to $B_{2\zeta r}(\mathcal{A}_{v_+} \cup \mathcal{A}_{v_-})$,  then $Q$ has to intersect $B_{2\zeta r}(S_v)$.  This completes the proof that item \eqref{it:intersecting_square} holds.

\emph{Proof of item \eqref{it:intersecting_arms}.}
Let $P$ be a path in $\mathcal{C}$ connecting $B_{2\zeta r}(\mathcal{D}_u)$ with $B_{2\zeta r}(V_{\rho r / A^2}(z_k) \setminus (\mathcal{A}_{u_+} \cup S_u \cup \mathcal{A}_{u_-}))$. Suppose that $P$ does not intersect $B_{2\zeta r}(\mathcal{A}_{u_+} \cup \mathcal{A}_{u_-})$.  First we will show by arguing by contradiction that $P \cap B_{2\zeta r}(S_u) \neq \emptyset$.  Indeed suppose that $P \cap B_{2\zeta r}(S_u) = \emptyset$.  Then we have that
\begin{align*}
P \subseteq \left(\mathcal{B} \cup B_{2\zeta r}(\mathcal{D}_u)\right) \cup \left(\mathcal{D} \cup B_{2\zeta r}\left(\left(\mathcal{M} \cup \mathcal{A}_{v_+} \cup \mathcal{A}_{v_-} \cup S_v \cup \mathcal{D}_v \right) \setminus S_u \right)\right)
\end{align*}
and we have already shown in the proof of item \eqref{it:intersecting_square} that the sets 
\begin{align*}
\mathcal{B} \cup B_{2\zeta r}(\mathcal{D}_u),  \mathcal{D} \cup B_{2\zeta r}\left(\left(\mathcal{M} \cup \mathcal{A}_{v_+} \cup \mathcal{A}_{v_-} \cup S_v \cup \mathcal{D}_v \right) \setminus S_u \right)
\end{align*}
are connected and disjoint.  But since $P$ intersects both of them,  this contradicts the connectedness of $P$.

Next, we claim that the path $P$ enters $B_{2\zeta r}(S_u)$ via $B_{2\zeta r}(\mathcal{D}_u)$. Indeed, suppose that this does not hold and note that the sets $\overline{B_{2\zeta r}(S_u)}$ and $\overline{B_{2\zeta r}(\mathcal{A}_{v_+} \cup \mathcal{A}_{v_-} \cup S_v \cup \mathcal{D}_v)}$ are disjoint which implies that the path $P$ can enter $B_{2\zeta r}(S_u)$ only via the set
\begin{align*}
    B_{2\zeta r}(\mathcal{D}_u) \cup B_{2\zeta r}(\mathcal{A}_{u_+} \cup \mathcal{A}_{u_-}) \cup B_{2\zeta r}(\mathcal{M}).
\end{align*}
Since we have assumed that $P$ does not intersect $B_{2\zeta r}(\mathcal{A}_{u_+} \cup \mathcal{A}_{u_-})$ and it does not enter $B_{2\zeta r}(S_u)$ via $B_{2\zeta r}(\mathcal{D}_u)$, we obtain that $P$ enters $B_{2\zeta r}(S_u)$ via $B_{2\zeta r}(\mathcal{M})$. Set
\begin{align*}
    t = \inf\{s > 0 : P(s) \in B_{2\zeta r}(S_u)\}
\end{align*}
and note that 
\begin{align*}
    P([0,t)) \cap B_{2\zeta r}(\mathcal{M} \setminus S_u) \neq \emptyset. 
\end{align*}
In particular, we have that
\begin{align*}
P([0,t)) \subseteq \left(\mathcal{B} \cup B_{2\zeta r}(\mathcal{D}_u)\right) \cup \left(\mathcal{D} \cup B_{2\zeta r}\left(\left(\mathcal{M} \cup \mathcal{A}_{v_+} \cup \mathcal{A}_{v_-} \cup S_v \cup \mathcal{D}_v \right) \setminus S_u \right)\right)
\end{align*}
and $P|_{[0,t)}$ intersects both $\mathcal{B} \cup B_{2\zeta r}(\mathcal{D}_u)$ and $\mathcal{D} \cup B_{2\zeta r}((\mathcal{M} \cup \mathcal{A}_{v_+} \cup \mathcal{A}_{v_-} \cup S_v \cup \mathcal{D}_v) \setminus S_u)$. But this contradicts the connectedness of $P|_{[0,t)}$ as above. It follows that $P$ enters $B_{2\zeta r}(S_u)$ via $B_{2\zeta r}(\mathcal{D}_u)$ and so
\begin{align*}
    P \cap B_{2\zeta r}(\mathcal{D}_u) \cap B_{2\zeta r}(S_u) \neq \emptyset.
\end{align*}

Recall that by item \eqref{it:obj4} and item \eqref{it:I1} in Proposition \ref{prop:definition_of_F_r(z)}, we know that the path $Q_u$ ends at $x_u \in I_u$ and  $(Q_u\setminus \{x_u\}) \cap S_u = \emptyset$ and that the tube $\cD_u$ (a fattened version of $Q_u$) satisfies
\begin{align*}
    B_{\delta_3 \rho r/A^2}(\cD_u) \cap B_{\delta_3 \rho r / A^2}(S_u) \subseteq B_{\delta_2 \rho r / A^2}(I_u).
\end{align*}
It follows that $P$ connects $B_{\delta_3 \rho r / A^2}(I_u)$ with $B_{2\zeta r}((\mathcal{M} \setminus S_u) \cup S_v \cup \mathcal{A}_{v_+} \cup \mathcal{A}_{v_-})$. 
Therefore combining with Lemma~\ref{lem:11} and since $\zeta \in (0,\frac{\delta_3 \rho}{4 A^2})$,  we obtain that $P$ intersects $B_{2\zeta r}(\mathcal{A}_{u_+} \cup \mathcal{A}_{u_-})$ and that contradicts our original assumption.  It follows that $P$ intersects $B_{2\zeta r}(\mathcal{A}_{u_+} \cup \mathcal{A}_{u_-})$.  A similar argument shows that any path in $\mathcal{C}$ connecting $B_{2\zeta r}(\mathcal{D}_v)$ with $B_{2\zeta r}(V_{\rho r / A^2}(z_k) \setminus (\mathcal{A}_{v_+} \cup S_v \cup \mathcal{A}_{v_-}))$ has to intersect $B_{2\zeta r}(\mathcal{A}_{v_+} \cup \mathcal{A}_{v_-})$.  This completes the proof of item \eqref{it:intersecting_arms}.

\emph{Proof of item \eqref{it:intersecting_square_again}.}
Now we prove that item \eqref{it:intersecting_square_again} holds.  Let $P$ be a path in $\mathcal{C}$ connecting $B_{2\zeta r}(\mathcal{A}_{u_+} \cup \mathcal{A}_{u_-})$ with $B_{2\zeta r}(\mathcal{M} \setminus S_u)$.  Suppose that $P$ does not intersect $B_{2\zeta r}(S_u)$.  Then, we must have
\begin{align*}
P \subseteq \mathcal{B} &\cup B_{2\zeta r}(\mathcal{D}_u) \cup B_{2\zeta r}(\mathcal{A}_{u_+} \setminus S_u) \cup B_{2\zeta r}(\mathcal{A}_{u_-} \setminus S_u)\\
&~~~\cup \left(\mathcal{D} \cup B_{2\zeta r}\left((\mathcal{M} \setminus S_u) \cup \mathcal{A}_{v_+} \cup \mathcal{A}_{v_-} \cup S_v\cup \mathcal{D}_v \right)\right).
\end{align*}

Recall that we have already shown in the proof of item \eqref{it:intersecting_branches} that
\begin{align*}
\mathcal{B} \cap \left(\mathcal{D} \cup B_{2\zeta r}(V_{\frac{\rho r}{A^2}}(z_k) \cup \mathcal{D}_v)\right) = \emptyset,
\end{align*}
and we have also shown in the proof of item \eqref{it:intersecting_square} that
\begin{align*}
    B_{2\zeta r}(\mathcal{D}_u) \cap \mathcal{D} = \emptyset.
\end{align*}
Since $\zeta<\frac{\delta_3\rho}{4A^2}$, on combining with conditions \eqref{it:geodesics_separated}, \eqref{it:Mdisconn}, \eqref{it:I2} and \eqref{it:I3} in Proposition~\ref{prop:definition_of_F_r(z)},  we obtain that the sets 
\begin{align*}
\mathcal{B} \cup B_{2\zeta r}(\mathcal{D}_u),  B_{2\zeta r}(\mathcal{A}_{u_+} \setminus S_u),   B_{2\zeta r}(\mathcal{A}_{u_-} \setminus S_u),  \mathcal{D} \cup B_{2\zeta r}\left((\mathcal{M} \setminus S_u) \cup \mathcal{A}_{v_+} \cup \mathcal{A}_{v_-} \cup \mathcal{D}_v \cup S_v \right)
\end{align*}
are pairwise disjoint. But this contradicts the connectedness of $P$.  Hence, we obtain that $P$ intersects $B_{2\zeta r}(S_u)$.  Similarly, if $Q$ is a path in $\mathcal{C}$ connecting $B_{2\zeta r}(\mathcal{A}_{v_+} \cup \mathcal{A}_{v_-})$ with $B_{2\zeta r}(\mathcal{M} \setminus S_v)$,  then $Q$ must intersect $B_{2\zeta r}(S_v)$.  

\emph{Proof of item \eqref{it:intersectmid}.} Next we prove that item \eqref{it:intersectmid} holds. Let $P$ be a path in $\mathcal{C}$ from $B_{2\zeta r}(\mathcal{A}_{u_+} \cup \mathcal{A}_{u_-})$ to $B_{2\zeta r}(\mathcal{A}_{v_+} \cup S_v \cup \mathcal{A}_{v_-})$. Suppose that $P$ does not intersect $B_{2\zeta r}(\mathcal{M} \setminus S_u)$. Recall that
\begin{align*}
    \mathcal{C} \subseteq \mathcal{B} \cup \mathcal{D} \cup B_{2\zeta r}(V_{\rho r / A^2}(z_k) \cup \mathcal{D}_u \cup \mathcal{D}_v).
\end{align*}
Thus, we obtain that
\begin{align*}
    P \subseteq (\mathcal{B} \cup B_{2\zeta r}(\mathcal{A}_{u_+} \cup S_u \cup \mathcal{A}_{u_-} \cup \mathcal{D}_u)) \cup (\mathcal{D} \cup B_{2\zeta r}(\mathcal{A}_{v_+} \cup S_v \cup \mathcal{A}_{v_-} \cup \mathcal{D}_v)).
\end{align*}
Moreover, we have shown in the proof of item \eqref{it:intersecting_branches} that the sets 
\begin{align*}
    B_{2\zeta r}(\mathcal{D}_u \cup \mathcal{A}_{u_+} \cup S_u \cup \mathcal{A}_{u_-}), B_{2\zeta r}(\mathcal{D}_v \cup \mathcal{A}_{v_+} \cup S_v \cup \mathcal{A}_{v_-})
\end{align*}
are disjoint and that
\begin{align*}
    \mathcal{B} \cap (\mathcal{D} \cup B_{2\zeta r}(V_{\rho r / A^2}(z_k) \cup \mathcal{D}_v)) = \emptyset.
\end{align*}
A similar argument shows that
\begin{align*}
    \mathcal{D} \cap (\mathcal{B} \cup B_{2\zeta r}(V_{\rho r/ A^2}(z_k)\cup \mathcal{D}_u)) = \emptyset.
\end{align*}
It follows that the sets
\begin{align*}
    \mathcal{B} \cup B_{2\zeta r}(\mathcal{D}_u \cup \mathcal{A}_{u_+} \cup S_u \cup \mathcal{A}_{u_-}), \mathcal{D} \cup B_{2\zeta r}(\mathcal{D}_v \cup \mathcal{A}_{v_+} \cup S_v \cup \mathcal{A}_{v_-})
\end{align*}
are disjoint and the path $P$ intersects both of them. But since $P$ is contained in their union, this contradicts the connectedness of $P$. Therefore, we obtain that $P$ intersects $B_{2\zeta r}(\mathcal{M} \setminus S_u)$. A similar argument shows that any path in $\mathcal{C}$ from $B_{2\zeta r}(\mathcal{A}_{v_+} \cup \mathcal{A}_{v_-})$ to $B_{2\zeta r}(\mathcal{A}_{u_+} \cup S_u \cup \mathcal{A}_{u_-})$ has to intersect $B_{2\zeta r}(\mathcal{M} \setminus S_v)$. 

\emph{Proof of item \eqref{it:intersect_endpoint}}. Now we complete the proof of item \eqref{it:disconnection_of_tubes} by proving that item \eqref{it:intersect_endpoint} holds. Let $P$ be a path in $\mathcal{C}$ from $B_{2\zeta r}(L_k)$ to $B_{2\zeta r}(\mathcal{A}_{u_+} \cup \mathcal{A}_{u_-})$. Suppose that $P$ does not intersect $B_{2\zeta r}(\mathcal{D}_u)$. Recall that we have already shown in the proof of item \eqref{it:intersecting_branches} that
\begin{align*}
    \mathcal{C} \cap B_{\rho r}(z_k) = B_{2\zeta r}(V_{\rho r / A^2}(z_k) \cup \mathcal{D}_u \cup \mathcal{D}_v) \cap B_{\rho r}(z_k).
\end{align*}
Further, since,
\begin{align*}
    \dist(V_{\rho r / A^2}(z_k) , \partial B_{\rho r}(z_k)) > 0,
\end{align*}
we also have
\begin{align*}
    \mathcal{C} \cap \partial B_{\rho r}(z_k) \subseteq B_{2\zeta r}(\mathcal{D}_u) \cup B_{2\zeta r}(\mathcal{D}_v).
\end{align*}
Since $P$ intersects $\mathcal{A}_{u_+} \cup \mathcal{A}_{u_-} \subseteq B_{\rho r}(z_k)$, we obtain that $P$ has to intersect $\partial B_{\rho r}(z_k)$. We parameterize $P$ by the unit interval $[0,1]$ so that $P(0) \in B_{2\zeta r}(L_k)$, and let $t \in [0,1]$ be the first time that $P$ intersects $\partial B_{\rho r}(z_k)$. Since $P$ does not intersect $B_{2\zeta r}(\mathcal{D}_u)$, we have $P(t) \in B_{2\zeta r}(\mathcal{D}_v)$. Moreover, it holds that
\begin{align*}
    \mathcal{C} = \mathcal{B} \cup \mathcal{D} \cup B_{2\zeta r}(V_{\rho r / A^2}(z_k) \cup \mathcal{D}_u \cup \mathcal{D}_v)
\end{align*}
and so we have that $P([0,t]) \subseteq \mathcal{B} \cup \mathcal{D}$ and that $P([0,t])$ intersects both $\mathcal{B}$ and $\mathcal{D}$. However, we have already shown in the proof of item \eqref{it:intersecting_branches} that $\mathcal{B} \cap \mathcal{D} = \emptyset$, and thus this contradicts the connectedness of $P|_{[0,t]}$. Therefore, we obtain that $P$ has to intersect $B_{2\zeta r}(\mathcal{D}_u)$.

A similar argument shows that any path in $\mathcal{C}$ from $B_{2\zeta r}(L_{k+1})$ to $B_{2\zeta r}(\mathcal{A}_{v_+} \cup \mathcal{A}_{v_-})$ has to intersect $B_{2\zeta r}(\mathcal{D}_v)$. This completes the proof of item \eqref{it:intersect_endpoint} and hence we have completed the proof of item \eqref{it:disconnection_of_tubes}.

It remains to show that items \eqref{it:onlyway} and \eqref{it:onlyway1} hold. %
First we note that conditions \eqref{it:geodesics_separated2}, \eqref{it:I4} in Proposition~\ref{prop:definition_of_F_r(z)} imply that %
\begin{align*}
B_{\frac{b \rho r}{2A^2}}(\{u_+ ,  u_-\}) \cap B_{2\zeta r}\left(\mathcal{D}_u \cup \mathcal{M} \cup \mathcal{A}_{v_+} \cup \mathcal{A}_{v_-}
\cup \mathcal{D}_v \right) = \emptyset.
\end{align*}
Also condition \eqref{it:geodloc} in Proposition~\ref{prop:definition_of_F_r(z)} implies that $\{u_+ ,  u_-\} \subseteq B_{\frac{\rho r}{4A}}(z_k)$ and hence
\begin{align*}
B_{\frac{b \rho r}{2A^2}}(\{u_+ ,  u_-\}) \subseteq B_{\frac{3\rho r}{4A}}(z_k).
\end{align*}
Moreover, we know that (see items \eqref{it:tubeht2}, \eqref{it:tubel2}) that none of the objects 
\begin{displaymath}
W_r^x,W_r^y,\wh{L}_x,\wh{L}_y,
\{L_m\}_{m \in [\![k_x,k_y]\!]},  \{B_{\rho r}(z_n)\}_{n \in [\![k_x,k_y]\!] \setminus \{k\}}
\end{displaymath}intersect $B_{3\rho r/4}(z_k)$ which implies that
\begin{align*}
B_{2\zeta r}\left(W_r^x \cup W_r^y \cup \wh{L}_x \cup \wh{L}_y \cup\bigcup_{m=k_x}^{k_y} L_m  \cup \bigcup_{m \neq k} B_{\rho r}(z_m) \right) \cap B_{\frac{3\rho r}{4A}}(z_k) = \emptyset
\end{align*}
since $\zeta < \frac{\delta_3 \rho}{4A^2}<\frac{\rho}{2A^2}$ (recall from Proposition \ref{prop:definition_of_F_r(z)} that $\delta_3<1$). Hence, we obtain that 
\begin{align*}
\mathcal{C} \cap B_{\frac{b\rho r}{2A^2}}(\{u_+ ,  u_-\}) = B_{2\zeta r}(\mathcal{A}_{u_+} \cup \mathcal{A}_{u_-}) \cap B_{\frac{b\rho r}{2A^2}}(\{u_+ ,  u_-\}).
\end{align*}
A similar argument shows that
\begin{align*}
\mathcal{C} \cap B_{\frac{b\rho r}{2A^2}}(\{v_+ ,  v_-\}) = B_{2\zeta r}(\mathcal{A}_{v_+} \cup \mathcal{A}_{v_-}) \cap B_{\frac{b\rho r}{2A^2}}(\{v_+ ,  v_-\}).
\end{align*}
This completes the proof of the lemma.
\end{proof}

\subsection{Subtracting a bump function from the field}\label{sec:new_field}

Next, we fix distinct points $\bz,\bw \in \mathbb{C} \setminus B_{4r}(0)$ and let $P = P^{\bz,\bw}$ be the a.s.\ unique $D_h$-geodesic from $\bz$ to $\bw$.  Let $\sigma_r$ (resp.\ $\widehat{\sigma}_r$) be the smallest $s>0$ for which the $D_h$-metric ball $\cB_s(\bz)$ (resp.\ $\cB_s(\bw)$) intersects $\overline{B_{3r}(0)}$.  Also let $\bx'$ (resp.\ $\by'$) be a point contained in $\partial B_{3r}(0) \cap \cB_{\sigma_r}(\bz)$ (resp.\ $\partial B_{3r}(0) \cap \cB_{\widehat{\sigma}_r}(\bw )$) chosen in a pre-fixed manner depending only on $\sigma(h\lvert_{\mathbb{C}\setminus B_{3r}(0)})$. In fact, as also mentioned in \cite[Footnote 7]{GM21}, the points $\bx',\by'$ should be a.s.\ unique and this should be possible to prove by an adaptation of \cite{MQ18}; however, the uniqueness of $\bx',\by'$ is not important for us.

 Now, set $\bx = \frac{2}{3} \bx'$ and $\by = \frac{2}{3} \by'$.  We consider the bump function $\phi$ defined by $\phi = \phi_r^{\bx,\by}$ if $|\bx-\by| \geq \delta r$ and $\phi = 0$ otherwise,  and note that the random points $\bx ,  \by$ are measurable with respect to $\sigma(h|_{\mathbb{C} \setminus B_{3r}(0)})$.  Then $\phi$ is determined by $\bx,\by$ and hence by $h|_{\mathbb{C} \setminus B_{3r}(0)}$.  We also set $U:=U_r^{\bx,\by}$ and $\mathcal{W}:=B_{\theta^2 r}(W_r^{\bx}) \cup B_{\theta^2 r}(W_r^{\by})$. Let $P^\phi=\Gamma_{\bz,\bw}^\phi$ be the a.s.\ unique (since $\bz,\bw$ are fixed) $D_{h-\phi}$-geodesic from $\bz$ to $\bw$ and let $\overline{P}^\phi = P^\phi \setminus (B_{\sigma_r}(\bz) \cup B_{\widehat{\sigma}_r}(\bw))$.  Note that the definitions of $\sigma_r,\wh{\sigma}_r,  \mathcal{B}_{\sigma_r}(\bz)$ and $\mathcal{B}_{\wh{\sigma}_r}(\bw)$ remain unaffected if we replace $h$ by $h-\phi$ since $\phi$ is supported on $B_{3r}(0)$.

\begin{lemma}
  \label{lem:Er-conds}
By appropriately choosing the parameters in the definition for the event $E_r(0)$, on the event $E_r(0)\cap \{\Gamma_{\bz,\bw}\cap B_{2r}(0)\neq \emptyset\}$, we have the following.
  \begin{enumerate}[(a)]
  \item \label{it:er1} $|\bx-\by| \geq \delta r$. 
  \item \label{it:er2} $D_{h-\phi}(\bx',\by') \leq e^{-\xi K_f} (B+4) r^{\xi Q} e^{\xi h_r(0)}$. 
  \item  \label{it:er3} $\overline{P}^\phi \subseteq B_{2\zeta r}(U \cup \mathcal{W})$.
  \item \label{it:er4} There is no segment of $\overline{P}^\phi$ of Euclidean diameter at least $\frac{\delta_3 \rho r}{100A^2}$ which is contained in $B_{2\zeta r}(\partial U) \setminus \mathcal{W}$ 
  \end{enumerate}
  
\end{lemma}
\begin{proof}
All the conditions are in fact shown in \cite[Section 5.6]{GM21} and we now provide the corresponding references. In order to be self-contained, we also sketch the proofs.

  Condition \eqref{it:er1} follows by using item \eqref{it:lqg_euclidean_distance_small} in the definition of $E_r(0)$ along with the fact that we are assuming that $\{\Gamma_{\bz,\bw}\cap B_{2r}(0)\neq \emptyset\}$-- this is because $\Gamma_{\bz,\bw}$ must then cross $\A_{2r,3r}(0)$ twice; the details are discussed in \cite[Lemma~5.12]{GM21}.

 For condition \eqref{it:er2}, the details are provided in \cite[Lemma~5.13]{GM21}. To summarise, using the item \eqref{it:lqg_distance_thin_line} in the definition of $E_r(0)$ and that $\phi\geq 0$, we obtain that $D_{h-\phi}(\bx',W_{r}^{\bx}),D_{h-\phi}(\by',W_{r}^{\by})$ are both upper bounded by $e^{-\xi K_f}r^{\xi Q}e^{\xi h_r(0)}$. Further, item \eqref{it:lqg_distance_thin_tube} in the definition of $E_r(0)$ and the fact that $\phi\geq K_g$ on $W_r^\bx,W_r^\by$ yields that the $D_{h-\phi}$ diameters of $W_r^{\bx}$ and $W_r^{\by}$ are at most $e^{-\xi K_g} M r^{\xi Q}e^{\xi h_r(0)}\leq e^{-\xi K_f} r^{\xi Q} e^{\xi h_r(0)}$. Finally, by item \eqref{it:lqg_distance_upper_bound} in the definition of $E_r(0)$ and the fact that $\phi\geq K_f$ on $U$, the $D_{h-\phi}$ internal diameter of $U$ is at most $e^{-\xi K_f} Br^{\xi Q}e^{\xi h_r(0)}$.

 Conditions \eqref{it:er3} and \eqref{it:er4} are handled in \cite[Lemma 5.14]{GM21}. In short, for \eqref{it:er4}, item \eqref{it:lqg_distance_thin_ngbd} in the definition of $E_r(0)$ and Weyl scaling imply that the $D_{h-\phi}$ length of any path with Euclidean diameter at least $\frac{\delta_3 \rho r}{10^4 A^2}$ which is contained  in $B_{2\zeta r}(\partial U)\setminus \cW$ is at least $100 e^{-\xi K_f} B r^{\xi Q} e^{\xi h_r(0)}$. By \eqref{it:er2} and since $B>1$ (see Lemma \ref{lem:main_event_high_prob}), such a segment cannot be a segment of $\overline{P}^\phi$.

 To obtain \eqref{it:er3}, it is used that as a consequence of item \eqref{it:lqg_distance_lower_bound} in the definition of $E_r(0)$ and the fact that $\phi$ is supported on $B_{\zeta r}(U)\cup \cW$, any path in $\A_{r/4,4r}(0)\setminus (B_{\zeta r}(U)\cup \cW)$ with Euclidean diameter at least $\zeta r$ has $D_{h-\phi}$ length at least $\alpha r^{\xi Q} e^{\xi h_r(0)}$ which is strictly larger than $e^{-\xi K_f} (B+4) r^{\xi Q} e^{\xi h_r(0)}$. Thus, there cannot be a segment of $\overline{P}^{\phi}$ with Euclidean length at least $\zeta r$ which lies strictly in $\A_{r/4,4r}(0)\setminus (B_{\zeta r}(U)\cup \cW)$.  Moreover condition~\eqref{it:lqg_euclidean_distance_small} in the definition of $E_r(0)$ combined with \eqref{it:er1} imply that the $D_h$-distance between $\mathcal{B}_{\sigma_r}(\bz)$ and $\mathcal{B}_{\wh{\sigma}_r}(\bw)$ is at least $2 \Delta r^{\xi Q} e^{\xi h_r(0)}$ which is larger than $e^{-\xi K_f} (B+4) r^{\xi Q} e^{\xi h_r(0)}$ by the definition of $K_f$.  If $\overline{P}^{\phi}$ did not enter the support $B_{\zeta r}(U) \cup \mathcal{W}$ of $\phi$,  then the $D_{h-\phi}$-length of $\overline{P}^{\phi}$ would be the same as its $D_h$-length,  which must be at least $2 \Delta r^{\xi Q} e^{\xi h_r(0)}$.  Hence, \eqref{it:er2} implies that $\overline{P}^{\phi}$ must enter $B_{\zeta r}(U) \cup \mathcal{W}$.  Therefore if $\overline{P}^{\phi} \not \subset B_{2\zeta r}(U \cup \mathcal{W})$,  then $\overline{P}^{\phi}$ would have to contain a segment in $\mathbb{A}_{r/4 ,  4r}(0) \setminus (B_{\zeta r}(U) \cup \mathcal{W})$ with Euclidean length at least $\zeta r$ but that yields a contradiction.  This completes the proof of the lemma. 
\end{proof}

Next we assume without loss of generality that the counterclockwise arc of $\partial B_{2r}(0)$ from $\bx$ to $\by$ is shorter than the clockwise arc.  Let $J_{\bx,\by}\subseteq \partial B_{2r}(0)$ be the clockwise arc from $\frac{\bx}{2}$ to $\frac{\by}{2}$.  Since the event $E_r(0)$ occurs,  there exists $z \in \mathcal{Z} \cap J_{\bx,\by}$ such that $F_{\frac{\rho r}{A^2}}(z)$ occurs. \textbf{For the remainder of this section, we shall work on the event $E_r(0)$ and for a $z$ as described above.}

We shall use $u_+ ,  v_+ ,  u_- ,  v_-$ to denote the corresponding points and $P_+ ,  P_-$ the corresponding $D_h$-geodesics associated with $F_{\frac{\rho r}{A^2}}(z)$.  We note that $P_+ \cup P_- \subseteq V_{\frac{\rho r}{A^2}}(z) \subseteq U$.  In the next lemma,  we will show that both $P_+$ and $P_-$ are also $D_{h-\phi}$-geodesics; in particular, this implies that the $\scX^{u_+,v_+}_{u_-,v_-}$ stays unchanged when one goes from the field $h$ to the field $h-\phi$.

\begin{lemma}
 On the event $E_r(0)$, with the point $z\in \cZ$ chosen as described above, the $D_h$-geodesics $P_+,P_-,\Gamma_{u_+,v_-},\Gamma_{u_-,v_+}$ are also the unique $D_{h-\phi}$-geodesics between their endpoints.
  \label{lem:chi_stays}
\end{lemma}

\begin{proof}
  The crucial aspect that we wish to harness is that on the set $V_{\rho r/A^2}(z) \subseteq B_{\rho r/A}(z)$, the restricted function $\phi\lvert_{B_{\rho r}(z)}$ is equal to its maximum value $K_f$. However, $K_g$, the global maximum of $\phi$ is strictly larger than $K_f$ and this is a difficulty which we shall need to bypass.

  Indeed, first recall that by \eqref{it:across} in Proposition~\ref{prop:definition_of_F_r(z)}, we necessarily have
  \begin{equation}
    \label{eq:19}
    \sup_{w_1,w_2\in B_{\rho r/A}(z)}D_h(w_1,w_2;B_{\rho r/\sqrt{A}}(z))< D_h(\partial B_{\rho r/\sqrt{A}}(z), \partial B_{\rho r}(z)).
  \end{equation}
  A particular consequence of the above is that for any $w_1,w_2\in B_{\rho r/A}(z)$ for which there is a $D_h$-geodesic $\Gamma_{w_1,w_2}$ satisfying $\Gamma_{w_1,w_2}\subseteq V_{\rho r/A^2}(z)\subseteq B_{\rho r/A}(z)$, any $D_{h-\phi}$ geodesic $\Gamma^\phi_{w_1,w_2}$ from $w_1,w_2$ must satisfy
  \begin{equation}
    \label{eq:21}
    \Gamma^\phi_{w_1,w_2}\subseteq B_{\rho r}(z).
  \end{equation}
  To see this, note that since $\phi$ is equal to $K_f$ on $\Gamma_{w_1,w_2}$, we have
  \begin{align}
    \label{eq:22}
    D_{h-\phi}(w_1,w_2)\leq e^{-\xi K_f} D_h(w_1,w_2)&= e^{-\xi K_f} D_h(w_1,w_2; B_{\rho r/\sqrt{A}}(z))\nonumber\\
    &< e^{-\xi K_f}D_h(\partial B_{\rho r/\sqrt{A}}(z), \partial B_{\rho r}(z))\nonumber\\
                                                                                                        &= e^{-\xi K_f}D_h(\partial B_{\rho r/\sqrt{A}}(z), \partial B_{\rho r}(z); \overline{B_{\rho r}(z)})\nonumber\\
    &\leq D_{h-\phi}(\partial B_{\rho r/\sqrt{A}}(z), \partial B_{\rho r}(z)),
  \end{align}
  where the last line above uses that the maximum value of $\phi\lvert_{B_{\rho r}(z)}$ is $K_f$. This proves \eqref{eq:21}.
  
  Now, for any path $\eta\subseteq B_{\rho r}(z)$ from $w_1$ to $w_2$, by using that $K_f$, the maximum value of $\phi\lvert_{B_{\rho r}(z)}$, is indeed attained on the set $V_{\rho r/A^2}(z)$, we have
  \begin{equation}
    \label{eq:24}
    \ell(\eta;D_{h-\phi})\geq e^{-\xi K_f} \ell(\eta;D_h)\geq e^{-\xi K_f}\ell(\Gamma_{w_1,w_2};D_h)=\ell(\Gamma_{w_1,w_2};D_{h-\phi}).
  \end{equation}
  Combining this with \eqref{eq:21}, we obtain that $\Gamma_{w_1,w_2}$, which was a priori defined to be a $D_h$-geodesic between $w_1,w_2$, is in fact also a $D_{h-\phi}$ geodesic between $w_1,w_2$.

  We are now ready to complete the proof-- this will be done by choosing $w_1\in \{u_+,u_-\}, w_2\in \{v_+,v_-\}$, Note that for all four possible choices of $w_1,w_2$, there is a unique $D_h$-geodesic $\Gamma_{w_1,w_2}$ and this is since we have chosen $z\in \cZ$ such that $F_{\rho r/A^2}(z)$ occurs (see items \eqref{it:uniquegeod1} and \eqref{it:concatenation} in Definition \ref{def:1*} and the definition of $F_{\rho r/A^2}(z)$). Further, for all $w_1,w_2$ as above, we have $\Gamma_{w_1,w_2}\subseteq V_{\rho r/A^2}(z)\subseteq B_{\rho r/A}(z)$. Thus, by the discussion above, $\Gamma_{w_1,w_2}$ is also a $D_{h-\phi}$ geodesic and as a result, we have $D_{h-\phi}(w_1,w_2)=e^{-\xi K_f}D_{h}(w_1,w_2)$.

  In particular, as a consequence of this, for all $w_1\in \{u_+,u_-\}, w_2\in \{v_+,v_-\}$ and \emph{any} $D_{h-\phi}$-geodesic $\Gamma_{w_1,w_2}^\phi$, we must also have
  \begin{equation}
    \label{eq:88}
\ell(\Gamma_{w_1,w_2}^\phi;D_h)\leq e^{\xi K_f}\ell(\Gamma_{w_1,w_2}^\phi;D_{h-\phi})=e^{\xi K_f}D_{h-\phi}(w_1,w_2)= D_{h}(w_1,w_2),
\end{equation}
thereby implying that $\Gamma_{w_1,w_2}^\phi$ is a $D_h$-geodesic. Thus, the unique $D_h$-geodesic $\Gamma_{w_1,w_2}$ is also the unique $D_{h-\phi}$-geodesic between $w_1$ and $w_2$.

\end{proof}

Before presenting the next lemma, we discuss a particular concern with our overall strategy. Recall that our design of the bump function seeks to force the $D_{h-\phi}$-geodesic $\Gamma_{\bz,\bw}^\phi$ to pass through a $\scX$. However, in order to ensure that a $\scX$ for $D_h$ remains a $\scX$ for $D_{h-\phi}$ (see Lemma \ref{lem:chi_stays}), we have setup things so that the function $\phi\lvert_{B_{\rho r}(z)}$ is equal to its maximum value in a neighbourhood of $\scX$. As a result, it is, a priori, possible that the geodesic $P^\phi$ circumvents passing through the $\scX$ by instead staying very close to the ``arms'' of the $\scX$ but not actually intersecting with any of the geodesics forming the $\scX$ (see Figure \ref{fig:circumvent}). The following lemma, which is a crucial part of the construction, rules out the above pathological behaviour.

Recall that for a simple path $\eta\colon [a,b]\rightarrow \CC$, we can view it as a directed path from $\eta(a)$ to $\eta(b)$ and thereby define its `left' and `right
 sides (see the notational comments from the introduction).
\begin{figure}
  \centering
  \includegraphics[width=\linewidth]{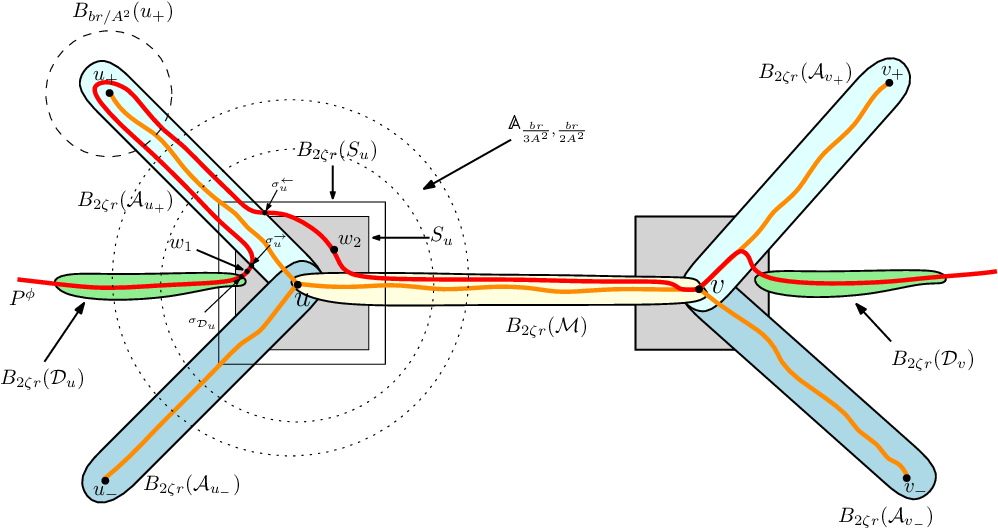}
  \caption{Illustration of the pathological scenario excluded via Lemma~\ref{lem:geodesic_intersects_chi}: Lemmas~\ref{lem:main_event_high_prob}-\ref{lem:chi_stays} imply that the part of the $D_{h-\phi}$-geodesic $P^{\phi}$ (red colour) between some times that it intersects the two green tubes has to stay within the region bounded by the tubes with colour green,  light blue, cyan,  and yellow respectively, and the two grey squares. We would like to show that between the aforementioned times that the red path intersects the two green tubes,  it intersects the union of the orange path from $u_+$ to $u$ and the orange path from $u_-$ to $u$ either to the right side of the former or to the left side of  the latter.  Similarly,  the time-reversal of the red path intersects the union of the orange path from $v$ to $v_+$ and the orange path from $v$ to $v_-$ either to the right side of the former or to the left side of the latter.  In order to prove the above,  we need to exclude pathological scenarios as the one depicted on the picture,  i.e.,  the red path makes two excursions within the cyan tube without intersecting any of the orange paths,  one from the grey square containing $u$ to a small neighbourhood of $u_+$ and the other from the small neighbourhood of $u_+$ to the grey square containing $u$.}
  \label{fig:circumvent}
\end{figure}
\begin{lemma}\label{lem:geodesic_intersects_chi}
  On the event $E_r(0)\cap \{\Gamma_{\bz,\bw}\cap B_{2r}(0)\neq \emptyset\}$, the following hold. %
  \begin{enumerate}
      \item \label{it:Pphi1} $P^\phi$ does not intersect $B_{\frac{b\rho r}{2A^2}}(u_-)\cup B_{\frac{b\rho r}{2A^2}}(u_+)$. %
     \item  $P^\phi$ does not intersect $B_{\frac{b\rho r}{2A^2}}(v_-)\cup B_{\frac{b\rho r}{2A^2}}(v_+)$.
         \item \label{it:Pphi2} $P^{\phi}$ intersects $P_+([0,\tau_u^+]) \cup P_-([0,\tau_u^-])$. With $P^\phi(s)$ being the first such intersection point, the curve $P^\phi\lvert_{[0,s]}$ intersects either on the right side of $P_+([0,\tau_u^+])$ or the left side of $P_-([0,\tau_u^-])$.
         \item \label{it:Pphi3} $P^{\phi}$ intersects $P_+([\tau_v^+,1]) \cup P_-([\tau_v^-,1])$. With $P^\phi(s)$ being the last such intersection point, the curve $P^\phi\lvert_{[s,D_{h-\phi}(\bz,\bw)}]$ intersects either on the right side of $P_+([\tau_v^+,1])$ or the left side of $P_-([\tau_v^-,1])$.

  \end{enumerate}
  
\end{lemma}

\begin{proof}
  We shall show only the statements \eqref{it:Pphi1}, \eqref{it:Pphi2} pertaining to the points $u,u_+,u_-$-- the corresponding statement for the points $v,v_+,v_-$ will follow analogously.  Moreover, all the objects except $P^\phi$ in both the items above are subsets of $B_{3r}(0)$ and since we have
\begin{align*}
\overline{P}^{\phi} \cap B_{3r}(0) = P^{\phi} \cap B_{3r}(0),
\end{align*}
it suffices to prove the claims of the lemma with $\overline{P}^{\phi}$ in place of $P^{\phi}$.

Recall that condition \eqref{it:er3} in Lemma~\ref{lem:Er-conds} implies that $\overline{P}^{\phi} \subseteq B_{2\zeta r}(U \cup \mathcal{W})$.  Hence, on combining with conditions \eqref{it:intersecting_branches}-\eqref{it:intersecting_square_again} in Lemma~\ref{lem:main_event_high_prob},  we obtain that $\overline{P}^{\phi}$ intersects all of the sets $B_{2\zeta r}(\mathcal{D}_u),  B_{2\zeta r}(\mathcal{A}_{u_+} \cup \mathcal{A}_{u_-}),  B_{2\zeta r}(S_u)$ and $B_{2\zeta r}(\mathcal{M} \setminus S_u)$.  We let $\sigma_u^{\rightarrow}$ (resp.\ $\sigma_u^{\leftarrow}$) be the first (resp.\ last) time that $\overline{P}^{\phi}$ intersects $\overline{B_{2\zeta r}(\mathcal{A}_{u_+} \cup \mathcal{A}_{u_-})}$ and let $\sigma_{\mathcal{D}_u}$ (resp.\ $\sigma_{\mathcal{M} \setminus S_u}$) be the last (resp.\ first) time before (resp.\ after) time $\sigma_u^{\rightarrow}$ (resp.\ $\sigma_u^{\leftarrow}$) that $\overline{P}^{\phi}$ intersects $\overline{B_{2\zeta r}(\mathcal{D}_u)}$ (resp.\ $\overline{B_{2\zeta r}(\mathcal{M} \setminus S_u)}$),  and note that conditions \eqref{it:intersecting_branches}-\eqref{it:intersect_endpoint} in Lemma~\ref{lem:main_event_high_prob} imply that
\begin{align*}
\sigma_{\mathcal{D}_u} < \sigma_u^{\rightarrow} < \sigma_u^{\leftarrow} < \sigma_{\mathcal{M} \setminus S_u} < \infty.
\end{align*}
We will show that there exist points $w_1 \in \overline{P}^{\phi}((\sigma_{\mathcal{D}_u} ,  \sigma_u^{\rightarrow})) \cap S_u$ and $w_2\in \overline{P}^{\phi}((\sigma_u^{\leftarrow} ,  \sigma_{\mathcal{M} \setminus S_u})) \cap S_u$ such that with $\tau_{w_1}$ (resp.\ $\tau_{w_2}$) denoting the time that $\overline{P}^{\phi}$ hits $w_1$ (resp.\ $w_2$), we have 
\begin{equation}
    \label{eq:wcond}
    \overline{P}^{\phi}|_{[\tau_{w_1} ,  \tau_{w_2}]}\subseteq B_{\frac{b\rho r}{2A^2}}(u)
\end{equation}

We first prove that the existence of points $w_1,w_2$ satisfying the above two conditions is sufficient to complete the proof of the lemma. Thereafter, we shall prove that such points $w_1,w_2$ indeed exist.

Assume that the points $w_1,w_2$ as above exist. %
Since $\overline{P}^{\phi} \subseteq B_{2\zeta r}(U \cup \mathcal{W})$, conditions \eqref{it:onlyway} and \eqref{it:onlyway1} in Lemma~\ref{lem:main_event_high_prob} imply that $\overline{P}^{\phi}$ can intersect $B_{\frac{b \rho r}{2A^2}}(u_-) \cup B_{\frac{b \rho r}{2A^2}}(u_+)$ only during the time-interval $[\sigma_u^{\rightarrow} ,  \sigma_u^{\leftarrow}] \subseteq [\tau_{w_1},\tau_{w_2}]$. %
However, by \eqref{it:geodesics_separated2} in Proposition \ref{prop:definition_of_F_r(z)}, we know that $B_{\frac{b\rho r}{2A^2}}(u)$ and $B_{\frac{b\rho r}{2A^2}}(u_+)\cup B_{\frac{b\rho r}{2A^2}}(u_-)$ are disjoint. Combining this with \eqref{eq:wcond} proves \eqref{it:Pphi1} in the lemma statement.

Next, assuming the existence of points $w_1,w_2$ satisfying \eqref{eq:wcond}, we prove that condition \eqref{it:Pphi2} in the lemma statement holds. 
Let $\wh{\tau}_u^+$ (resp.\ $\wh{\tau}_u^-$) denote the last time that $P_+$ (resp.\ $P_-$) intersects $\partial B_{\frac{b \rho r}{2A^2}}(u)$ before time $\tau_u$.  Since 
\begin{align*}
P_+([0,\tau_u^+)) \cap P_-([0,\tau_u^-)) = \emptyset
\end{align*}
and both $P_+$ and $P_-$ are simple curves,  we obtain that the set
\begin{align}
\label{eq:ABconn}
B_{\frac{b \rho r}{2A^2}}(u) \setminus \left(P_+([\wh{\tau}_u^+ ,  \tau_u^+]) \cup P_-([\wh{\tau}_u^- ,  \tau_u^-]) \right)
\end{align}
has exactly two connected components $A_u,B_u$ such that $\partial A_u$ (resp.\ $\partial B_u$) contains the right (resp.\ left) side of $P_+|_{[\wh{\tau}_u^+ ,  \tau_u^+]}$ and the left (resp.\ right) side of $P_-|_{[\wh{\tau}_u^-,\tau_u^-]}$.  Then condition \eqref{it:positive_distance_from_marked_points} in Proposition~\ref{prop:definition_of_F_r(z)} implies that $\mathcal{C}_u^2 \subseteq A_u$ and $\mathcal{C}_u^1 \subseteq B_u$.

Since $\zeta<\frac{\delta_3\rho}{4A^2}$ (see Lemma \ref{lem:main_event_high_prob}), we obtain that $B_{2\zeta r}(\mathcal{D}_u)\subseteq B_{\delta_3\rho r/(A^2)}(\cD_u)$. On combining this with items \eqref{it:I2}, \eqref{it:boxdist} in Proposition~\ref{prop:definition_of_F_r(z)}, we have
\begin{align*}
B_{2 \zeta r}(\mathcal{D}_u) \cap B_{2\zeta r}(S_u) \subseteq B_{\delta_2\rho r/A^2}(I_u) \subseteq \mathcal{C}_u^2. 
\end{align*}
Hence by item \eqref{it:intersecting_square} in Lemma \ref{lem:main_event_high_prob}, we have that $\overline{P}^{\phi}|_{(\sigma_{\mathcal{D}_u} ,  \sigma_u^{\rightarrow})}$ intersects $A_u$ and since
\begin{align*}
P_+([0,\tau_u^+]) \cup P_-([0,\tau_u^-]) \subseteq \mathcal{A}_{u_+} \cup \mathcal{A}_{u_-}
\end{align*}
and the path $\overline{P}^{\phi}|_{(\sigma_{\mathcal{D}_u} ,  \sigma_u^{\rightarrow})}$ is contained in $B_{\frac{b \rho r}{2A^2}}(u)$ (see \eqref{eq:inside_the_square}) and does not intersect $\mathcal{A}_{u_+} \cup \mathcal{A}_{u_-}$,  we must have that
\begin{align*}
w_1\in \overline{P}^{\phi}((\sigma_{\mathcal{D}_u} ,  \sigma_u^{\rightarrow})) \subseteq A_u.
\end{align*}
Similarly we have that
\begin{align*}
w_2\in \overline{P}^{\phi}((\sigma_u^{\leftarrow} ,  \sigma_{\mathcal{M} \setminus S_u})) \subseteq B_u.
\end{align*}
Now, by \eqref{eq:wcond}, we know that $\overline{P}^{\phi}((\tau_{w_1},\tau_{w_2})\subseteq B_{\frac{b \rho r}{2A^2}}(u)$. Since $A_u$ and $B_u$ are distinct connected components of the set in \eqref{eq:ABconn}, it follows that $\overline{P}^{\phi}|_{(\sigma_{\mathcal{D}_u} ,  \sigma_{\mathcal{M} \setminus S_u})}$ must intersect $P_+([\wh{\tau}_u^+ ,  \tau_u^+]) \cup P_-([\wh{\tau}_u^- ,  \tau_u^-])$, and it does so by first intersecting $B_{\frac{b \rho r}{2A^2}}(u) \cap \partial A_u$.  But then it has to intersect either the right side of $P_+|_{[0,\tau_u^+]}$ or the left side of $P_-|_{[0,\tau_u^-]}$.  This proves that condition \eqref{it:Pphi2} in the statement of the lemma holds.

Finally in order to complete the proof of the lemma, it suffices to show the existence of points $w_1$ and $w_2$ satisfying \eqref{eq:wcond}. In order to do so, we need to establish that $\overline{P}^\phi( (\sigma_{\cD_u},\sigma_u^\rightarrow))\cap S_u$ and $\overline{P}^{\phi}((\sigma_u^{\leftarrow} ,  \sigma_{\mathcal{M} \setminus S_u})) \cap S_u$ are both non-empty. First, we show that
\begin{align}
\label{eq:inside_the_square}
\overline{P}^{\phi}((\sigma_{D_u} ,  \sigma_u^{\rightarrow})) \subseteq B_{2\zeta r}(S_u)\textrm{ and } \overline{P}^{\phi}((\sigma_u^{\leftarrow} ,  \sigma_{\mathcal{M} \setminus S_u})) \subseteq B_{2\zeta r}(S_u).
\end{align}
Indeed, since $\sigma_u^\rightarrow$ is the first time that $\overline{P}^\phi$ intersects $\overline{B_{2\zeta r}(\cA_{u_+}\cup \cA_{u_-})}$, condition \eqref{it:intersecting_arms} in Lemma~\ref{lem:main_event_high_prob} implies that
\begin{equation}
  \label{eq:92}
\overline{P}^{\phi}((\sigma_{\mathcal{D}_u} ,  \sigma_u^{\rightarrow})) \cap B_{2\zeta r}\left(V_{\frac{b \rho r}{A^2}}(z) \setminus S_u\right) = \emptyset.  
\end{equation}
Thus, since $\overline{P}^{\phi} \subseteq B_{2\zeta r}(U \cup \mathcal{W})$ and since $\overline{P}^{\phi}((\sigma_{\mathcal{D}_u} ,  \sigma_u^{\rightarrow})) \subseteq B_{\rho r}(z)$, it follows that the first inclusion in \eqref{eq:inside_the_square} must hold. By an analogous argument using condition \eqref{it:intersectmid} and the definition of $\sigma_{\cM\setminus S_u}$, the second inclusion therein follows as well.

In view of \eqref{eq:inside_the_square}, the our goal now is to in fact show that $\overline{P}^{\phi}((\sigma_{\cD_u},\sigma_u^\rightarrow))$ intersects $S_u$ and not just $B_{2\zeta r}(S_u)$-- this will allow us to choose a point $w_1\in \overline{P}^{\phi}((\sigma_{\cD_u},\sigma_u^\rightarrow))\cap S_u$. Indeed, suppose $\overline{P}^{\phi}((\sigma_{\cD_u},\sigma_u^\rightarrow)) \cap S_u = \emptyset$. Note that item \eqref{it:I2} in Proposition \ref{prop:definition_of_F_r(z)} implies that 
\begin{align*}
\dist(B_{2\zeta r}(\cD_u) ,  B_{2\zeta r}(\mathcal{A}_{u_+} \cup \mathcal{A}_{u_-})) \geq \frac{50 \delta_2 \rho r}{A^2} - 4 \zeta r\geq \frac{50 \delta_3 \rho r}{A^2} - 4 \zeta r
\end{align*}
and so 
\begin{align*}
\text{diam}(\overline{P}^{\phi}((\sigma_{\cD_u} ,  \sigma_u^{\rightarrow}))) \geq \frac{50 \delta_3 \rho r}{A^2} - 4 \zeta r
\end{align*}
since $\overline{P}^{\phi}|_{[\sigma_{\cD_u} ,  \sigma_u^{\rightarrow}]}$ connects $B_{2\zeta r}(\cD_u)$ with $B_{2\zeta r}(\mathcal{A}_{u_+} \cup \mathcal{A}_{u_-})$.  Therefore, there must exist a segment of $\overline{P}^{\phi}((\sigma_{\cD_u},\sigma_u^\rightarrow))$ of Euclidean diameter
at least $50\delta_3\rho r/A^2-4\zeta r$ which is a subset of $B_{2\zeta r}(\partial U)$.  Then, since $\zeta<\frac{\delta_3\rho}{4A^2}$ (see Lemma~\ref{lem:main_event_high_prob} for the choice of $\zeta$),
$\overline{P}^{\phi}((\sigma_{\cD_u},\sigma_u^\rightarrow))$ contains a segment of Euclidean diameter at least $\frac{\delta_3 \rho r}{10^4 A^2}$ which is contained in $B_{2\zeta r}(\partial U)$. Thus condition~\eqref{it:lqg_distance_thin_ngbd} in the definition of $E_r(0)$ combined with the fact that $\phi \leq K_f$ on $B_{\rho r}(z)$ imply that the $D_{h-\phi}$-length of $\overline{P}^{\phi}([\sigma_{\cD_u},\sigma_u^\rightarrow])$ is at least $e^{-\xi K_f}100B r^{\xi Q} e^{\xi h_r(0)}$.  However we have already shown (see Lemma \ref{lem:Er-conds} \eqref{it:er2}) that the $D_{h-\phi}$-length of $\overline{P}^{\phi}$ is at most $e^{-\xi K_f} (B+4) r^{\xi Q} e^{\xi h_r(0)}$ and so we obtain a contradiction.  Therefore, we have shown that there must exist a $w_1 \in \overline{P}^{\phi}((\sigma_{\cD_u},\sigma_u^\rightarrow)) \cap S_u$, and we now fix such a point.

Now, we similarly show that the required point $w_2$ exists. Suppose that $\overline{P}^{\phi}((\sigma_u^\leftarrow,\sigma_{\mathcal{M} \setminus S_u})) \cap S_u = \emptyset$. By an argument analogous to the above, $\overline{P}^{\phi}((\sigma_u^\leftarrow,\sigma_{\mathcal{M} \setminus S_u}))$ must contain a segment of Euclidean diameter at least $\frac{\delta_3 \rho r}{10^4 A^2}$ which is contained in $B_{2\zeta r}(\partial U)$.  Therefore we obtain a contradiction as before and so there exists $w_2 \in \overline{P}^{\phi}((\sigma_u^\leftarrow,\sigma_{\mathcal{M} \setminus S_u})) \cap S_u$.

We now show that these $w_1,w_2$ satisfy \eqref{eq:wcond}. Let $\tau_{w_1},\tau_{w_2}$ be the hitting times of $w_1$ and $w_2$ by $\overline{P}^{\phi}$ respectively.  Then \eqref{it:upper_holder} in Proposition \ref{prop:definition_of_F_r(z)} combined with the fact that $\phi = K_f$ on $S_u$ imply that the $D_{h-\phi}$-length of $\overline{P}^{\phi}([\tau_{w_1},\tau_{w_2}])$ is at most
\begin{align*}
e^{-\xi K_f} \delta_1^\chi \left(\frac{\rho r}{A^2}\right)^{\xi Q} e^{\xi h_{\frac{\rho r}{A^2}}(z)}.
\end{align*}
Now, with the aim of obtaining a contradiction, let us assume that \eqref{eq:wcond} does not hold. If this is the case, then $\overline{P}^{\phi}([\tau_{w_1},\tau_{w_2}])$ must make a crossing across $\A_{\frac{b\rho r}{3A^2},\frac{b\rho r}{2A^2}}(u)\subseteq B_{\rho r}(z)$. Now, by combining the above with \eqref{it:lower_holder} in Proposition \ref{prop:definition_of_F_r(z)}, we obtain that the $D_{h}$-length of $\overline{P}^{\phi}([\tau_{w_1},\tau_{w_2}])$ is at least
\begin{displaymath}
  (\frac{b}{6})^{\chi'} (\frac{\rho r}{A^2})^{\xi Q} e^{\xi h_{\frac{\rho r}{A^2}}(z)},
\end{displaymath}
and so the $D_{h-\phi}$-length of $\overline{P}^{\phi}([\tau_{w_1},\tau_{w_2}])$ is at least $e^{-\xi K_f} (\frac{b}{6})^{\chi'} (\frac{\rho r}{A^2})^{\xi Q} e^{\xi h_{\frac{\rho r}{A^2}}(z)}$.  Therefore we obtain a contradiction since $\delta_1$ was chosen to be sufficiently small so that $(\frac{b}{6})^{\chi'} > \delta_1^\chi$ (see Proposition~\ref{prop:definition_of_F_r(z)}). Thus, \eqref{eq:wcond} must indeed hold and this completes the proof of the lemma.

\end{proof}

\begin{lemma}\label{lem:geodesic_goes_through_chi}
  We have $P_+ \cap P_-\subseteq P^{\phi}$ and $P^\phi\cap (P_+\cup P_-)$ is a curve itself. The path $P_+$ (resp.\ $P_-$) only intersects the left (resp.\ right) side of $P^\phi$.
\end{lemma}
\begin{proof}
Lemma~\ref{lem:geodesic_intersects_chi} implies that $P^{\phi}$ intersects both $P_+([0,\tau_u^+]) \cup P_-([0,\tau_u^-])$ and $P_+([\tau_v^+,1]) \cup P_-([\tau_v^-,1])$ and let $w_1$ (resp.\ $w_2$) be the first (resp.\ last ) intersection points corresponding to these. By Lemma \ref{lem:chi_stays}, $P_+,P_-,\Gamma_{u_+,v_-},\Gamma_{u_-,v_+}$ are the unique $D_{h-\phi}$-geodesics between their endpoints and note that all these geodesics contain the segment $P_+\cap P_-$. Thus, the portion of $P^{\phi}$ between $w_1,w_2$ must also contain $P_+\cap P_-$.  Moreover, both $P^{\phi} \cap P_+$ and $P^{\phi} \cap P_-$ are segments of $P^{\phi}$ and
\begin{align*}
    P_+ \cap P_- \subseteq (P^{\phi} \cap P_+) \cap (P^{\phi} \cap P_-),
\end{align*}
and so we obtain that $P^{\phi} \cap (P_+ \cup P_-)$ is a segment of $P^{\phi}$.  In particular, $P^{\phi} \cap (P_+ \cup P_-)$ must be a curve. On combining the above with \eqref{it:Pphi2}, \eqref{it:Pphi3} in Lemma \ref{lem:geodesic_intersects_chi}, we obtain the final assertion about the side of intersection as well. This completes the proof of the lemma.

\end{proof}

\subsection{The $D_h$-geodesic goes through many $\scX$s.}\label{sec:geodesic_goes_through_chi}

Now we are going to show that with high probability,  we have that every $D_h$-geodesic of the form $\Gamma_{\bz,w}$ goes through at least one of the $\scX$s introduced in the definition of the event $F_r(z)$ simultaneously for all $\bz $ in some bounded grid which we will specify,  and all $w \in U$ such that the $D_h$-distance between $\bz$ and $w$ is sufficiently large.

For this, we are going to apply Corollary \ref{cor:main_corollary} with $\lambda_1 = \frac{1}{4},  \lambda_2 = 2,  \lambda_3 = 3,  \lambda_4 = 4$, and $\lambda_5 = 5$.  We fix $\nu > 0$ and we let the constant $\mathbbm{p} \in (0,1)$ of Section~\ref{sec:tubes} be given by the constant $\mathbbm{p}$ of Proposition~\ref{thm:theorem_uniqueness_lqg_metric} corresponding to $\nu$.  Then we choose the constants $\delta,  \Delta,B,\zeta,\alpha,\theta,M,\Lambda_0$ as in Lemma~\ref{lem:main_event_high_prob} so that $\mathbb{P}(E_r(z))\geq \mathbbm{p}$ for all $z \in \mathbb{C},  r>0$.  It remains to define the events $\mathfrak{C}_r^{\bz,\bw,\eta}(z)$ appearing in \cite[Theorem~4.2]{GM21}.  %
Fix $r>0$ and points $\bz,\bw \in \mathbb{C} \setminus B_{4r}(z)$ and suppose that we have the setup of Sections~\ref{sec:setup}-\ref{sec:new_field}. %
We define $\mathfrak{C}_r^{\bz,\bw,\eta}(z)$ to be the event that the following conditions hold.

\begin{enumerate}[(1)]
\item %
  \label{it:chiispresent}
There exists a $z'\in z+\cZ$ such that for the points $u_+,u_-,v_+,v_-$ corresponding to $z'$ and the radius $\rho r/A^2$, the event $\scX^{u^+,v^+}_{u_-,v_-}$ occurs.

\item \label{it:subset}  With $P_+=\Gamma_{u_+,v_+}, P_-=\Gamma_{u_-,v_-}$ corresponding to the points $u_-,u_+,v_-,v_+$ associated to the $z'$ above, $\eta\cap(P_+\cup P_-)$ is a curve and we have $P_+\cap P_-\subseteq \eta$.
\item \label{it:int1} $P_+$ (resp.\ $P_-$) only intersects the left (resp.\ right) side of $\eta$.

\item \label{it:notint} $\eta$ does not intersect the set $B_{\frac{b\rho r}{2A^2}}(u_+) \cup B_{\frac{b\rho r}{2A^2}}(v_+) \cup B_{\frac{b\rho r}{2A^2}}(u_-) \cup B_{\frac{b\rho r}{2A^2}}(v_-)$.
\item \label{it:dirich1} With $\Lambda= e^{\Lambda_0}$, we have
\begin{align*}
\exp\left(-(h,\phi)_{\nabla} - \frac{1}{2} (\phi,\phi)_{\nabla} \right) \leq \Lambda,\,\,\text{for all}\,\,\phi \in \mathcal{G}_r.
\end{align*}
\end{enumerate}

We are now ready to state and prove the following lemma.

\begin{lemma}\label{thm:main_chi_result}
The events $E_r(z)$ and $\mathfrak{C}_r^{\bz,\bw,\eta}(z)$ defined in this section satisfy the assumptions stated in Proposition \ref{thm:theorem_uniqueness_lqg_metric}.
\end{lemma}

\begin{proof}

  Since all the relevant events are measurable with respect to $h$ viewed modulo an additive constant, just for this proof, we adjust the normalisation of $h$ such that $h_{4r}(z)=0$ (as opposed to the usual convention $h_1(0)=0$). By item \eqref{it:measurable} in Lemma \ref{lem:main_event_high_prob}, $E_r(z)$ is a.s.\ determined by $(h-h_{5r}(z))|_{\mathbbm{A}_{\frac{r}{4},4r}(z)}$. Secondly, by Proposition \ref{prop:definition_of_F_r(z)}, for all $z'\in \cZ$, the event $F_{\rho r/A^2}(z')$ is measurable with respect to $\sigma(h\lvert_{B_{\rho r}(z')})\subseteq \sigma(h\lvert_{B_{3r}(z)})$. Further, by \eqref{it:across} in Proposition~ \ref{prop:definition_of_F_r(z)}, on the event $F_{\rho r/A^2}(z')$, the geodesics $P_+,P_-$ are contained in $B_{\rho r}(z')$ and are determined by $\sigma(h\lvert_{B_{\rho r}(z')})$. Finally, by condition ~\eqref{it:geodloc} in Proposition~\ref{prop:definition_of_F_r(z)},  we have that $B_{\frac{br}{2A^2}}(u_+) \cup B_{\frac{br}{2A^2}}(v_+) \cup B_{\frac{br}{2A^2}}(u_-) \cup B_{\frac{br}{2A^2}}(v_-)\subseteq B_{\rho r}(z')$ as well. As a consequence, since $B_{\rho r}(z')\subseteq B_{3r}(z)$ for all $z'\in \cZ$, we obtain that items \eqref{it:subset}, \eqref{it:int1}, \eqref{it:notint} in the definition of $\mathfrak{C}_r^{\bz,\bw,\eta}(z)$
  are measurable with respect to $\sigma(h\lvert_{B_{4r}(z)})$ and $\eta$ stopped at its last exit time from $B_{4r}(z)$. Finally, all functions in $\cG_r$ are supported on $\A_{r/4,3r}(z)\subseteq B_{3r}(z)$ and this shows that item \eqref{it:dirich1} in the definition of $\mathfrak{C}_r^{\bz,\bw,\eta}(z)$ is measurable with respect to $\sigma(h\lvert_{B_{3r}(z)})$. This shows that events $E_r(z)$ and $\mathfrak{C}_r^{\bz,\bw,\eta}(z)$ satisfy condition \eqref{it:locality_property} in Proposition~\ref{thm:theorem_uniqueness_lqg_metric}. In fact, for $\mathfrak{C}_r^{\bz,\bw,\eta}(z)$, it shows the stronger statement: $\mathfrak{C}_r^{\bz,\bw,\eta}(z)$ is determined by $h|_{B_{3r}(z)}$ and the path $\eta$ stopped at the last time that it exits $B_{3r}(z)$.

  Item \eqref{it:lower_bound_on_prob_of_E_r(z)} in Lemma \ref{lem:main_event_high_prob} shows that $\PP(E_r(z))$ can be made arbitrarily high and thus condition \eqref{it:high_prob} in Proposition~\ref{thm:theorem_uniqueness_lqg_metric} is satisfied as well.

  To show that condition \eqref{it:inclusion_of_events} in the definition of $\mathfrak{C}_r^{\bz,\bw,\eta}(z)$ holds as well, we use an absolutely continuity argument modeled after \cite[Lemma 5.4]{GM21}. First note that since we are working with fixed $\bz,\bw\notin B_{4r}(z)$, there is a unique geodesic $\Gamma_{\bz,\bw}$ and thus, as usual, we write $\mathfrak{C}_{r}^{\bz,\bw}(z)$ as shorthand for $\mathfrak{C}_{r}^{\bz,\bw,\Gamma_{\bz,\bw}}(z)$. Locally, we use $\mathfrak{C}_r^{\phi}(z)$ to denote the event $\mathfrak{C}_r^{\bz,\bw}(z)$ occurring with the field $h-\phi$ in place of $h$.  Now, observe that
  \begin{equation}
    \label{eq:28}
E_r(z) \cap \{\Gamma_{\bz,\bw} \cap B_{2r}(z) \neq \emptyset\} \subseteq \mathfrak{C}_r^{\phi}(z) \cap \{\Gamma_{\bz,\bw}^\phi \cap B_{2r}(z) \neq \emptyset\}.
\end{equation}
To see the above, we first choose $z'\in \cZ$ such that $F_{\rho r/A^2}(z')$ occurs; this is guaranteed as we have assumed the occurence of $E_r(z)$. In particular, we know that the event $\scX^{u_+,v_+}_{u_-,v_-}$ occurs. Now, by Lemma \ref{lem:chi_stays}, we immediately obtain that the event $\scX^{u_+,v_+}_{u_-,v_-}$ must occur for the field $h-\phi$ as well, thereby implying that condition \eqref{it:chiispresent} in $\mathfrak{C}_r^{\phi}(z)$ is satisfied on the event $E_r(z)$. Now, furthermore, note that Lemmas \ref{lem:geodesic_intersects_chi}, \ref{lem:geodesic_goes_through_chi} imply that conditions \eqref{it:subset} to \eqref{it:notint} in the definition of $\mathfrak{C}_r^{\phi}(z)$ are satisfied on the event $E_r(z) \cap \{\Gamma_{\bz,\bw} \cap B_{2r}(z) \neq \emptyset\}$. Finally, item \eqref{it:bump_function_dirichlet_energy} in the definition of $E_r(z)$ implies that condition \eqref{it:dirich1} in the definition of $\mathfrak{C}_r^{\bz,\bw,\eta}(z)$ (for $h-\phi$ in place of $h$) is also satisfied on the event $E_r(z)$.

Recall from the beginning of the proof that $h_{4r}(z)=0$. Thus, since $B_{3r}(z)\cap \partial B_{4r}(z)=\emptyset$, we can use the Markov property of the whole plane GFF (see e.g.\ \cite[Lemma A.1]{Gwy19}) to write $h\lvert_{B_{3r}(z)}$ as the sum of a harmonic function measurable with respect to $h\lvert_{\CC\setminus B_{3r}(z)}$ and an independent zero boundary GFF on $B_{3r}(z)$. Thus, by the absolute continuity properties (see the beginning of Section \ref{sec:bump_functions}) of the zero-boundary GFF, the conditional law of $h-\phi$ given $h\lvert_{\CC\setminus B_{3r}(z)}$ is a.s.\ mutually absolutely continuous to the conditional law of $h$ with the corresponding Radon Nikodym derivative being precisely $M_h= \exp(-(h,\phi)_\nabla -(\phi,\phi)_\nabla/2)$. As a result, we can write
\begin{align}
  \label{eq:29}
  \PP(E_r(z) \cap \{\Gamma_{\bz,\bw} \cap B_{2r}(z) \neq \emptyset\} \giv h|_{\mathbb{C} \setminus B_{3r}(z)})&\leq  \PP(\mathfrak{C}_r^{\phi}(z) \cap \{\Gamma_{\bz,\bw}^\phi \cap B_{2r}(z) \neq \emptyset\} \giv h|_{\mathbb{C} \setminus B_{3r}(z)})\nonumber\\
                                                                     &= \EE[M_h \ind_{\mathfrak{C}_r^{\bz,\bw}(z)\cap \{\Gamma_{\bz,\bw}\cap B_{2r}(z)\neq \emptyset\}}\lvert h\lvert_{\CC\setminus B_{3r}(z)}]\nonumber\\
  &\leq \Lambda \PP(\mathfrak{C}_r^{\bz,\bw}(z)\cap \{\Gamma_{\bz,\bw}\cap B_{2r}(z)\neq \emptyset\}\lvert h\lvert_{\CC\setminus B_{3r}(z)}),
\end{align}
where the last line uses that $M_h\leq \Lambda$ on $\mathfrak{C}_{r}^{\bz,\bw}(z)$ as holds by item \eqref{it:dirich1} in the definition of $\mathfrak{C}_r^{\bz,\bw,\eta}(z)$. This shows that the last condition in Proposition \ref{thm:theorem_uniqueness_lqg_metric} is indeed satisfied by the definitions of $E_r(z)$ and $\mathfrak{C}_{r}^{\bz,\bw}(z)$. This completes the proof.
  
\end{proof}
We are now ready to state the final result of this section and this is the only result that shall be used in the sequel. We shall now use Corollary \ref{cor:main_corollary} along with our definition of $\mathfrak{C}_r^{\bz,\bw,\eta}(z)$ to show that for all $\bz$ lying on a suitably fine lattice and for all $w$ which are not too close to $\bz$, \emph{all} geodesics $\Gamma_{\bz,w}$ pass through $\scX$s.
\begin{proposition}
  \label{prop:9}
  Fix a bounded open set $U$. Set $\varepsilon_n=2^{-n}$ for all $n\in \NN$. Then almost surely, the following holds for all $n$ large enough. First, with $\mathcal{H}_n=((\frac{\varepsilon_n}{2})^{1/\chi}\ZZ^2)\cap U$, for all $u\in U$, there exists a $\bz\in \mathcal{H}_n$ such that $D_h(u,\bz)\leq \varepsilon_n / 2$.

  Further, for any fixed $\alpha>0$, there exist random $0<\bphi<\bpsi<\alpha/20$ depending only on $U$ and $\alpha$ such that the following almost surely holds for all $n$ large enough. For all $\bz\in \mathcal{H}_n$, all $w\in U$ satisfying $D_h(w,\bz)\geq \alpha$, and every $D_h$-geodesic $\Gamma_{\bz,w}$, there exists a $(z,r)\in \CC\times [\varepsilon_n^{1/(2\chi')},\varepsilon_n^{1/ (4\chi')}]$ %
  satisfying the following properties.
  \begin{enumerate}
  \item \label{it:F} There exists a $z'\in z+\cZ$ such that for the points $u_+,u_-,v_+,v_-$ corresponding to $z'$ and the radius $\rho r/A^2$, the event $\scX^{u^+,v^+}_{u_-,v_-}$ occurs.
  \item \label{it:overlapping_geodesics}
    With $P_+=\Gamma_{u_+,v_+}, P_-=\Gamma_{u_-,v_-}$ being the unique $D_h$-geodesics corresponding to the points $u_-,u_+,v_-,v_+$ associated to the $z'$ above and with $\tau^-<\tau^+$ being such that $P_-\cap P_+=\Gamma_{\bz,w}\lvert_{[\tau^-,\tau^+]}$, we have $[\tau^-,\tau^+]\subseteq [\bphi,\bpsi]$. %
  \item \label{it:leftright} $P_+$ (resp.\ $P_-$) intersects $\Gamma_{\bz,w}$ on its left (resp.\ right) side.
  \item \label{it:chiinside}We have $P_+\cup P_- \subseteq B_{\varepsilon_n^{1/(4\chi')}}(P_+\cap P_-)$.
  \item \label{it:endpoints_of_arms_away_from_geodesic} 
    The set $B_{16(2048\varepsilon_{n})^{1/\chi'}}(\Gamma_{\bz,w}\lvert_{[\tau^--2(25\times 2048\varepsilon_{n})^{\chi/(4\chi')},\tau^+ + 2(25\times 2048\varepsilon_{n})^{\chi/(4\chi')}]})$ does not intersect any of $u_-,u_+,v_-,v_+$.
 \end{enumerate}
\end{proposition}

\begin{proof}
First we note that by the a.s.\  bi-H\"older continuity of $D_h$ with respect to the Euclidean metric (Lemma~\ref{lem:holder_regularity}), it is a.s.\  the case that there exists a random $\varepsilon$ such that for all $z\in U$ and $w$ with $|z-w| \leq \varepsilon$,  it holds that
\begin{align}\label{eqn:bi_holder_continuity}
|z-w|^{\chi'} \leq D_h(z,w) \leq |z-w|^{\chi}.   
\end{align}
As a result, almost surely, by taking $n$ large enough, we can ensure that for each $u\in U$, there exists a $\bz\in \cH_n$ such that $D_h(u,\bz)\leq \varepsilon_n/2$.

Now, as a consequence of \eqref{eqn:bi_holder_continuity}, we can choose a dyadic (random) $\ell$ such that for all $\bz\in \cH_n,w\in U$ with $D_h(\bz,w)\geq \alpha$, we necessarily have $|\bz-w|\geq 16\ell$. We can also ensure that $\ell$ is chosen small enough so that
\begin{equation}
  \label{eq:62}
  D_h(\bz,\partial B_{2\ell}(\bz))\leq \alpha/20,
\end{equation}
and this choice will be useful shortly. Now, Lemma \ref{thm:main_chi_result} assures us that the events $E_r(z)$ and $\mathfrak{C}_r^{\bz,\bw,\eta}(z)$ defined earlier satisfy the assumptions from Proposition \ref{thm:theorem_uniqueness_lqg_metric}. Thus, we can invoke Corollary \ref{cor:main_corollary} with $\nu=1$, $\lambda_1 = \frac{1}{4},  \lambda_2 = 2,  \lambda_3 = 3,  \lambda_4 = 4$, and $\lambda_5 = 5$ and with $n$ replaced by $\lfloor n/(4\chi')\rfloor$ to obtain that, almost surely, for all $n$ large enough, for all $\bz\in \cH_n$ and $w\in U$ with $D_h(\bz,w)\geq \alpha$, and every $D_h$-geodesic $\Gamma_{\bz,w}$, there exists $(z,r)\in \CC\times [\varepsilon_n^{1 / (2\chi')},\varepsilon_n^{1/ (4\chi')}]$ such that
\begin{align}
  \label{eq:50}
  B_{4r}(z) \subseteq \mathcal{B}_{\tau_{2 \ell}(\bz)}^{\bullet}(\bz) \setminus \mathcal{B}_{\tau_{\ell}(\bz)}^{\bullet}(\bz), \quad  \Gamma_{\bz ,  w} \cap B_{2r}(z) \neq \emptyset
\end{align}
and $\mathfrak{C}_r^{\bz ,  w,\Gamma_{\bz,w}}(z)$ (as defined just above Lemma \ref{thm:main_chi_result}) occurs-- we now let $u_-,u_+,v_-,v_+$ be the points corresponding to the $z'\in z+\cZ$ for which $F_{\rho r/A^2}(z')$ occurs. By construction (see Figure \ref{fig:detailedchi}), note that $P_+,P_-\subseteq B_{2r}(z)$. Recall from \eqref{it:subset} in the definition of $\mathfrak{C}_r^{\bz ,  w,\Gamma_{\bz,w}}(z)$ that $P_+\cap P_-\subseteq \Gamma_{\bz,w}$. Now, by \eqref{eq:50}, we know that for any point $x\in B_{4r}(z)$, we have
\begin{equation}
  \label{eq:59}
  D_h(\bz,\partial B_{\ell}(\bz))\leq D_h(\bz,x)\leq D_h(\bz, \partial B_{2\ell}(\bz)).
\end{equation}

Note that \eqref{eqn:bi_holder_continuity} implies that 
\begin{align*}
    \ell^{\chi'} \leq D_h(\bz , \partial B_{\ell}(\bz)) \leq D_h(\bz , x) \leq D_h(\bz , \partial B_{2\ell}(\bz)) \leq (2\ell)^{\chi}
\end{align*}
and so setting $\bphi = \ell^{\chi'}$ and $\bpsi = (2\ell)^{\chi}$, we obtain that since $|\bz-w|\geq 16\ell$, and since $P_+,P_-\subseteq B_{2r}(z)$, for all $x\in P_+\cap P_-$, we must have $D_h(\bz,x)\in [\bphi, \bpsi]$. Possibly by taking $\ell$ to be smaller, we can assume that $\bpsi \leq \frac{\alpha}{20}$. This proves \eqref{it:overlapping_geodesics}.%
Condition \eqref{it:leftright} immediately follows from the corresponding condition \eqref{it:int1} in the definition of the event $\mathfrak{C}_r^{\bz,w,\Gamma_{\bz,w}}(z)$ earlier.

Now, recall that condition \eqref{it:notint} in the definition of $\mathfrak{C}_r^{\bz,w,\Gamma_{\bz,w}}(z)$ implies that
\begin{align*}
  \Gamma_{\bz,w}\cap   \left(B_{\frac{b \rho r}{2A^2}}(u_+) \cup B_{\frac{b \rho r}{2A^2}}(u_-) \cup B_{\frac{b \rho r}{2A^2}}(v_+) \cup B_{\frac{b \rho r}{2A^2}}(v_-)\right)=\emptyset.
\end{align*}

In particular, possibly by taking $n$ to be larger, we obtain that the Euclidean distance between $\Gamma_{\bz,w}$ and $\{u_- , u_+ , v_- , v_+\}$ is at least
\begin{align*}
    \frac{b \rho r}{2A^2} \geq \frac{b \rho \varepsilon_n^{1 / (2\chi')}}{2A^2} > 16 (2048 \varepsilon_n)^{1 / \chi'}.
\end{align*}
As a result, it follows that \eqref{it:endpoints_of_arms_away_from_geodesic} holds as well.

Finally, to prove \eqref{it:chiinside}, we note that (see \eqref{it:geodloc} in Proposition \ref{prop:definition_of_F_r(z)}) $P_+,P_- \subseteq B_{\frac{\rho r}{4A}}(z')$ where we recall that $z'\in z+\cZ$ is such that $F_{\frac{\rho r}{A^2}}(z')$ occurs. In particular, we have that 
\begin{align*}
    P_+ \cup P_- \subseteq B_{\frac{\rho r}{A}}(P_+ \cap P_-).
\end{align*}
Since $\frac{\rho r}{A} \leq \frac{\rho \varepsilon_n^{1 / (4\chi')}}{A} < \varepsilon_n^{1 / (4\chi')}$ for $n$ large enough, we obtain that \eqref{it:chiinside} holds as well. This completes the proof of the proposition.

\end{proof}

\begin{remark}
  \label{rem:extracond}
  Recall the setting of Remark \ref{rem:slightmod}, where one replaces the events $F_r(z)$ by the events $F_r(z)\cap H_r(z)\in \sigma(h\lvert_{B_{A^2 r}(z)})$ with $H_r(z)$ stipulating certain extra conditions on $P_1\cup P_2$ and the field $h$ in its vicinity. By using the same arguments, a version of Proposition \ref{prop:9}, where in item \eqref{it:F}, one demands the occurence of $\scX^{u_+,v_+}_{u_-,v_-}\cap H_{\rho r/A^2}(z')$ instead of just $\scX^{u_+,v_+}_{u_-,v_-}$ can also be obtained. The proof is the same, though there is one subtle point in the proof which we now elucidate. Namely, in Lemma \ref{lem:chi_stays}, on the event $E_r(0)$, now defined with the modified events $F_{\rho r/A^2}(z')$, for $z'\in \cZ$, we not only have that $P_+,P_-,\Gamma_{u_+,v_-},\Gamma_{u_-,v_+}$ are the unique $D_{h-\phi}$ geodesics between their endpoints, but also, the event $H_{\rho r/A^2}(z')$ occurs for the field $h$ replaced by $h-\phi$.  To see this, with $\pi(\cdot)$ being used to show a field viewed modulo an additive constant, the assumption $H_{\rho r/A^2}(z')\in \sigma(P_1\cup P_2, \bigcap_{\varepsilon>0}(\pi(h)\lvert_{B_\varepsilon(P_1\cup P_2)}))$ is important. Indeed, we already know that on $E_r(z)$, for $z'\in z+ \cZ$, all of $P_+,P_-,\Gamma_{u_+,v_-},\Gamma_{u_-,v_+}$ are the unique $D_{h-\phi}$ geodesics between their endpoints. As a result, since $\phi$ is constant and equal to the value $K_f$ on $V_{\rho r/A^2}(z')\supseteq P_1\cup P_2$, for all small enough $\varepsilon$, the fields $h\lvert_{B_{\varepsilon}(P_1\cup P_2)}$ and $(h-\phi)\lvert_{B_{\varepsilon}(P_1\cup P_2)}$ agree when viewed modulo additive constants and thus $H_{\rho r/A^2}(z')$ also occurs for the field $h-\phi$ since, because of $z'\in z+ \cZ$, we already know that it occurs for the field $h$. This allows us to work with Lemma \ref{thm:main_chi_result} with condition \eqref{it:subset} in the definition of $\mathfrak{C}_r^{\bz,\bw,\eta}(z)$ now instead demanding the occurrence of $\scX^{u_+,v_+}_{u_-,v_-}\cap H_{\rho r/A^2}(z')$ instead of just $\scX^{u_+,v_+}_{u_-,v_-}$.

Although the more general result described in this remark is not strictly necessary for this work, it is used in the follow up work \cite{BK26} to prove zero-one laws along the bulk of LQG geodesics.
\end{remark}

\section{Strong confluence of geodesics}
In this technical section, we will finish the proof of Theorem \ref{thm:3}. This will be done by combining Proposition~\ref{prop:9} with a geometric argument where one ``fits'' in a geodesic starting from a typical point between two possibly atypical geodesics which are very close but do not intersect each other; we shall adapt this geometric argument from \cite[Section 3.3]{MQ20}. We recall that \cite{MQ20} works in the Brownian case ($\gamma=\sqrt{8/3})$. In our setting of $\gamma$-LQG with general values $\gamma\in (0,2)$, owing to losses incurred from using H\"older continuity to go back and forth between Euclidean and $D_h$-distances, our statement of Proposition \ref{prop:9} is weaker than the corresponding statement in \cite{MQ20} (see Lemma 3.5 therein). For this reason, we shall have to undertake certain subtle modifications (see for e.g.\ Lemma \ref{lem:12}) in the argument from \cite[Section 3.3]{MQ20}, and doing so is the goal of this section. We begin by proving a strong confluence statement for the `one-sided' Hausdorff distance and then later upgrade it to the desired result Theorem \ref{thm:3}.

\subsection{Strong confluence for the one-sided Hausdorff distance}
\label{sec:strong-confl-one}

We shall use $d_H$ to denote the Hausdorff distance with respect to the $\gamma$-LQG metric $D_h$.  Also, with $\eta_1,\eta_2$ being two simple paths, we use $d_H^1(\eta_1,\eta_2)$ to denote the one-sided Hausdorff distance between them which is defined as follows.  First we note that the set $\mathbb{C} \cup \{\infty\} \setminus \eta_1$ is a simply connected domain whose boundary (seen as a collection of prime ends) is the union of two parts $\eta_1^{\mathrm{L}}$ and $\eta_1^{\mathrm{R}}$ which respectively correspond to the left and right sides of $\eta_1$. %
Now,  for $p,q \in \{\mathrm{L},\mathrm{R}\}$,  we want to define $\ell(\eta_1^p,\eta_2^q)$ as the infimum over all $r>0$ such that $\eta_2\setminus \eta_1$ is also contained in the $r$-neighborhood\footnote{This neighbourhood can be defined as the set of all points $x$ for which there exists a path $\gamma$ starting from a point $v\in\eta_1^{p}$ to $x$ with $\gamma\cap \eta_1=\{v\}$ and satisfying $\ell(\gamma;D_h)<r$.} of $\eta_1^p$ with respect to the interior-internal metric on $\mathbb{C} \setminus \eta_1$ induced by $D_h$,  and such that $\eta_1\setminus \eta_2$ is contained in the $r$-neighborhood of $\eta_2^q$ with respect to the interior-internal metric on $\mathbb{C} \setminus \eta_2$ induced by $D_h$.  %
We define $d_H^1(\eta_1 ,  \eta_2) = \min_{p,q\in \{\mathrm{L},\mathrm{R}\}}\{\ell(\eta_1^p,\eta_2^q)\}$.  Note that for any disjoint $D_h$-geodesics $\eta_1 ,  \eta_2$,  the distance $D_h(\eta_1 ,  \eta_2 ; \mathbb{C} \setminus \eta_1) = D_h(\eta_1 ,  \eta_2) < \infty$-- to see this,  we simply take a $D_h$-geodesic from $\eta_1$ to $\eta_2$ and note that this geodesic only intersects $\eta_1$ once.  Therefore, for disjoint $D_h$-geodesics $\eta_1,\eta_2$, we necessarily have $d_H^1(\eta_1 ,  \eta_2) < \infty$.  Moreover we note that the definition of the one-sided Hausdorff distance implies that for any two simple paths $\eta_1,\eta_2$, we have $d_H(\eta_1 ,  \eta_2) \leq d_H^1(\eta_1 ,  \eta_2)$.

The purpose of this section is to prove the following proposition which states that if two $D_h$-geodesics are very close to each other in the one-sided Hausdorff sense, then they must intersect.

\begin{proposition}\label{prop:close_one_sided_they_intersect}
  It is a.s.\  the case that the following is true.  Let $\eta_* \colon [0, T_*] \to \mathbb{C}$ be a $D_h$-geodesic.  Then there exists a $\delta$ sufficiently small depending on $\eta_*$ such that there does not exist any $D_h$-geodesic $\eta \colon [0,  T] \to \mathbb{C}$ satisfying $d_H^1(\eta_* ,  \eta) \leq \delta$ along with $\eta_*\cap \eta=\emptyset$. %
\end{proposition}

We now focus on proving Proposition~\ref{prop:close_one_sided_they_intersect}.  First, we state the following simple lemma which states that if two geodesics in a metric space are close in the Hausdorff sense,  then both their endpoints and their lengths have to be close as well.

\begin{lemma}(\cite[Lemma~3.8]{MQ20})\label{lem:endpoints_close}
Suppose that $(X,d)$ is a geodesic metric space and $\eta_i \colon [0,T_i] \to X$ for $i=1,2$ are $d$-geodesics.  Then, by possibly reversing the time of $\eta_2$,  we have 
\begin{align*}
\frac{|T_1-T_2|}{2}, \frac{d(\eta_1(0),\eta_2(0))}{5}, \frac{d(\eta_1(T_1),\eta_2(T_2))}{5} \leq d_H(\eta_1,\eta_2).
\end{align*}
\end{lemma}

Next we describe the setup of the proof of Proposition~\ref{prop:close_one_sided_they_intersect}.  Fix a bounded open set $U \subseteq \mathbb{C}$; it is easy to see that it is sufficient to prove Proposition \ref{prop:close_one_sided_they_intersect} only for geodesics $\eta_*,\eta\subseteq U$. Suppose we have a $D_h$-geodesic $\eta_*\colon [0,T_*]\rightarrow \CC$ contained in $U$, and we think of $\eta_*$ as fixed from now onwards. %

Now, for any $D_h$-geodesic $\eta \colon [0,T] \to \mathbb{C}$ such that $\eta\subseteq U$, we have (by Lemma \ref{lem:endpoints_close}) that
\begin{align}
  \label{eq:conds}
\frac{|T_*-T|}{2}, \frac{D_h(\eta_*(0),\eta(0))}{5}, \frac{D_h(\eta_*(T_*),\eta(T))}{5}\leq d_H^1(\eta_*,\eta).
\end{align}

Since we are allowed to work with $\eta$ disjoint from $\eta_*$ such that $d_H^1(\eta_*, \eta)$ is as small as we want, we can restrict to working with geodesics $\eta\subseteq U$ satisfying $d_H^1(\eta_*,\eta)\leq T_*/500$ and thereby,
\begin{equation}
  \label{eq:34}
  |T_*-T|\leq T_*/250.
\end{equation}

\emph{Without loss of generality, to fix notation, we shall always assume that $d_H^1(\eta_*,\eta)= \ell(\eta_*^{\mathrm{L}},\eta^{\mathrm{R}})$.} Now, given a geodesic $\eta$ as above, let $\gamma_\eta^{1}$ be a $D_h$-geodesic from $\eta_*(0)$ to $\eta(0)$ and $\gamma_\eta^{2}$ be a $D_h$-geodesic from $\eta_*(T_*)$ to $\eta(T)$. It can be checked that we can choose $\gamma_\eta^{1}$ and $\gamma_\eta^{2}$ such that all of
\begin{displaymath}
  \gamma_\eta^{1}\cap \eta_*, \gamma_\eta^{1}\cap \eta, \gamma_\eta^{2}\cap \eta_*, \gamma_\eta^{2}\cap \eta
\end{displaymath}
are connected. Consider the set given by $\CC\setminus (\eta_*\cup \gamma_\eta^{2}\cup \eta\cup \gamma_\eta^{1})$ which we note is a set consisting of two connected components-- one bounded and one unbounded.  Since the $D_h$-lengths of $\gamma_{\eta}^{1},\gamma_{\eta}^{2}$ are both at most $5d_H^1(\eta_*,\eta)\leq T_*/100$, the $D_h$-lengths of all the intersections
\begin{equation}
  \label{eq:31}
  \gamma_{\eta}^{1}\cap \eta_*, \gamma_{\eta}^{2}\cap \eta_*, \gamma_{\eta}^{1}\cap \eta, \gamma_{\eta}^{2}\cap \eta
\end{equation}
are at most $T_*/100$ as well, and as a result, we can uniquely define $\wt{U}_\eta$ to be the connected component of $\mathbb{C} \setminus (\eta_* \cup \gamma_{\eta}^2 \cup \eta \cup \gamma_{\eta}^1)$ whose boundary contains all of the points $\eta^{\mathrm{R}}(s)$ for $s\in [T_*/10,T-T_*/10]$. Note that the induced metric on the set $\overline{\wt{U}_\eta}$ is well-defined and admits geodesics as well, as can be seen by using that the boundary $\wt{U}_\eta$ is a union of geodesics for $D_h$.

Now, for an $l\in [T_*/10,T-T_*/10]$,  define $b_{\eta}^{l}$ so as to minimize the $D_h(\cdot,\cdot;\overline{\widetilde{U}_\eta})$-distance from $\eta_*(b_\eta^{l})$ to $\eta(l)$.  We let $\xi_{\eta}^l$ be the ``rightmost'' $D_h(\cdot,\cdot;\overline{\widetilde{U}_\eta})$ geodesic from $\eta_*(b_\eta^l)$ to $\eta(l)$, that such a unique rightmost geodesic can be seen by a standard argument involving taking subsequential limits of geodesics and by using that $\overline{\widetilde{U}_\eta}$ has geodesic boundaries. Following \cite{MQ20}, we shall refer to these paths as the ``ladders'' between the geodesics $\eta_*$ and $\eta$. Since these ladders are $D_h(\cdot,\cdot;\overline{\widetilde{U}_\eta})$-geodesics, and by the choice of $b_\eta^l$, we must have
\begin{equation}
  \label{eq:37}
  \xi_\eta^l\cap \eta_*=\{\eta_*(b_\eta^l)\},%
\end{equation}
and additionally,
\begin{equation}
  \label{eq:36}
  \xi_\eta^l\cap \eta \textrm{ must be connected.}
\end{equation}

Further, the ladders $\xi_\eta^l$ defined above satisfy the following two simple lemmas.
\begin{lemma}
  \label{lem:14}
  For all $l\in [T_*/10,T-T_*/10]$, we have $\ell(\xi_\eta^l;D_h)\leq d_H^1(\eta_*,\eta)=\ell(\eta_*^{\mathrm{L}},\eta^{\mathrm{R}})$. %
\end{lemma}
\begin{proof}
  By the definition of $d_H^1(\eta_*,\eta)$ (which we have assumed is equal to $\ell(\eta_*^{\mathrm{L}},\eta^{\mathrm{R}})$), there exists a path $\gamma$ from $\eta^{\mathrm{R}}(l)$ to $\eta^{\mathrm{L}}_*$ such that $\gamma$ intersects $\eta\cup \eta^*$ only at its endpoints and satisfies $\ell(\gamma;D_h)\leq d_H^1(\eta_*,\eta)$. The goal now is to show that we can choose $\gamma$ such that $\gamma\subseteq \overline{\wt{U}_\eta}$. This would complete the proof as then we would necessarily have $\ell(\xi_\eta^l;D_h)\leq \ell(\gamma;D_h)\leq d_H^1(\eta_*,\eta)$.

  Suppose that $\gamma$ intersects $\gamma_\eta^1\cup \gamma_\eta^2$ and let $s_0$ be the first time it does so. Without loss of generality, assume that $\gamma(s_0)\in \gamma_\eta^2$. Now, we claim that $\gamma\lvert_{[s_0,\ell(\gamma;D_h)]}\cap \gamma_\eta^1=\emptyset$. Indeed, if this were not so, then we would have
  \begin{equation}
    \label{eq:104}
    T_*=D_h(\eta_*(T_*),\eta_*(0))\leq \ell(\gamma_\eta^2;D_h)+ \ell(\gamma\lvert_{[s_0,\ell(\gamma;D_h)]};D_h)+ \ell(\gamma_\eta^1;D_h)\leq (5+1+5)d_H^1(\eta_*,\eta)\leq 11T_*/500,
  \end{equation}
  which is absurd.

  Now, define $t_0$ such that $\eta_*\cap \gamma_\eta^2=\eta_*([t_0,T_*])$. We first consider the case when $\gamma(\ell(\gamma;D_h))\notin \eta_*([t_0,T_*])$. Since $\gamma$ intersects $\eta\cup \eta_*$ only at its endpoints and does not intersect $\gamma_\eta^1$ after time $s_0$, the first and last intersections $\gamma(s_0),\gamma(s_0')$ of $\gamma\lvert_{(0,\ell(\gamma;D_h))}$ with $\partial \wt{U}_\eta$ must both be on $\gamma_\eta^2$. Replacing $\gamma\lvert_{[s_0,s_0']}$ by the portion of $\gamma_\eta^2$ between these points, and since $\gamma_{\eta}^2$ is a $D_h$-geodesic, the resulting path $\gamma$ satisfies $\gamma\subseteq \overline{\wt{U}_\eta}$. This completes the proof.

  The second case is when $\gamma(\ell(\gamma;D_h))\in \eta_*([t_0,T_*])$ Then, since $\gamma_\eta^2$ itself is a $D_h$-geodesic, we can modify $\gamma$ (without changing its length) by replacing $\gamma\lvert_{[s_0,\ell(\gamma;D_h)]}$ with a portion of $\gamma_\eta^2$. However, the resulting $\gamma$ passes through $\eta_*(t_0)$ and we now excise its portion after $\eta_*(t_0)$. The obtained path $\gamma$ travels from $\eta^{\mathrm{R}}(l)$ to $\eta_*(t_0)$ and satisfies $\gamma\subseteq \overline{\wt{U}_\eta}$. This completes the proof.

\end{proof}
\begin{lemma}
  \label{lem:7}
  For all $l\in [T_*/10,T-T_*/10]$, the ladder $\xi_{\eta}^l$ does not intersect either $\gamma_{\eta}^1$ or $\gamma_{\eta}^2$.  
\end{lemma}
\begin{figure}
  \centering
  \includegraphics[width=0.9\linewidth]{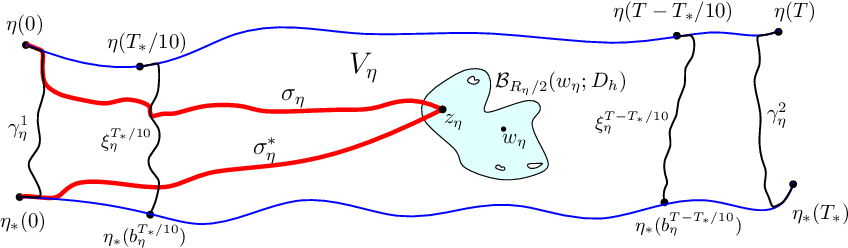}
  \caption{Illustration of the setup of Lemmas~\ref{lem:7}-\ref{lem:geodesic_between_geodesics}: The curves $\eta_* ,  \eta$ are depicted in red while the curves $\gamma_\eta^1,\xi_\eta^{T_*/10},\xi_\eta^{T-T_*/10}$,  and $\gamma_\eta^2$ are depicted in black.  The cyan coloured region represents the LQG metric ball centered at the grid point $w_{\eta}$ with radius $\frac{R_{\eta}}{2}$,  where the latter metric ball also contains the point $z_{\eta}$.  The $D_h$-geodesics $\sigma_{\eta}^*,\sigma_{\eta}$ are drawn in red and start from $z_{\eta}$ and emanate at $\eta_{*}(0)$ and $\eta(0)$ respectively, and they stay within the union of $\eta,\gamma_\eta^2,\eta_*,\gamma_\eta^1$ along with the region $\wt{U}_\eta$ surrounded by them.}
  \label{fig:sigma}
\end{figure}
\begin{proof}
  Recall from \eqref{eq:31} that the $D_h$-lengths of intersections of $\gamma_{\eta}^{1},\gamma_{\eta}^{2}$ with $\eta_*,\eta$ are all at most $T_*/100$. As a consequence, $\eta(l), \eta_*(b_\eta^l), \eta(T-T_*/10), \eta(T_*/10)\in \overline{\wt{U}_\eta}$. We will show that $\xi_\eta^l$ cannot intersect either $\gamma_{\eta}^{1}$ or $\gamma_{\eta}^{2}$. Indeed, if $\xi_\eta^l\cap \gamma_{\eta}^{1}\neq \emptyset$, by using Lemma \ref{lem:14}, we can write
 \begin{equation}
   \label{eq:32}
   T_*/10=D_h(\eta(T_*/10), \eta(0))\leq D_h(\eta(l),\eta(0)) \leq \ell(\xi_\eta^l;D_h)+\ell(\gamma_{\eta}^{1};D_h)\leq T_*/500+ T_*/100= 3T_*/250,
 \end{equation}
 which is not possible. Similarly, if $\xi_\eta^l\cap \gamma_{\eta}^{2}\neq \emptyset$, we would have
 \begin{equation}
   \label{eq:33}
      T-T_*/10=D_h(\eta(T-T_*/10), \eta(T))\leq D_h(\eta(l),\eta(T))\leq \ell(\xi_\eta^l;D_h)+\ell(\gamma_{\eta}^{2};D_h)\leq 3T_*/250,
    \end{equation}
where the above uses Lemma \ref{lem:14}. This contradicts \eqref{eq:34}, and is therefore not possible as well. This completes the proof.
\end{proof}
Let $V_\eta$ be the open set of points disconnected from $\infty$ by the union of $\eta_*|_{[b_{\eta}^{T_* / 10},b_{\eta}^{T - T_* / 10}]},  \xi_\eta^{T_*/10}$, $\eta|_{[T_*/10,T-T_*/10]}$, $\xi_\eta^{T-T_*/10}$ (see Figure \ref{fig:sigma}).  Note that $V_\eta$ is topologically a rectangle. We now have the following easy lemma.

\begin{lemma}
  \label{lem:13}
  For all $l\in (T_*/10+2d_H^1(\eta_*,\eta), T-T_*/10-2d_H^1(\eta_*,\eta))$, we have $\xi_\eta^l\subseteq \overline{V_\eta}$.
\end{lemma}

\begin{proof}
  Since all the ladders are chosen to be rightmost $D_h(\cdot,\cdot;\overline{\wt{U}_\eta})$ geodesics, $\xi_{\eta}^l$ is necessarily sandwiched between $\xi_{\eta}^{T_*/10},\eta, \xi_{\eta}^{T-T_*/10},\eta_*$. Thus, it suffices to show that for all $l\in (T_*/10+2d_H^1(\eta_*,\eta), T-T_*/10-2d_H^1(\eta_*,\eta))$, $\xi_\eta^l$ cannot intersect any of $\xi_{\eta}^{T_*/10},\xi_{\eta}^{T-T_*/10}$. We just show that $\xi_\eta^l\cap \xi_{\eta}^{T_*/10}=\emptyset$ and the intersection with the other path can be handled analogously. Now, since $\ell(\xi_\eta^l;D_h),\ell(\xi_{\eta}^{T_*/10};D_h)\leq d_H^1(\eta_*,\eta)$, if $\xi_\eta^l\cap \xi_{\eta}^{T_*/10}$ were non-trivial, then we would have
  \begin{equation}
    \label{eq:99}
    D_h(\eta(l),\eta(T_*/10))\leq \ell(\xi_\eta^l;D_h)+\ell(\xi_{\eta}^{T_*/10};D_h)\leq 2d_H^1(\eta_*,\eta)
  \end{equation}
  But since $\eta$ is geodesic and since $l\in (T_*/10+2d_H^1(\eta_*,\eta), T-T_*/10-2d_H^1(\eta_*,\eta))$, the above cannot hold. This completes the proof.
\end{proof}

Later, we shall apply Proposition \ref{prop:9} with the point $\bz$ therein lying inside $V_\eta$. To do so, we shall require that the two long sides of $V_\eta$ not be too close to each other, in the sense that there is an LQG metric ball contained within $V_\eta$ whose radius is not too small. Indeed, we shall frequently use the radii
\begin{equation}
  \label{eq:100}
\underline{R}_\eta = \sup\{r > 0 :\,\,\text{there exists}\,\,  w \in V_\eta, \cB_r(w) \subseteq V_\eta\} \quad \text{and} \quad R_\eta = \min\{\underline{R}_\eta,d^1_H(\eta_*,\eta)\}.  
\end{equation}
Recalling the notation $\varepsilon_n=2^{-n}$, let $k_\eta \in \mathbb{N}$ be such that
\begin{equation}
  \label{eq:53}
  2\varepsilon_{k_\eta} < R_\eta \leq 4 \varepsilon_{k_\eta.}
\end{equation}
Recall the grid $\mathcal{H}_n$ introduced in Proposition~\ref{prop:9}. Then, as long as we are working with $\eta$ with $d_H^1(\eta_*,\eta)$ small enough, and thereby $\varepsilon_{k_\eta}$ small enough, Proposition~\ref{prop:9} implies that
we can choose $w_\eta,z_\eta \in V_\eta$ (see Figure \ref{fig:sigma}) such that
\begin{equation}
  \label{eq:35}
  z_\eta \in \cH_{k_{\eta}} \cap\cB_{\varepsilon_{k_\eta}}(w_\eta)\subseteq \cB_{R_\eta/2}(w_\eta) \subseteq V_\eta.
\end{equation}

With the above setup at hand, we are now ready to provide the proof of Proposition~\ref{prop:close_one_sided_they_intersect}.  First, we will prove that we can find a geodesic $\sigma$ from $z_\eta$ to either $\eta(0)$ or $\eta_*(0)$ which is contained in $\overline{\wt{U}_\eta}\cup \eta$ or $\overline{\wt{U}_\eta}\cup \eta_*$ respectively. Then since the geodesic $\sigma$ starts from a point on the grid $\cH_{k_{\eta}}$ and travels a macroscopic Euclidean distance, our previous result (Proposition \ref{prop:9}) about the existence of $\scX$s implies that the geodesic has to pass through one of the $\scX$s whose LQG-size is at least of order $\varepsilon_{k_\eta}^{1/2}$ and is thus much larger than $R_\eta$. Combining with the uniqueness of geodesics $P_+,P_-$ in the definition of a $\scX$ (see \eqref{it:sides} in Definition \ref{def:1*}),  we obtain that both $\eta$ and $\eta_*$ pass through the same $\scX$ along $\sigma$ and therefore also intersect each other. In the following lemma, we prove the existence of a geodesic $\sigma$ with the above properties.

\begin{lemma}\label{lem:geodesic_between_geodesics} 
  At least one of the following hold. %
  \begin{enumerate}
  \item \label{it:geodnostar} There is a $D_h$-geodesic $\sigma_\eta$ from $z_\eta$ to $\eta(0)$ which is contained in $\overline{\wt{U}_\eta}\cup \eta$.
  \item \label{it:geodstar} There is a $D_h$-geodesic $\sigma_\eta^*$ from $z_\eta$ to $\eta_*(0)$ which is contained in $\overline{\wt{U}_\eta}\cup \eta_*$.
  \end{enumerate}
\end{lemma}

\begin{proof}
The proof closely follows the argument in the proof of \cite[Lemma~3.9]{MQ20} but we include it for completeness.  Suppose that the claim of the lemma does not hold.  Then every geodesic from $z_\eta$ to $\eta(0)$ exits $\overline{\wt{U}}_\eta\cup \eta$ and every geodesic from $z_\eta$ to $\eta_*(0)$ exits $\overline{\wt{U}}_\eta\cup \eta_*$. Let $\sigma_\eta^*$ be a $D_h$-geodesic from $z_\eta$ to $\eta_*(0)$ and note that $\sigma_\eta^*$ must exit $\overline{\wt{U}_\eta}\cup \eta_*$ through either $\eta,  \eta_*$ or $\gamma_\eta^{1} ,  \gamma_\eta^{2}$.

\paragraph{\emph{Case 1.\ $\sigma_{\eta}^*$ first exits $\overline{\wt{U}_{\eta}}\cup \eta_*$ through either $\eta_*$ or $\gamma_{\eta}^1$.}}

We claim that $\sigma_\eta^*$ cannot exit $\overline{\wt{U}_\eta}\cup \eta_*$ through $\eta_*$.  Indeed, if this were the case, then we can consider the first point $v_0\in \eta_*$ hit by $\sigma_\eta^*$. Then the concatenation of the part of $\sigma_\eta^*$ from $z$ to $v_0$ and the part of $\eta_*$ from $v_0$ to $\eta_*(0)$ is also a geodesic from $z$ to $\eta_*(0)$ and it is contained in $\overline{\wt{U}_\eta}\cup\ \eta_*$.  But this contradicts our assumption.  Similarly, if $\sigma_\eta^*$ were to first exit $\overline{\wt{U}_\eta}\cup \eta_*$ via $\gamma_{\eta}^1$, we could take the first intersection $v_0$ of $\sigma_\eta^*$ with $\gamma_{\eta}^1$ and then from then onwards use the part of $\gamma_{\eta}^1$ from $v_0$ to $\eta_*(0)$ (note that this portion of $\gamma_{\eta}^1$ does lie in the set $\overline{\wt{U}_\eta}\cup \eta_*$).

\paragraph{\emph{Case 2.\ $\sigma_{\eta}^*$ first exits $\overline{\wt{U}_{\eta}}\cup \eta_*$ through $\gamma_{\eta}^2$.}}

Next we show that $\sigma_\eta^*$ cannot first exit $\overline{\wt{U}_\eta}\cup \eta_*$ through $\gamma_\eta^{2}$. In fact, we establish that $\sigma_\eta^*$ cannot ever intersect $\gamma_\eta^{2}$. Indeed, suppose on the contrary that $\sigma_\eta^*$ intersects $\gamma_\eta^{2}$ at some point $v_2$.  Then, for topological reasons, it has to exit $V_\eta$ first and let $v_1 \in \sigma_\eta^* \cap \partial V_\eta$ be the point at which $\sigma_{\eta}^*$ first exits $V_{\eta}$. Then it is easy to see that
\begin{align}\label{eqn:1}
  &D_h(v_1 ,  \eta_*(0))\nonumber\\
  &\leq \max(\ell(\eta\lvert_{[0,T-T_*/10]};D_h), \ell(\eta_*\lvert_{[0,b_{\eta}^{T - T_* / 10}]};D_h))+ \max(\ell(\xi_\eta^{T_*/10};D_h),\ell(\xi_\eta^{T-T_*/10};D_h)) + \ell(\gamma_{\eta}^1 ; D_h) \nonumber\\
                        &\leq (\ell(\eta\lvert_{[0,T-T_*/10]};D_h)+ 2d_H^1(\eta_*,\eta))+d_H^1(\eta_*,\eta) + T_*/100\nonumber\\
                        &\leq (T- T_*/10 + 2T_*/500) +T_*/500 + T_*/100\nonumber\\
  &\leq T_* + T_*/250-T_*/10+2T_*/500+T_*/500 + T_* / 100\nonumber\\
                        &\leq T_* - 2T_*/25,
\end{align}
where to obtain the second last line, we have used \eqref{eq:34}. Further, since $\sigma_\eta^*$ first hits $v_1$ before hitting $v_2$,  we also have that 
\begin{align*}
  D_h(v_1 ,  \eta_*(0)) = D_h(v_1,v_2) + D_h(v_2,\eta_*(0)) \geq D_h(v_2,\eta_*(0)) &\geq D_h(\eta_*(T_*),\eta_*(0)) - D_h(\eta_*(T_*),v_2)\nonumber\\
                                                                                    &\geq T_* - D_h(\eta_*(T_*), \eta(T))\nonumber\\
  &\geq T_*- T_*/100,
\end{align*}
where we have used \eqref{eq:conds} to obtain the last line. This contradicts ~\eqref{eqn:1}. 

\paragraph{\emph{Case 3.\ $\sigma_{\eta}^*$ first exits $\overline{\wt{U}_{\eta}^*}\cup \eta_*$ through $\eta$.}}

The only remaining possibility is that $\sigma_\eta^*$ first exits $\overline{\wt{U}_\eta}\subseteq \overline{\wt{U}_\eta}\cup \eta_*$ through $\eta$. We show that by possibly modifying $\sigma_\eta^*$, we can assume that after $\sigma_\eta^*$ exits $\overline{\wt{U}_\eta}$ through $\eta$,  it does not re-enter $\wt{U}_\eta$.  Indeed,  recall from the previous paragraph that $\sigma_\eta^*$ cannot intersect $\gamma_\eta^{2}$. As a result, if $\sigma_\eta^*$ re-enters $\wt{U}_\eta$, it must do so via one of $\eta,\eta_*,\gamma_{\eta}^{1}$. Now, if $\sigma_\eta^*$ intersects $\eta$ again after exiting $\wt{U}_\eta$,  then we can replace the part of $\sigma_\eta^*$ between the first and last time that it intersects $\eta$ by the part of $\eta$ between these points.  Also if $\sigma_\eta^*$ intersects $\eta_*$ or $\gamma_\eta^{1}$ after it exits $\overline{\wt{U}_\eta}$ for the first time,  then we can modify $\sigma_\eta^*$ so that it follows $\eta_*$ or $\gamma_\eta^{1}$ after the first time that it intersects $\eta_*$ or $\gamma_\eta^{1}$ until it reaches $\eta_*(0)$.

In summary, we have shown that in order for item \eqref{it:geodstar} not to hold, we must have a geodesic $\sigma_\eta^*$ first exiting $\overline{\wt{U}_\eta}$ through $\eta$ remaining outside afterwards. By an analogous argument, in order for item \eqref{it:geodnostar} not to hold, we must have a geodesic $\sigma_\eta$ from $z_\eta$ to $\eta(0)$ first exiting $\overline{\wt{U}_\eta}$ through $\eta_*$ and thereafter remaining outside $\wt{U}_\eta$ afterwards. But then, by planarity, we must have that $\sigma_\eta$ and $\sigma_\eta^*$ intersect at some point $v$.  Therefore,  we can construct a new $D_h$-geodesic from $z_\eta$ to $\eta_*(0)$ which first exits $\overline{\wt{U}_\eta}$ through $\eta_*$ by concatenating the part of $\sigma_\eta$ from $z_\eta$ to $v$ and the part of $\sigma_\eta^*$ from $v$ to $\eta_*(0)$.  But we have already established that this cannot happen and so we obtain a contradiction.  Thus, at least one of items \eqref{it:geodnostar} and \eqref{it:geodstar} must hold and this completes the proof of the lemma.

\end{proof}

Recall the radius $R_{\eta}$ from \eqref{eq:100}; it is possible that $R_{\eta}$ (hence $\varepsilon_{k_{\eta}}$) is much smaller than $d_H^1(\eta_*,\eta)$. In the following lemma, which is a careful expansion of an argument implicit in \cite[Proof of Proposition 3.1]{MQ20}, we argue that $R_{\eta}$ cannot be much smaller than $\ell(\xi_{\eta}^l ; D_h)$. We note that, in the remainder of the section, the constant $2048$ shall appear very frequently. Thus, to avoid unnecessary clutter, we shall now set $K=2048$ and shall freely use $K$ throughout this section.
\begin{lemma}
  \label{lem:9}
For all $d_H^1(\eta_*,\eta)$ small enough and for all $l\in [T_*/10+3d_H^1(\eta_*,\eta), T-T_*/10-3d_H^1(\eta_*,\eta)]$, we have $\ell(\xi_{\eta}^l;D_h)< 512 R_\eta\leq K\varepsilon_{k_\eta}$.
\end{lemma}
\begin{proof}
  \emph{Step 0.\ Overview.} Let $v=\xi_{\eta}^l(L/2)$ be the midpoint of $\xi_{\eta}^l$ where $L = \ell(\xi_{\eta}^{l} ; D_h)$. It suffices to produce a metric ball of radius $L/512$ lying entirely inside $V_\eta$. If $\overline{\cB_{L/8}(v)}\cap \partial V_\eta=\emptyset$, then we have already produced a metric ball of radius $L/8$ lying inside $V_\eta$, so for the entire proof, we shall assume that
  \begin{equation}
    \label{eq:110}
   \overline{\cB_{L/8}(v)}\cap \partial V_\eta\neq \emptyset.
  \end{equation}
  We now give an outline for all the steps in the proof. In Step 1, we show that $\overline{\cB_{L/8}(v)}\cap \partial V_\eta\subseteq \eta$. In Step 2, it is shown that $\overline{\mathcal{B}_{L/8}(v)}$ cannot intersect both $\eta([0,l])$ and $\eta([l,T])$. Next, in Step 3, we will prove a small lemma, which we refer to as the ``\emph{local-to-global}'' result. In particular, we will show that for any points $x,y \in \xi_{\eta}^l$, if we assume that there exists a $D_h$-geodesic $\Gamma_{x,y}$ connecting $x$ with $y$ that stays in $\mathcal{B}_{L/8}(v)$, then the part of $\xi_{\eta}^l$ connecting $x$ with $y$ must also be a $D_h$-geodesic. %
  Finally, in Step 4, with $\cC$ denoting the connected component of $V_\eta\setminus \xi_\eta^l$ to the left of $\xi_\eta^l$, a point $w\in \cB_{L/8}(v)\cap \cC\subseteq V_\eta$ will be constructed such that $\mathcal{B}_{L/512}(w) \subseteq \cC$, thereby proving that $R_{\eta} > L / 512$ and completing the proof.

\emph{Step 1.\ $\overline{\cB_{L/8}(v)}$ does not intersect any of $\xi_\eta^{T_*/10},\xi_\eta^{T-T_*/10},\eta_*$.}
First, note that $\overline{\cB_{L/2}(v)}$ (and thereby $\overline{\cB_{L/8}(v)}$) is disjoint from both $\xi_\eta^{T_*/10}, \xi_\eta^{T-T_*/10}$ as else the $D_h$-distance from $\eta(l)$ to $\eta(T_*/10)$ or $\eta(T-T_*/10)$ would be at most
\begin{equation}
  \label{eq:103}
\frac{L}{2} + \frac{L}{2} + \max(\ell(\xi_\eta^{T_*/10};D_h),\ell(\xi_\eta^{T-T_*/10};D_h))\leq L+ d_H^1(\eta_* ,  \eta) \leq 2 d_H^1(\eta_* ,  \eta),  
\end{equation}
which is not possible as $l\in [T_*/10+3d_H^1(\eta_*,\eta), T-T_*/10-3d_H^1(\eta_*,\eta)]$. Now, we show that as long as $d_H^1(\eta_*,\eta)$ is small enough, $\overline{\cB_{L/8}(v)}$ must be disjoint from $\eta_*$.

Define $\alpha=\inf\{r:\cB^\bullet_{r}(v)\cap (\gamma_\eta^1\cup \gamma_\eta^2)\neq \emptyset\}$. The utility of the above choice of $\alpha$ is that for any time $s$ such that $\eta(s)\in \overline{\cB_{\alpha/2}(v)}\subseteq \cB_{\alpha/2}^\bullet(v)$, there is no path inside $\overline{\cB_{\alpha/2}(v)}\setminus \eta$ connecting $\eta^{\mathrm{L}}(s)$ to $\eta_*$. Now, note that by the triangle inequality,
\begin{align}
  \label{eq:101}
 T_*/10+3d_H^1(\eta_*,\eta)&\leq  \min(D_h(\eta(l),\eta(0)), D_h(\eta(l),\eta(T_*)))\nonumber\\
 &\leq \ell(\gamma_\eta^1;D_h)+ \ell(\gamma_\eta^2;D_h)+\alpha\nonumber\\
 &\leq 10d_H^1(\eta_*,\eta)+ \alpha,
\end{align}
thereby implying that $\alpha\geq T_*/10-7d_H^1(\eta_*,\eta)>T_*/20$. Thus, if we work with $d_H^1(\eta_*,\eta)\leq T_*/40$, then we are guaranteed that $L/8\leq \alpha /16$, which implies that for any two points $x,x'\in \overline{\cB_{L/8}(v)}$, we must have $\Gamma_{x,x'}\subseteq  \overline{\cB_{\alpha/2}(v)}$ for every geodesic $\Gamma_{x,x'}$.%

We now show that $\overline{\cB_{L/8}(v)}$ cannot intersect $\eta_*$. With the aim of obtaining a contradiction, assume that there exists a point $x\in \eta_*\cap \overline{\cB_{L/8}(v)}$. We now consider a geodesic $\Gamma_{v,x}\subseteq \overline{\cB_{\alpha/2}(v)}$. Since there is no path inside $\overline{\cB_{\alpha/2}(v)}\setminus \eta$ connecting the left side of $\eta$ to $\eta_*$, if $\Gamma_{v,x}$ exists $\overline{\wt{U}_\eta}$, it must finally re-enter it again via $\eta$ itself. Thus, by replacing this excursion of $\Gamma_{v,x}$ outside $\overline{\wt{U}_\eta}$ by a portion of $\eta$, we can assume that $\Gamma_{v,x}$ stays in $\overline{\wt{U}_\eta}$ until it intersects $\eta_*$ for the first time, say at a point $\Gamma_{v,x}(s)\in \eta_*$. But now, $\Gamma_{v,x}\lvert_{[0,s]}$ is a path from $\eta(v)$ to $\eta^*$ lying entirely $\overline{\wt{U}_\eta}$ with $\ell(\Gamma_{v,x}\lvert_{[0,s]};D_h)\leq L/8$. Since $v=\xi_\eta^l(L/2)$, concatenating the above path with $\xi_\eta^l\lvert_{[L/2,L]}$ yields a path from $\eta(l)$ to $\eta_*$ lying entirely within $\overline{\wt{U}_\eta}$ with $D_h$-length at most $5L/8$. However, this contradicts the definition of $b_\eta^l$ since it was chosen so as to minimise the $D_h(\cdot,\cdot;\overline{\wt{U}_\eta})$ distance from $\eta_*$ to $\eta(l)$. Thus $\overline{\cB_{L/8}(v)}$ cannot intersect $\eta_*$.

\paragraph{\emph{Step 2.\ $\overline{\mathcal{B}_{L/8}(v)}$ cannot simultaneously intersect both $\eta([0,l])$ and $\eta([l,T])$.}}
To show this, we first argue that
\begin{equation}
  \label{eq:102}
  D_h(v,\eta(l)) = D_h(v , \eta(l) ; \overline{\wt{U}_{\eta}}).
\end{equation}
To prove the above, we consider a geodesic $\Gamma_{v,\eta(l)}$. Recall from \eqref{eq:103} that $\overline{\mathcal{B}_{L/2}(v)}$ is disjoint from $\xi_{\eta}^{T_* / 10}\cup \xi_{\eta}^{T - T_* / 10}$ and thus can only exit $\overline{\wt{U}_{\eta}}$ via $\eta\cup \eta_*$. If $\Gamma_{v,\eta(l)}$ were to hit $\eta_*$ before hitting $\eta$, then by concatenating with $\xi_\eta^l\lvert_{[L/2,L]}$, we would obtain a path from $\eta(l)$ to $\eta_*$ which stays entirely inside $\overline{\wt{U}_\eta}$ and which has $D_h$-length strictly less than $L$, thereby contradicting the definition of $b_{\eta}^l$. Thus, $\Gamma_{v,\eta(l)}$ must intersect $\eta$ prior to intersecting $\eta_*$. However, once $\Gamma_{v,\eta(l)}$ does intersect $\eta$, say at a point $x$, from $x$ onwards, we can simply utilise the $D_h$-geodesic $\eta$ to reach the point $\eta(l)$. As a result, we can legitimately assume that $\Gamma_{v,\eta(l)}\subseteq \overline{\wt{U}_\eta}$, and we must thus have \eqref{eq:102}.

We now return to the primary goal of the current step. With the aim of obtaining a contradiction, assume that there exist $t_1\in [0,l],t_2 \in [l,T]$ such that $\eta(t_1),\eta(t_2)\in \overline{\cB_{L/8}(v)}$. Since $\eta$ is a $D_h$-geodesic, we must have $t_2-t_1\leq L/4$. However, by using \eqref{eq:102}, we must also have
  \begin{align}
    \label{eq:38}
    t_2-t_1\geq t_2- l=D_h(\eta(l),\eta(t_2)) &\geq D_h(v,\eta(l))- D_h(v,\eta(t_2))\nonumber\\
    &= D_h(v , \eta(l) ; \overline{\wt{U}_{\eta}})- D_h(v,\eta(t_2))\geq L/2-L/8= 3L/8,
  \end{align}
  which is a contradiction.%
  Thus, in view of \eqref{eq:110} and Step 1.\ in the proof, $\overline{\cB_{L/8}(v)}$ must intersect precisely one of $\eta([0,l])$ and $\eta([l,T])$. For the remainder of the proof, we assume that it intersects the latter, and the same arguments shall work in the other case as well.

\paragraph{\emph{Step 3. Local-to-global result.}}
We now prove a short result that will be used repeatedly throughout the proof. The assertion is that if for some $0\leq r_1<r_2\leq L$, we have a $D_h$-geodesic $\Gamma_{\xi_\eta^l(r_1),\xi_\eta^l(r_2)}$ lying entirely inside $\cB_{L/8}(v)$, then in fact, $\xi_\eta^l\lvert_{[r_1,r_2]}$ must be a $D_h$-geodesic. Indeed, as per our assumption, if $\overline{\cB_{L/8}(v)}$ exits $\overline{V_\eta}$, it only does so through $\eta([l,T])$ which we emphasize is to the right of $\xi_\eta^l$. As a result, if the path $\Gamma_{\xi_\eta^l(r_1),\xi_\eta^l(r_2)}$ does exit $\overline{V_\eta}$, it exits and re-enters $\overline{V_\eta}$ via the path $\eta([l,T])$. However, $\eta$ itself is a $D_h$-geodesic and thus, if the above happens, we can modify $\Gamma_{\xi_\eta^l(r_1),\xi_\eta^l(r_2)}$ so that $\Gamma_{\xi_\eta^l(r_1),\xi_\eta^l(r_2)}$ and $\eta$ intersect on an interval. As a result, the modified geodesic $\Gamma_{\xi_{\eta}^{l}(r_1),\xi_{\eta}^{l}(r_2)}$ does indeed lie in $\overline{V_\eta}$. As a result, $\xi_\eta^l\lvert_{[r_1,r_2]}$, which was initially just a $D_h(\cdot,\cdot;\overline{\wt{U}_{\eta}})$ must also be a $D_h$-geodesic. Throughout the proof, we shall refer to the above short result as the \emph{local-to-global} result.

\paragraph{\emph{Step 4. Constricting a radius $L/512$ metric ball inside $V_\eta$}}
Recall the set $V_\eta\setminus \xi_\eta^l$ has two connected components, and we use $\cC$ to denote the one to the left of $\xi_\eta^l$. Now we will prove that there exists a ball $\mathcal{B}_{L / 512}(w)$ contained in $\mathcal{C}$. To find such a $D_h$-metric ball, we will first show that there exist $v_0 , v_1 \in \partial \mathcal{B}_{L / 64}(v)\cap \cC$ such that $D_h(v_0,v_1)$ is comparable to $L$. Then, this will imply that there exists a point $v_2$ on an ``arc'' of $\partial \mathcal{B}_{L / 64}(v)$ connecting $v_0$ with $v_1$ that lies to the left of $\xi_{\eta}^l$ and such that both of $D_h(v_0 , v_2)$ and $D_h(v_2,v_1)$ are of order $L$. Hence, we can construct (see Figure \ref{fig:finitetarget}) the ball $\mathcal{B}_{L/512}(w)$ by following a $D_h$-geodesic $\Gamma_{v_2,v}$ and taking $w$ to be the point on $\Gamma_{v_2,v}$ which is sufficiently close to $v_2$ but whose $D_h$-distance from $v_2$ is of order $L$.

\paragraph{\emph{Step 4.1. Construction of the points $v_0,v_1$.}}
We consider the points $v_0=\xi_\eta^l(L/2-L/64)$, $v_1=\xi_\eta^l(L/2+L/64)$. We first note that the path $\xi_\eta^l\lvert_{[L/2-L/64,L/2+L/64]}$ which was a priori defined only to be a rightmost $D_h(\cdot,\cdot;\overline{\wt{U}_{\eta}})$-geodesic must in fact be a $D_h$-geodesic from $v_0$ to $v_1$. To see this, first note that $v_0,v_1\in \overline{\cB_{L/64}(v)}$. As a result, any $D_h$-geodesic $\Gamma_{v_0,v_1}$ must satisfy
\begin{equation}
  \label{eq:111}
  \Gamma_{v_0,v_1}\subseteq \cB_{L/8}(v).
\end{equation}
To see this, note that $\ell(\Gamma_{v_0,v_1};D_h) \leq 2\times L/64\leq L/32$ but $D_h(\partial \cB_{L/64}(v), \partial \cB_{L/8}(v))= 7L/64> L/32$. Thus we have shown that $\Gamma_{v_0,v_1}\subseteq \cB_{L/8}(v)$. By the local-to-global result, $\xi_\eta^l\lvert_{[L/2-L/64,L/2+L/64]}$ is a $D_h$-geodesic. As a consequence, we have
  \begin{equation}
    \label{eq:39}
    D_h(v_0,v_1)=L/32.
  \end{equation}
  \begin{figure}
  \centering
  \includegraphics[width=0.9\linewidth]{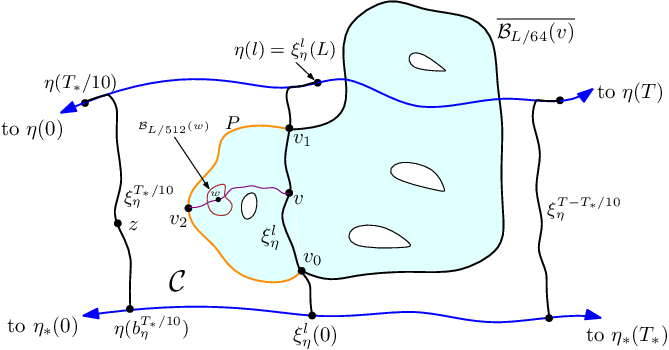}
  \caption{Steps 4.2, 4.3, 4.4 in the proof of Lemma \ref{lem:9}: Here $v=\xi_\eta^l(L/2)$ is the midpoint of the ladder $\xi_\eta^l$ and $v_1=\xi_\eta^l(L/2+L/64),v_0= \xi_\eta^l(L/2-L/64)$. The region between the ladders $\xi_\eta^{T_*/10}$ and $\xi_\eta^l$ is denoted by $\cC$. The metric ball $\overline{\cB_{L/64}(v)}$ is shown in cyan. As shown in \eqref{eq:97}, $\xi_\eta^l\setminus \xi_\eta^l\lvert_{[L/2-L/64,L/2+L/64]}\subseteq (\overline{\cB_{L/64}(v)})^c$. Further, since $\overline{\cB_{L/64}(v)}$ does not intersect any of $\eta_*,\xi_\eta^{T_*/10},\eta([0,l])$, there exist paths from $v_0,v_1\in \partial \cB_{L/64}(v)$ to $z\in \xi_\eta^{T_*/10}$ intersecting $\overline{\cB_{L/64}(v)}$ only at their starting points. Therefore (see Step 4.2), $v_0,v_1\in \partial \cB_{L/64}^{z,\bullet}(v)$. The path $P\subseteq \partial B_{L/64}^{z,\bullet}(v)$ connects $v_0$ to $v_1$ and stays to the left of $\xi_\eta^l$. The point $v_2\in P$ is chosen (Step 4.3) such that it is equidistant to $v_0$ and $v_1$. Finally, $w\in \Gamma_{v_2,v}$ can be chosen so as to satisfy $\cB_{L/512}(w)\subseteq \cC$ (see Step 4.4).}
  \label{fig:finitetarget}
\end{figure}
  \paragraph{\emph{Step 4.2. The set $\cB_{L/64}^{z,\bullet}(v)$ satisfying $v_0,v_1\in \partial \cB_{L/64}^{z,\bullet}(v)$}} We refer the reader to Figure \ref{fig:finitetarget} for depiction of this step in the proof. Fix a point $z\in \xi_\eta^{T_*/10}$ and consider the ``filled metric ball with finite target point'' $\cB_{L/64}^{z,\bullet}(v)$ defined as the complement of the unique connected component of $\CC\setminus \overline{\cB_{L/64}(v)}$ which contains $z$. By \cite[Lemma 2.4]{GPS20}, it is known that the boundary $\partial \cB_{L/64}^{z,\bullet}(v)\subseteq \partial \cB_{L/64}(v)$ is a Jordan curve. %
  
  Now, note that for $r\in [0, L/2-L/64)\cup (L/2+L/64,L]$, we have
  \begin{equation}
    \label{eq:97}
    \xi_\eta^l(r)\notin \overline{\cB_{L/64}(v)}.
  \end{equation}
  We just show the above for $r\in [0,L/2-L/64)$ and the other interval can be handled analogously. Suppose that there exists $r\in [0,L/2-L/64)$ such that $\xi_\eta^l(r)\in \overline{\cB_{L/64}(v)}$. By the same argument as used in \eqref{eq:111}, any geodesic $\Gamma_{\xi_\eta^l(r),v}$ must stay within $\cB_{L/8}(v)$ and thus, by the local-to-global result, $\xi_\eta^l\lvert_{[r,L/2]}$ must be a $D_h$-geodesic. However, this implies that $D_h(\xi_\eta^l(r),v)= L/2-r> L/64$ and thus $\xi_\eta^l(r)\notin \overline{\cB_{L/64}(v)}$, thereby proving \eqref{eq:97}.

  Using \eqref{eq:97} along with the fact that $\overline{\cB_{L/64}(v)}$ does not intersect any of $\eta_*,\xi_\eta^{T_*/10},\eta([0,l])$, we observe that (see Figure \ref{fig:finitetarget}) the union $\eta([0,l])\cup \xi_\eta^{T_*/10}\cup \xi_\eta^{l}\lvert_{[L/2+L/64,L]}$ (resp.\ $\eta_*\cup \xi_\eta^{T_*/10}\cup \xi_\eta^{l}\lvert_{[0,L/2-L/64]}$) contains a path from $v_1$ (resp.\ $v_0$) to $z$ intersecting $\overline{\cB_{L/64}(v)}$ only at $v_1$ (resp.\ $v_0$). This shows that $v_0,v_1$ satisfy $v_0,v_1\in \partial  \cB_{L/64}^{z,\bullet}(v)$.

 \paragraph{\emph{Step 4.3. The point $v_2\in \partial \cB_{L/64}^{z,\bullet}(v)\cap \cC$ equidistant from $v_0,v_1$}}
 
 Recall that $\partial \cB_{L/64}^{z,\bullet}(v)$ is a Jordan curve. Now, consider the path $P\subseteq \partial \cB_{L/64}^{z,\bullet}(v)$ from $v_0$ to $v_1$ lying to the left of $\xi_\eta^l$. Note that we must have $P\setminus \{v_0,v_1\}\subseteq \cC$ since $\overline{\cB_{L/64}(v)}$ does not intersect any of $\eta_*,\xi_\eta^{T_*/10},\eta([0,l])$. Now, by a continuity argument, where one varies a point along $P$, there must exist a $v_2\in P\setminus \{v_0,v_1\}=\partial \mathcal{B}_{L / 64}^{z,\bullet}(v)\cap \cC$ satisfying
  \begin{equation}
    \label{eq:40}
    D_h(v_2,v_0)= D_h(v_2,v_1).
  \end{equation}
  Since $D_h(v_0,v_1)=L/32$ (see \eqref{eq:39}), by the triangle inequality, we must have $D_{h}(v_2,v_0)+D_h(v_2,v_1)\geq D_h(v_0,v_1)$, thereby yielding that
  \begin{equation}
    \label{eq:98}
    D_h(v_2,v_0)=D_h(v_2,v_1)\geq L/64=D_h(v_2,v).
  \end{equation}
  \paragraph{\emph{Step 4.4. Construction of a ball $\mathcal{B}_{L/512}(w)\subseteq \cC$.}}
  Now we consider a $D_h$-geodesic $\Gamma_{v_2,v}$ and note that $\ell(\Gamma_{v_2,v};D_h)=L/64$. Since $\Gamma_{v_2,v}\subseteq \overline{\cB_{L/64}(v)}$ and since $\xi_{\eta}^l\lvert_{[L/2-L/64,L/2+L/64]}$ is a $D_h$-geodesic, by possibly modifying $\Gamma_{v_2,v}$ to agree with $\xi_{\eta}^l\lvert_{[L/2-L/64,L/2+L/64]}$ on a sub-segment, we can assume that $\Gamma_{v_2,v}\subseteq \overline{\cC}$. Define $w=\Gamma_{v_2,v}(L/256)$. We now claim that
  \begin{equation}
    \label{eq:41}
    \cB_{L/512}(w)\subseteq \cC.
  \end{equation}
  For this, since $D_h(w,v)=L/64-L/256$, by the triangle inequality we already know that
  \begin{equation}
    \label{eq:43}
    \cB_{L/512}(w)\subseteq \cB_{L/64}(v),
  \end{equation}
and since $\cB_{L/64}(v)$ can only hit $\partial V_{\eta}$ at $\eta([l,T])$ and since $w\in \cC$, we only need to show that $\cB_{L/512}(w)\cap \xi_\eta^l=\emptyset$. With the aim of obtaining a contradiction, assume that there exists a point $u\in \xi_\eta^l\cap B_{L/512}(w)$. Then we must have
  \begin{equation}
    \label{eq:42}
    D_h(u,v_2)\leq D_h(u,w)+D_h(w,v_2)\leq L/512+ L/256 = 3L/512.
  \end{equation}
Now, using \eqref{eq:43} along with \eqref{eq:97}, note that $u\in \xi\lvert_{[L/2-L/64,L/2+L/64]}$ which we recall is a $D_h$-geodesic. As a result, at least one of the quantities $D_h(u,v_0), D_h(u,v), D_h(u,v_1)$ must be at most $L/128$. Further, by construction, all the distances $D_h(v_2,v_0),D_h(v_2,v), D_h(v_2,v_1)$ are at least $L/64$. Thus, by the triangle inequality, we have $D_h(u,v_2)\geq L/64-L/128=L/128$. However, this contradicts \eqref{eq:42}.   It follows that
\begin{align*}
\mathcal{B}_{L / 512}(w) \subseteq \cC\subseteq V_{\eta},
\end{align*}
and hence $R_{\eta} > L / 512$.  This completes the proof of the lemma.

\end{proof}

As a consequence of the above lemma and a H\"older continuity argument (using Lemma~\ref{lem:holder_regularity}), as long as $d_H^1(\eta_*,\eta)$ is small enough, for all $l\in [T_*/10+3d_H^1(\eta_*,\eta), T-T_*/10-3d_H^1(\eta_*,\eta)]$, we have %
\begin{equation}
  \label{eq:70}
  \xi_{\eta}^l\subseteq B_{(K\varepsilon_{k_\eta} )^{1/\chi'}}(\eta(l)).
\end{equation}
As in the above, for the remainder of this section, we shall always assume that $d_H^1(\eta_*,\eta)$ and thereby $\varepsilon_{k_\eta}$ is small enough-- this will be used to use H\"older continuity to translate between $D_h$-distances and LQG distances inside small balls contained within $\overline{\wt{U}_{\eta}}$.

For $T_*/10\leq l_1\leq l_2\leq T-T_*/10$, we shall use $W_{l_1,l_2}$ to refer to the compact set defined as the union of the curves $\xi_\eta^{l_1}, \eta_*\lvert_{[b_\eta^{l_1}, b_\eta^{l_2}]}, \xi_\eta^{l_2}, \eta\lvert_{[l_1, l_2]}$ along with the bounded open set disconnected by them from $\infty$. We caution that the set $W_{l_1,l_2}$ is not necessarily a topological rectangle-- indeed, the ladder $\xi_\eta^{l_1}$ and and the curve $\eta\lvert_{[l_1,l_2]}$ can overlap on a segment. Note that by the triangle inequality, for any $l_1,l_2$ as in the above, we have
\begin{align}
  \label{eq:75}
  b_\eta^{l_2}-b_\eta^{l_1}= D_h(\eta_*(b_\eta^{l_1}), \eta_*(b_\eta^{l_2}))&\leq D_h(\eta_*(b_\eta^{l_1}), \eta(l_1)) + D_h(\eta(l_1), \eta(l_2))+ D_h(\eta(l_2), \eta_*(b_\eta^{l_2}))\nonumber\\
  &\leq \ell(\xi_\eta^{l_1};D_h)+l_2-l_1+\ell(\xi_\eta^{l_2};D_h).
\end{align}

Shortly, we shall need to use that geodesic segments which stay entirely inside $\overline{V_\eta}$ cannot avoid intersecting with ladders for too long. Roughly, this can be argued by showing that if the above happens, then the region between ``successive'' ladders will have large $D_h$-diameter which can be shown to not occur by using H\"older continuity since the $D_h$-lengths of the ladders are small. We now state a precise version of the above statement, and this will be used shortly.
\begin{lemma}
  \label{lem:12}
Suppose that \eqref{it:geodstar} in Lemma \ref{lem:geodesic_between_geodesics} holds. We then have the following:
  \begin{enumerate}
  \item \label{it:1.1}For any $0\leq \alpha<\beta\leq \ell(\sigma_\eta^*;D_h)$ such that $\sigma_\eta^*\lvert_{[\alpha,\beta]}\subseteq \overline{V_\eta}$ and $\beta-\alpha \geq (25d_H^1(\eta_*,\eta))^{\chi/\chi'}$, the path $\sigma_\eta^*\lvert_{[\alpha,\beta]}$ must intersect $\xi_\eta^{l_1},\xi_\eta^{l_2}$ for some $l_1,l_2\in [T_*/10,T-T_*/10]$ satisfying $l_2-l_1\geq 4d_H^1(\eta_*,\eta)$
  \item \label{it:1.2} For any $0\leq \alpha<\beta\leq \ell(\sigma_\eta^*;D_h)$ such that $\sigma_\eta^*\lvert_{[\alpha,\beta]}\subseteq W_{T_*/10+3d_H^1(\eta_*,\eta),T-T_*/10-3d_H^1(\eta_*,\eta)}$ and $\beta-\alpha \geq (25K\varepsilon_{k_\eta})^{\chi/\chi'}$, the path $\sigma_\eta^*\lvert_{[\alpha,\beta]}$ must intersect $\xi_\eta^{l_1},\xi_\eta^{l_2}$ for some $l_1,l_2\in [T_*/10+3d_H^1(\eta_*,\eta),T-T_*/10-3d_H^1(\eta_*,\eta)]$ satisfying $l_2-l_1\geq 4K\varepsilon_{k_\eta}$.
  \end{enumerate}

\end{lemma}
\begin{proof}
  We begin by proving the first statement. %
  Define $l_1'=\sup\{l: \sigma_\eta^*\lvert_{[\alpha,\beta]} \textrm{ is to the right of } \xi_\eta^l\}$\footnote{Here, by being to the right, we mean ``weakly'' to the right in the sense that $\sigma_\eta^*\lvert_{[\alpha,\beta]}$ is allowed to overlap with $\xi_\eta^l$ on a segment but is not allowed to cross into the region to the left of $\xi_\eta^l$.} and $l_2'=\inf\{l: \sigma_\eta^*\lvert_{[\alpha,\beta]} \textrm{ is to the left of }\xi_\eta^l \}$ and note that, since $\sigma_\eta^*\lvert_{[\alpha,\beta]}\subseteq \overline{V_\eta}$, we have
  \begin{equation}
    \label{eq:80}
    l_1',l_2'\in [T_*/10,T-T_*/10].
  \end{equation}
  Fix $\delta \in (0,1)$ sufficiently small (eventually we will take $\delta \to 0$).
  Now, consider the region $W_{l_1'-\delta,l_2'+\delta}$ and consider the four curves bounding it. Note that we must have $\sigma_\eta^*\lvert_{[\alpha,\beta]}\subseteq W_{l_1'-\delta,l_2'+\delta}$. By Lemma \ref{lem:14}, we know that
  \begin{equation}
    \label{eq:81}
   \ell(\xi_\eta^{l_1'-\delta};D_h), \ell(\xi_\eta^{l_2'+\delta};D_h)\leq d_H^1(\eta_*,\eta).
  \end{equation}
  Further, it is trivial that $\ell(\eta\lvert_{[l_1'-\delta,l_2'+\delta]};D_h)= l_2'-l_1'+2\delta$. Finally, by \eqref{eq:75}, it follows that $\ell(\eta_*\lvert_{[b_\eta^{l_1'-\delta},b_\eta^{l_2'+\delta}]};D_h)\leq 2d_H^1(\eta_*,\eta)+ l_2'-l_1'+2\delta$. By H\"older continuity (Lemma~\ref{lem:holder_regularity}), we obtain that for $\delta \in (0,1)$ and $d_H^1(\eta_* , \eta)$ sufficiently small, the Euclidean diameter of the union of all the four curves bounding $W_{l_1' - \delta , l_2' + \delta}$ is at most $(4d_H^1(\eta_*,\eta)+2l_2'-2l_1'+4\delta)^{1/\chi'}$. As a consequence, the entire region $W_{l_1'-\delta,l_2'+\delta}$ lies in an Euclidean ball of radius $(4d_H^1(\eta_*,\eta)+2 l_2'-2 l_1'+4\delta)^{1/\chi'}$. Again, by using H\"older continuity, we have $\beta-\alpha=\diam_{D_h}(\sigma_\eta^*\lvert_{[\alpha,\beta]})\leq \diam_{D_h}(W_{l_1'-\delta,l_2'-\delta})\leq (4d_H^1(\eta_*,\eta)+2l_2'-2l_1'+4\delta)^{\chi/\chi'}$. Taking $\delta \to 0$ gives that
  \begin{align*}
     25 d_H^1(\eta_* , \eta)\leq (\beta-\alpha)^{\chi'/\chi} \leq 4 d_H^1(\eta_* , \eta) + 2(l_2' - l_1')
  \end{align*}
  and hence $l_2' - l_1' \geq \frac{21}{2} d_H^1(\eta_* , \eta)$. Note that the definitions of $l_1'$ and $l_2'$ imply that for all $\varepsilon \in (0,\frac{l_2' - l_1'}{2})$, there exist $l_1 \in [l_1' , l_1' + \varepsilon], l_2 \in [l_2' - \varepsilon , l_2']$ such that $\sigma_{\eta}^*|_{[\alpha , \beta]}$ intersects both $\xi_{\eta}^{l_1}$ and $\xi_{\eta}^{l_2}$. This completes the proof of \eqref{it:1.1}. 

  The proof of \eqref{it:1.2} proceeds similarly. Indeed, we define $l_1',l_2'$ in the same manner and now have $l_1',l_2'\in [T_*/10+3d_H^1(\eta_*,\eta),T-T_*/10-3d_H^1(\eta_*,\eta)]$ instead of \eqref{eq:80}. Finally, \eqref{eq:81} is replaced by the better bound
  \begin{equation}
    \label{eq:82}
     \ell(\xi_\eta^{l_1'-\delta};D_h), \ell(\xi_\eta^{l_2'+\delta};D_h)\leq K \varepsilon_{k_\eta}
   \end{equation}
   for all $\delta \in (0,1)$ sufficiently small coming from Lemma \ref{lem:9}. The remainder of the proof proceeds verbatim.
 \end{proof}
 We are now finally ready to complete the proof of Proposition~\ref{prop:close_one_sided_they_intersect}.
\begin{proof}[Proof of Proposition~\ref{prop:close_one_sided_they_intersect}]
\emph{Step 1. Outline and setup.}
Before beginning, we refer the reader to Figure \ref{fig:fitchi} for a visual depiction of the proof. Now, consider the path $\xi_\eta^{T_*/2}$. Note that either every path in $\overline{\wt{U}_{\eta}}\cup \eta\cup \eta_*$ connecting $z_{\eta}$ to $\{\eta(0) , \eta_*(0)\}$ intersects $\xi_{\eta}^{T_*/2}$, or every path in $\overline{\wt{U}_{\eta}}\cup \eta\cup\eta_*$ connecting $z_{\eta}$ to $\{\eta(T),\eta_*(T_*)\}$ intersects $\xi_{\eta}^{T_*/2}$, and without loss of generality, let us assume the former.
  By Lemma \ref{lem:geodesic_between_geodesics}, at least one of \eqref{it:geodnostar}, \eqref{it:geodstar} therein need to hold. Without loss of generality, we assume that the latter holds, that is there exists a geodesic $\sigma_\eta^*$ from $z_\eta$ to $\eta_*(0)$ which is contained in $\overline{\wt{U}_\eta}\cup \eta_*$. As a consequence of the above assumptions, we know that
  \begin{equation}
    \label{eq:79}
    D_h(z_\eta, \xi_\eta^{T_*/5})\geq T_*/2-T_*/5 - 2d_H^1(\eta_*,\eta)>T_*/10.
  \end{equation}

  We now choose a rational $\alpha$ such that $\alpha\leq T_* / 10$ and let
  \begin{equation}
    \label{eq:85}
    0<\bphi<\bpsi<\alpha/20<T_*/200
  \end{equation}
  be the quantities obtained by applying Proposition \ref{prop:9} with $U$ and $\alpha$. The quantities $\bphi$ and $\bpsi$ will be used throughout the proof; we emphasize that $\bphi$ and $\bpsi$ do not depend on $\eta$.

  Let us now give a brief outline of the strategy of the proof. In Step 2, we will show that $\sigma_{\eta}^*|_{[\bphi / 2 , 2 \bpsi]}$ lies between two ladders far away from the ladders $\xi_\eta^{T_*/10},\xi_\eta^{T-T_*/10}$ bounding $V_{\eta}$. In Step 3, we will use Proposition~\ref{prop:9} to construct a $\scX$ that $\sigma_{\eta}^*|_{[\bphi / 2 , 2\bpsi]}$ passes through such that its common segment $P_+\cap P_-$ lies in a region $W_{\rho^+ , \rho^-}$, with $|\rho^--\rho^+|$ being considerably large. Next, in Step 4, we will show that $W_{\rho^+,\rho^-}$ is contained in a small neighborhood of $\sigma_{\eta}^*$ with Euclidean size of order $\varepsilon_{k_{\eta}}^{1 / \chi'}$. Finally, in Step 5, we will show that the arms of the $\scX$ exit $W_{\rho^+,\rho^-}$ and use this along with the uniqueness of the $D_h$-geodesics forming the $\scX$ to complete the proof of the proposition.

\emph{Step 2. $\sigma_{\eta}^*|_{[\bphi / 2 , 2\bpsi]}$ lies between two ladders far away from the two ladders bounding $V_{\eta}$.}
Note that since $\sigma_{\eta}^*\subseteq \overline{\wt{U}_\eta}\cup \eta_*$ is a $D_h$-geodesic, we have by the rightmost nature of the paths $\xi_{\eta}^l$ that $\sigma_{\eta}^*$ does not cross $\xi_{\eta}^{T - T_* / 10}$ and hence it stays to the left of $\xi_{\eta}^{T - T_* / 10}$. Also \eqref{eq:79} implies that if $d_H^1(\eta_* , \eta)$ is sufficiently small, then $\sigma_{\eta}^*|_{[0,(25 d_H^1(\eta_* , \eta))^{\chi / \chi'}]}$ stays to the right of $\xi_{\eta}^{T_* / 5}$. In particular, we have that $\sigma_{\eta}^*|_{[0,(25 d_H^1(\eta_* , \eta))^{\chi / \chi'}]} \subseteq \overline{V_{\eta}}$. Thus, \eqref{it:1.1} in Lemma~\ref{lem:12} implies that the path $\sigma_\eta^*\lvert_{[0, (25d_H^1(\eta_*,\eta))^{\chi/\chi'}]}$ intersects $\xi_\eta^l$ for some $l\leq T-T_*/10-4d_H^1(\eta_*,\eta)$. Moreover, the rightmost nature of $\xi_{\eta}^l$ implies that $\sigma_{\eta}^*$ has to stay to the left of $\xi_{\eta}^l$ after it intersects it for the first time. In particular, by choosing $d_H^1(\eta_* , \eta)$ sufficiently small such that $(25 d_H^1(\eta_* , \eta))^{\chi / \chi'} < \frac{\bphi}{2}$, we obtain that $\sigma_{\eta}^*|_{[\bphi / 2 , 2 \bpsi]}$ stays to the left of $\xi_{\eta}^{T - T_* / 10 - 4 d_H^1(\eta_* , \eta)}$. Thus, combining with \eqref{eq:79} and \eqref{eq:85}, we obtain that the path $\sigma_{\eta}^*|_{[\bphi / 2 , 2 \bpsi]}$ lies between $\xi_{\eta}^{T_* / 5}$ and $\xi_{\eta}^{T - T_* / 10 - 4 d_H^1(\eta_* , \eta)}$. In particular, we have
  \begin{equation}
    \label{eq:83}
    \sigma_\eta^*\lvert_{[\bphi/2,2\bpsi]} \subseteq W_{T_*/10+3d_H^1(\eta_*,\eta),T-T_*/10-3d_H^1(\eta_*,\eta)}.
  \end{equation}

  \end{proof}
\begin{figure}
  \centering
  \includegraphics[width=\linewidth]{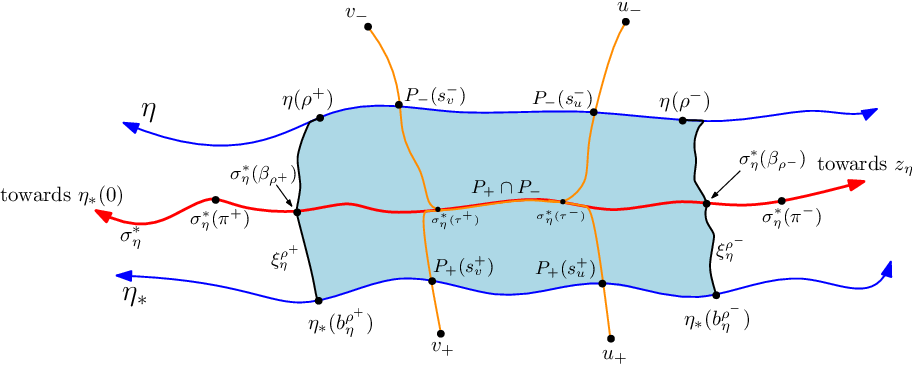}
  \caption{Illustration of the $\scX$ constructed in the proof of Proposition~\ref{prop:close_one_sided_they_intersect} using Proposition~\ref{prop:9}: The light blue region is the interior of the set $W_{\rho^+,\rho^-}$. We show that $\sigma_\eta^*$ (red) passes through one of the $\scX$s obtained using Proposition~\ref{prop:9} whenever it crosses the blue region; note that $\sigma_\eta^*$ is parametrised as a path \emph{from} $z_\eta$ \emph{to} $\eta_*(0)$. The LQG size of the $\scX$ is chosen so that the LQG size of its arms (orange paths) is larger than the Hausdorff LQG distance between the bottom and top boundaries of $W_{\rho^+,\rho^-}$.  This forces the orange paths to exit $W_{\rho^+,\rho^-}$. Further, the $\scX$ is chosen so that its LQG distance to the ladders $\xi_\eta^{\rho^+},\xi_{\eta}^{\rho^-}$ bounding $W_{\rho^+,\rho^-}$ is much larger than its LQG size. This implies that the orange paths have to exit $W_{\rho^+,\rho^-}$ and need to do so through either its top or bottom boundary.  Combining with the uniqueness of the $D_h$-geodesics $P_+,P_-$ used to form the $\scX$, along with the fact that $P_+$ (resp.\ $P_-$) lies to the left (resp.\ right) of $\sigma_\eta^*$, we obtain that the top and bottom boundaries of $W_{\rho^+,\rho^-}$ have to contain the common segment $P_+\cap P_-$ of the geodesics forming the $\scX$, and must thereby overlap on a segment.}
  \label{fig:fitchi}
\end{figure}

 \emph{Step 3. Constructing a $\scX$ that $\sigma_{\eta}^*|_{[\bphi / 2 , 2\bpsi]}$ passes through and a region $W_{\rho^+,\rho^-}$ containing $P_+\cap P_-$.}
 Now, we simply apply Proposition \ref{prop:9} and this yields $(z,r)\in  \CC\times [\varepsilon_{k_{\eta}}^{1/(2\chi')},\varepsilon_{k_{\eta}}^{1/ (4\chi')}]$ such that the conditions therein hold. With $u_-,u_+,v_-,v_+$ being the points associated to $\scX_{u_-,v_-}^{u_+,v_+}$, let $\tau^-<\tau^+$ be such that $P_-\cap P_+=\sigma_\eta^*\lvert_{[\tau^-,\tau^+]}$ and note that by condition \eqref{it:overlapping_geodesics} in Proposition \ref{prop:9}, $[\tau^-,\tau^+]\subseteq [\bphi,\bpsi]$. %
 Define the times $\pi^-,\pi^+$ by%
\begin{equation}
  \label{eq:65}
  \pi^-=\tau^--2(25 K\varepsilon_{k_{\eta}})^{\chi/(4\chi')}, \pi^+=\tau^++2(25 K\varepsilon_{k_{\eta}})^{\chi/(4\chi')},
\end{equation}
and note that since $d_H^1(\eta_* , \eta)$ is taken to be small, we can assume that $[\pi^-,\pi^+]\subseteq [\bphi/2,2\bpsi]$. Now, as a consequence of \eqref{eq:83} and assertion \eqref{it:1.2} in Lemma \ref{lem:12}, we know that $\sigma_\eta^*\lvert_{[\pi^-, \tau^--(25 K\varepsilon_{k_\eta})^{\chi/(4\chi')}]}$ and $\sigma_\eta^*\lvert_{[\tau^++(25 K\varepsilon_{k_\eta})^{\chi/(4\chi')},\pi^+]}$ must each intersect some ladders, say, $\xi_{\eta}^{\rho^-},\xi_{\eta}^{\rho^+}$ respectively satisfying
\begin{equation}
  \label{eq:84}
  T_*/10+3d_H^1(\eta_*,\eta)\leq \rho^+ < \rho^-\leq T-T_*/10-3d_H^1(\eta_*,\eta).
\end{equation}
Let $\beta_{\rho^-}$ be the last time that $\sigma_\eta^*\lvert_{[\pi^-, \tau^--(25 K\varepsilon_{k_\eta})^{\chi/(4\chi')}]}$ intersects $\xi_{\eta}^{\rho^-}$ and let $\beta_{\rho^+}$ be the first time that $\sigma_\eta^*\lvert_{[\tau^++(25 K\varepsilon_{k_\eta})^{\chi/(4\chi')},\pi^+]}$ intersects $\xi_{\eta}^{\rho^+}$. %
Since $\sigma_{\eta}^*$ stays to the left of $\xi_{\eta}^l$ after the first time that it intersects it, $\sigma_\eta^*\lvert_{[\beta_{\rho^-},\beta_{\rho^+}]}$ must lie in between $\xi_\eta^{\rho^-}$ and $\xi_\eta^{\rho^+}$, thereby implying that $\sigma_\eta^*\lvert_{[\beta_{\rho^-},\beta_{\rho^+}]}\subseteq W_{\rho^-,\rho^+}$. Also, by definition, it is clear that we must have
\begin{equation}
  \label{eq:76}
  \bphi/2 \leq\pi^-\leq \beta_{\rho^-}\leq \tau^--(25 K\varepsilon_{k_\eta})^{\chi/(4\chi')}<\tau^++(25 K\varepsilon_{k_\eta})^{\chi/(4\chi')}\leq \beta_{\rho^+}\leq \pi^+\leq 2\bpsi.
\end{equation}
Now, \eqref{eq:84} and Lemma \ref{lem:9} together imply that the ladders $\xi_\eta^l$ for $l\in [\rho^+,\rho^-]$ satisfy
\begin{equation}
  \label{eq:87}
  \ell(\xi_\eta^l;D_h)< K\varepsilon_{k_\eta}.
\end{equation}

Note that as a consequence of condition \eqref{it:endpoints_of_arms_away_from_geodesic} in Proposition \ref{prop:9}, we have
\begin{equation}
  \label{eq:66}
  B_{16(K\varepsilon_{k_\eta})^{1/\chi'}}(\sigma_\eta^*\lvert_{[\pi^-,\pi^+]})\cap \{u_-,u_+,v_-,v_+\}=\emptyset.
\end{equation}

\emph{Step 4. $W_{\rho^+,\rho^-}$ is contained in a small Euclidean neighborhood of $\sigma_{\eta}^*|_{[\bphi / 2 , 2\bpsi]}$.}
  Now, consider the region $W_{\rho^+,\rho^-}$. The plan now is to show that %
  \begin{equation}
    \label{eq:48}
   W_{\rho^+,\rho^-}\subseteq B_{16(K\varepsilon_{k_{\eta}})^{1/\chi'}}( \sigma_\eta^*\lvert_{[\pi^-,\pi^+]}).
 \end{equation}
 To do so, we alternatively prove the following-- we prove that for any $l_1,l_2$ such that $\rho^{+}\leq l_1\leq l_2\leq \rho^-$ additionally satisfying $|l_2-l_1|\leq \varepsilon_{k_\eta}$, the region $W_{l_1,l_2}$ %
 must satisfy
 \begin{equation}
   \label{eq:71}
   W_{l_1,l_2}\subseteq B_{16(K\varepsilon_{k_{\eta}})^{1/\chi'}}( \sigma_\eta^*\lvert_{[\pi^-,\pi^+]}).
 \end{equation}
  Note that $W_{\rho^+,\rho^-}$ can be written as a union of finitely many such $W_{l_1,l_2}$s and thus it is sufficient to prove \eqref{eq:71}.
 
  First, by employing \eqref{eq:70} and \eqref{eq:87}, and by using $\diam(\cdot)$ to refer to the Euclidean diameter of a set, we have
  \begin{equation}
    \label{eq:49}
    \xi_\eta^{l_1}, \xi_\eta^{l_2}\subseteq B_{2(K\varepsilon_{k_\eta})^{1/\chi'}}(\sigma_\eta^*\lvert_{[\beta_{\rho^+},\beta_{\rho^-}]})\subseteq B_{2(K\varepsilon_{k_\eta})^{1/\chi'}}(\sigma_\eta^*\lvert_{[\pi^-,\pi^+]}) \textrm{ and } \diam(\xi_\eta^{l_1}),\diam(\xi_\eta^{l_2})\leq (K\varepsilon_{k_\eta})^{1/\chi'}.
  \end{equation}
  Further, by using that $|l_2-l_1|\leq \varepsilon_{k_\eta}$ along with \eqref{eq:70} and H\"older continuity (see Lemma~\ref{lem:holder_regularity}), we also have
  \begin{equation}
    \label{eq:72}
    \eta\lvert_{[l_1, l_2]}\subseteq B_{3(K\varepsilon_{k_\eta})^{1/\chi'}}(\sigma_\eta^*\lvert_{[\pi^-,\pi^+]})\textrm{ and } \diam(\eta\lvert_{[l_1, l_2]})\leq \varepsilon_{k_\eta}^{1/\chi'}.
  \end{equation}

  Finally, we come to the curve $\eta_*\lvert_{[b^{l_1}_\eta,b^{l_2}_\eta]}$. As a consequence of \eqref{eq:75}, \eqref{eq:87} and $|l_2-l_1|\leq \varepsilon_{k_\eta}$, we have $b_\eta^{l_2}-b_\eta^{l_1}\leq 3K\varepsilon_{k_\eta}$. Combining the above with H\"older continuity (Lemma~\ref{lem:holder_regularity}) and \eqref{eq:49}, we obtain that
  \begin{equation}
    \label{eq:73}
    \eta_*\lvert_{[b_\eta^{l_1}, b_\eta^{l_2}]}\subseteq B_{4(K\varepsilon_{k_\eta})^{1/\chi'}}( \sigma_\eta^*\lvert_{[\pi^-,\pi^+]}) \textrm{ and } \diam(\eta_*\lvert_{[b_\eta^{l_1}, b_\eta^{l_2}]})\leq 3(K\varepsilon_{k_\eta})^{1/\chi'}.
  \end{equation}

As a result of the above, we have shown that all the four bounding curves of $W_{l_1,l_2}$ are contained in $B_{4(K\varepsilon_{k_\eta})^{1/\chi'}}( \sigma_\eta^*\lvert_{[\pi^-,\pi^+]})$ and that each of their Euclidean diameters is at most $3(K\varepsilon_{k_\eta})^{1/\chi'}$. Fix an arbitrary point $\omega$ lying on the boundary curves; as a consequence of the above, all these bounding curves lie in $B_{12(K\varepsilon_{k_\eta})^{1/\chi'}}(\omega)$. Since $B_{12(K\varepsilon_{k_\eta})^{1/\chi'}}(\omega)$ is convex, we must also have $W_{l_1,l_2}\subseteq B_{12(K\varepsilon_{k_\eta})^{1/\chi'}}(\omega)$. However, we have already shown that $\omega\in B_{4(K\varepsilon_{k_\eta})^{1/\chi'}}( \sigma_\eta^*\lvert_{[\pi^-,\pi^+]})$ and as a result, the triangle inequality immediately yields the desired inclusion \eqref{eq:71} and thereby also proves \eqref{eq:48}.

 \emph{Step 5. Conclusion of the proof by proving that the arms of the $\scX$ exit $W_{\rho^+,\rho^-}$.}
  Now, by condition \eqref{it:chiinside} in Proposition \ref{prop:9}, we know that
  \begin{equation}
    \label{eq:78}
    P_+\cup P_-\subseteq B_{\varepsilon_{k_\eta}^{1/(4\chi')}}(P_+\cap P_-)=B_{\varepsilon_{k_\eta}^{1/(4\chi')}}(\sigma_\eta^*\lvert_{[\tau^-,\tau^+]}).
  \end{equation}
    By \eqref{eq:76}, we know that $\beta_{\rho^+}-\tau^+\geq (25K\varepsilon_{k_\eta})^{\chi/(4\chi')}$ and that $\tau^--\beta_{\rho^-}\geq (25K\varepsilon_{k_\eta})^{\chi/(4\chi')}$ and thus by the triangle inequality along with \eqref{eq:87}, we must have that
    \begin{align*}
        D_h(\sigma_{\eta}^*|_{[\tau^- , \tau^+]} , \xi_{\eta}^{\rho^-} \cup \xi_{\eta}^{\rho^+}) &\geq (25 K \varepsilon_{k_{\eta}})^{\chi / (4 \chi')} - \ell(\xi_{\eta}^{\rho^-} ; D_h) - \ell(\xi_{\eta}^{\rho^+} ; D_h)\\
        &\geq (25 K \varepsilon_{k_{\eta}})^{\chi / (4 \chi')} - 2K \varepsilon_{k_{\eta}} \geq (24 K \varepsilon_{k_{\eta}})^{\chi / (4 \chi')}.
    \end{align*}
    By H\"older continuity (Lemma~\ref{lem:holder_regularity}), we obtain that if $d_H^1(\eta_* , \eta)$ is sufficiently small, we have
    \begin{align*}
        \text{dist}(\sigma_{\eta}^*|_{[\tau^- , \tau^+]} , \xi_{\eta}^{\rho^-} \cup \xi_{\eta}^{\rho^+}) \geq (24 K \varepsilon_{k_{\eta}})^{1 / (4 \chi')}.
    \end{align*}
    Thus, on combining with \eqref{eq:78}, we obtain that
   \begin{equation}
      \label{eq:77}
      (P_+\cup P_-) \cap (\xi_\eta^{\rho^-}\cup \xi_\eta^{\rho^+})=\emptyset.
    \end{equation}

Combining \eqref{eq:66} with \eqref{eq:48},  we obtain that $u_-,u_+,v_-,v_+\notin W_{\rho^+,\rho^-}$,  and thus the geodesics $P_-,P_+$ must exit $W_{\rho^+,\rho^-}$ (recall that $P_+\cap P_-\subseteq W_{\rho^+,\rho^-}$). However, due to \eqref{eq:77}, neither of the geodesics $P_+,P_-$ can exit $W_{\rho^+,\rho^-}$ via $\xi_\eta^{\rho^-}$ or $\xi_\eta^{\rho^+}$. As a result, the geodesics $P_+, P_-$ must exit $W_{\rho^+,\rho^-}$ via $\eta_*\lvert_{[b_\eta^{\rho^+},b_\eta^{\rho^-}]}  \cup \eta\lvert_{[\rho^+,\rho^-]}$. Now, recall from condition \eqref{it:leftright} in Proposition \ref{prop:9} that $P_+$ is to the left of $\sigma_{\eta}^*$ when $\sigma_{\eta}^*$ is viewed as a path from $z_{\eta}$ to $\eta_*(0)$ while $P_-$ is to its right (see Figure \ref{fig:fitchi}). As a result, there must exist times
\begin{align}
  \label{eq:56}
  &0<s_u^+<\tau_{u}^+<\tau_v^+<s_v^+<D_h(u_+,v_+),\nonumber\\
  &0<s_u^-<\tau_{u}^-< \tau_v^-<s_v^-<D_h(u_-,v_-).
\end{align}
such that $P_+(s_u^+), P_+(s_v^+)\in \eta_*\lvert_{[b_\eta^{\rho^+},b_\eta^{\rho^-}]}$ and $P_-(s_u^-), P_+(s_v^-)\in \eta \lvert_{[\rho^+,\rho^-]}$. Finally, recall from condition \eqref{it:overlapping_geodesics} in Proposition \ref{prop:9} that $P_+$ and $P_-$ are the unique $D_h$-geodesics between their endpoints. As a result, we must necessarily have
\begin{equation}
  \label{eq:57}
  P_+\cap P_-\subseteq \eta_* \lvert_{[b_\eta^{\rho^+},b_\eta^{\rho^-}]} \cap \eta \lvert_{[\rho^+, \rho^-]}\subseteq \eta\cap \eta_*.
\end{equation}
This shows that the intersection $\eta_*\cap \eta$ is non-empty as long as $d_H^1(\eta_*,\eta)$ is small enough, thereby completing the proof.

\subsection{Strong confluence for the (two-sided) Hausdorff distance}
\label{sec:strong-confl-two}

The goal of this section is to complete the proof of Theorem \ref{thm:3}. First, we shall prove the result on the non-existence of geodesic bubbles (Proposition \ref{prop:11}) and then subsequently use it to prove Theorem \ref{thm:3}. To prove Proposition \ref{prop:11}, we shall also require some additional known results that we now state. First, we require the following deterministic bound from \cite{Gwy21} on the number of distinct $D_h$-geodesics between two points.

\begin{proposition}(\cite[Theorem~1.7]{Gwy21})\label{prop:finitely_many_geodesics}
There is a finite,  deterministic number $m = m(\gamma) \in \mathbb{N}$ such that a.s.\  any two points in $\mathbb{C}$ are joined by at most $m$ distinct $D_h$-geodesics.
\end{proposition}

Further, we require the non-existence of geodesic bubbles for geodesics starting from a \emph{fixed} point.

\begin{proposition}(\cite[Lemma 3.9]{GPS20})
\label{prop:geodesics_don't_form_bubbles}
Fix $\bz\in \CC$. Almost surely, simultaneously for all $z \in \mathbb{C}$, the following holds. For every $D_h$-geodesic $P$ from $\bz$ to $z$ and any $t\in (0,D_h(\bz,z))$, $P\lvert_{[0,t]}$ is the unique $D_h$-geodesic between its endpoints.
\end{proposition}

We shall also require the following simple but useful lemma that is implicitly shown in \cite{MQ20}.
\begin{lemma}\label{prop:geodesics_intersect}
  It is almost surely the case that the following is true. Let $P_n\colon [0,T_n]\rightarrow \CC$ be a sequence of $D_h$-geodesics converging to some $D_h$-geodesic $P \colon [ 0 , T]\rightarrow \CC$ as $n \to \infty$ in the Hausdorff sense with respect to $D_h$. Let $0<s<t<T$. Then, given $\varepsilon>0$, for all $n$ large enough, we have at least one of the following:
  \begin{enumerate}[(a)]
  \item \label{it:cr1} $d_H^1(P\lvert_{[s,t]},P_n\lvert_{[s,t]})\leq \varepsilon$,
  \item \label{it:cr2} $P\lvert_{[s,t]}\cap P_n\lvert_{[s,t]} \neq \emptyset$.
  \end{enumerate}
\end{lemma}
\begin{proof}[Proof sketch]
  A proof of the above is implicitly present in \cite[Proof of Proposition 5.1]{MQ20} and now only provide a rough sketch. Using that $\lim_{n\rightarrow \infty}d_H(P_n,P)=0$, one can first show that we necessarily have
  \begin{equation}
    \label{eq:96}
    \lim_{n\rightarrow \infty}d_H(P_n\lvert_{[s,t]},P\lvert_{[s,t]})=0.
  \end{equation}
  Now choose a $\delta$ small enough such that %
 there is no path $\eta\subseteq \cB_\delta(P\lvert_{[s,t]})\setminus P$ connecting the $\delta$ radius $D_h(\cdot,\cdot;\CC\setminus P)$ neighbourhood of $P^{\mathrm{L}}\lvert_{[s,t]}$ to the $\delta$ radius $D_h(\cdot,\cdot;\CC\setminus P)$ neighbourhood of $P^{\mathrm{R}}\lvert_{[s,t]}$. By \eqref{eq:96}, we must have $P_n\lvert_{[s,t]}\subseteq \cB_\delta(P\lvert_{[s,t]})$ for all $n$ large enough. As a result, for all $n$ large enough, either $P_n\lvert_{[s,t]}$ must necessarily cross $P\lvert_{[s,t]}$ implying \eqref{it:cr2} or $d_H^1(P_n\lvert_{[s,t]},P\lvert_{[s,t]})$ must be small, thereby implying \eqref{it:cr1}.
\end{proof}

On being combined with Proposition \ref{prop:close_one_sided_they_intersect}, the above lemma yields the following proposition.
\begin{proposition}
  \label{prop:12}
  It is almost surely the case that the following is true. Let $P_n:[0,T_n]\rightarrow \CC$ be a sequence of $D_h$-geodesics converging to some $D_h$-geodesic $P:[0,T]\rightarrow \CC$ in the Hausdorff sense. Let $0\leq s<t\leq T$. Then for all $n$ large enough, we have $P_n\lvert_{[s,t]}\cap P\lvert_{[s,t]}\neq \emptyset$.
\end{proposition}

We now use the above to provide the proofs of Proposition \ref{prop:11} and Theorem \ref{thm:3}.
\begin{proof}[Proof of Proposition \ref{prop:11}]
The proof follows the same argument as in the proof of \cite[Proposition~5.1]{MQ20} except that Proposition~\ref{prop:12} is used instead of \cite[Proposition~3.1]{MQ20},  and Propositions~\ref{prop:finitely_many_geodesics} and ~\ref{prop:geodesics_don't_form_bubbles} are used in place of \cite[Proposition~4.9]{MQ20} and \cite[Corollary~7.7]{LeGal10} respectively.  

More precisely, suppose that there exists another $D_h$-geodesic $\wt{P}\colon [0,T]\rightarrow \CC$ from $P(0)$ to $P(T)$ such that $\wt{P}\neq P$ but $P(r)=\wt{P}(r)$ for all $r\in [s,t]$.  Fix $s_0 \in (0,s)$ and let $\{z_k\}$ be a sequence such that $z_k \in 2^{-k} \mathbb{Z}^2$ for each $k$ and $z_k \to P(s_0)$ as $k \to \infty$.  For each $k$,  we let $P_k \colon [0,T_k] \to \mathbb{C}$ be a $D_h$-geodesic from $z_k$ to $P(T)$.  By possibly passing to a further subsequence,  we can assume that $P_k$ converges in the Hausdorff sense to some $D_h$-geodesic $P^{s_0}$.  Suppose that $P^{s_0} = P|_{[s_0,T]}$.  Then Proposition~\ref{prop:12} implies that we must have that $P_k \cap P|_{[s_0 ,  s]} \neq \emptyset$ for all $k$ sufficiently large. Then if $u_k,v_k$ are such that $s_0 \leq v_k \leq s$ and $P_k(u_k) = P(v_k)$,  we can construct a new geodesic from $P_k(0)$ to $P(T)$ by concatenating $P_k|_{[0,u_k]} ,  P|_{[v_k ,  s]} ,\wt{P}\lvert_{[s,t]}$ and $P|_{[t,T]}$.  But since we have assumed that $t<T$, this contradicts Proposition \ref{prop:geodesics_don't_form_bubbles} applied for $D_h$-geodesics starting from the rational point $z_k$, and so we obtain that $P^{s_0} \neq P|_{[s_0 ,  T]}$. Analogously, we also obtain that $P^{s_0}\neq \wt{P}\lvert_{[s_0,T]}$.

Finally, let $\{s_k\}$ be a decreasing sequence in $(0,s_0)$ converging to $0$ as $k \to \infty$.  Then, the previous paragraph implies that we can construct, for each $k$, a $D_h$-geodesic $P^{s_k}$ from $P(s_k)$ to $P(T)$ which is not equal to any of $P|_{[s_k ,  T]}$, $\wt{P}|_{[s_k ,  T]}$ and is also not equal to the concatenation of $P|_{[s_k ,  s_j]}$ with $P^{s_j}$ for $1 \leq j \leq k-1$.  This implies that $P(0)$ and $P(T)$ can be connected by infinitely many geodesics and thus violates Proposition~\ref{prop:finitely_many_geodesics}.  Therefore, we have obtained a contradiction, and this completes the proof of the proposition.

\end{proof}
\begin{proof}[Proof of Theorem~\ref{thm:3}]

  Let the points $u,v$, the geodesic $P \colon [0,T] \to \mathbb{C}$ from $u$ to $v$ and the geodesics $P_n \colon [0,T_n] \to \mathbb{C}$ be as in the statement of the theorem. Without loss of generality, we can assume that $\varepsilon \in (0,T/2)$. Since $\lim_{n\rightarrow \infty}d_H(P_n,P)=0$, it can be checked that $\lim_{n\rightarrow \infty}d_H(P_n\lvert_{[0,\varepsilon]},P\lvert_{[0,\varepsilon]})=0$ and $\lim_{n\rightarrow \infty}d_H(P_n\lvert_{[T_n-\varepsilon,T_n]},P\lvert_{[T-\varepsilon,T]})=0$. On combining this with Proposition \ref{prop:12}, we obtain that for all $n$ large enough, we must have
\begin{align*}
P\lvert_{[0,\varepsilon]} \cap P_n\lvert_{[0,\varepsilon]} \neq \emptyset \quad \text{and} \quad P\lvert_{[T-\varepsilon,T]} \cap P_n\lvert_{[T_n - \varepsilon ,  T_n]} \neq \emptyset.
\end{align*}
Finally, Proposition~\ref{prop:11} implies that the $D_h$-geodesic from any point of $P\lvert_{[0,\varepsilon]} \cap P_n\lvert_{[0,\varepsilon]}$ to any point of $P\lvert_{[T-\varepsilon ,  T]} \cap P_n\lvert_{[T_n - \varepsilon ,  T_n]}$ is unique and so it follows that
\begin{equation}
  \label{eq:113}
  P\lvert_{[\varepsilon ,  T- \varepsilon]},P_n\lvert_{[\varepsilon,T_n-\varepsilon]}\subseteq P_n\cap P
\end{equation}
for all $n \geq n_0$. Further, by choosing $n_0$ to be large enough, we can ensure that
\begin{equation}
  \label{eq:112}
  P_n(0)\in \cB_{\varepsilon}(u), P_n(T_n)\in \cB_{\varepsilon}(v)
\end{equation}
for all $n\geq n_0$ as well. This shows that for all $n\geq n_0$, we have
\begin{equation}
  \label{eq:114}
      (P_n\setminus P)\cup (P\setminus P_n)\subseteq \cB_{2\varepsilon}(u)\cup \cB_{2\varepsilon}(v),
    \end{equation}
    and replacing $\varepsilon$ by $\varepsilon/2$ now completes the proof.

\end{proof}

\noindent Having proved Proposition \ref{prop:11} and Theorem \ref{thm:3}, it now remains to provide the proofs of Propositions \ref{thm:4}, \ref{prop:15}.
\begin{proof}[Proof of Proposition~\ref{thm:4}]
Let $\{z_n\}_{n \in \mathbb{N}}$ be an enumeration of $\mathbb{Q}^2$ and for all $n , m \in \mathbb{N}$ with $n \neq m$, we let $P_{n,m}$ be the a.s.\ unique $D_h$-geodesic which connects $z_n$ and $z_m$. With $\dim_{D_h}(\cdot)$ denoting the Hausdorff dimension with respect to the LQG metric, note that $\text{dim}_{D_h}(\cup_{i,j} P_{i,j}) = 1$ a.s.\ since $\text{dim}_{D_h}(P_{i,j}) = 1$ a.s.\ for each $i,j \in \mathbb{N}$; indeed, any $D_h$-geodesic, by definition, has dimension $1$ with respect to $D_h$. We will finish the proof by using strong confluence to show that $\cW\subseteq \cup_{i,j} P_{i,j}$.

Let $P \colon [0 , T] \to \mathbb{C}$ be a $D_h$-geodesic and fix $\varepsilon \in (0,T/4)$. Then there exist subsequences $\{z_{n_j}\} , \{z_{m_j}\}$ of $\{z_j\}$ so that $z_{n_j} \to P(\varepsilon)$ and $z_{m_j} \to P(T - \varepsilon)$ as $j \to \infty$. We claim that 
\begin{align*}
    d_H(P_{n_j , m_j} , P|_{[\varepsilon , T - \varepsilon]}) \to 0 \quad \text{as} \quad j \to \infty.
\end{align*}
Indeed, for any subsequence $\{P_{n_k' , m_k'}\}$ of $\{P_{n_k,m_k}\}$, we can pass to a further subsequence $\{P_{n_k'',m_k''}\}$ which converges in the Hausdorff metric to a $D_h$-geodesic $\widehat{P}$ connecting $P(\varepsilon)$ to $P(T - \varepsilon)$. Now, the non-existence of geodesic bubbles (Proposition~\ref{prop:11}) implies that it is a.s.\ the case that $\widehat{P}$ is equal to $P|_{[\varepsilon , T - \varepsilon]}$. Thus we have shown that every subsequence of $\{P_{n_k,m_k}\}$ has a further subsequence which converges to $P|_{[\varepsilon,T-\varepsilon]}$ in the Hausdorff sense and therefore $\{P_{n_k,m_k}\}$ itself converges to $P|_{[\varepsilon,T-\varepsilon]}$ in the Hausdorff sense. Now, Theorem~\ref{thm:3} implies that it is a.s.\ the case that there exists $k_0 \in \mathbb{N}$ such that
\begin{align*}
    P([2 \varepsilon , T - 2\varepsilon]) \subseteq P_{n_k , m_k} \quad \text{for all} \quad k \geq k_0
\end{align*}
and hence 
\begin{align*}
    P([2\varepsilon , T - 2\varepsilon]) \subseteq \cup_{i , j} P_{i,j}.
\end{align*}
This completes the proof of the proposition since $P$ and $\varepsilon>0$ were arbitrary.   
\end{proof}

\begin{proof}[Proof of Proposition \ref{prop:15}]
   It suffices to show that, almost surely, for any sequence of points $u_n\rightarrow P(s)$ and any sequence $v_n\rightarrow P(t)$, for all $n$ large enough, we have $P\lvert_{[s+\varepsilon,t-\varepsilon]}\subseteq \Gamma_{u_n,v_n}$ and $\Gamma_{u_n,v_n}\lvert_{[\varepsilon,D_h(u_n,v_n)-\varepsilon]}\subseteq P$. To see this, we first use the non-existence of geodesic bubbles (Proposition \ref{prop:11}) to obtain that $P\lvert_{[s,t]}$ is the unique $D_h$-geodesic between its endpoints. As a consequence, by a standard Arzela-Ascoli argument, the geodesics $\Gamma_{u_n,v_n}$ must converge to $P\lvert_{[s,t]}$ in the Hausdorff sense, and an application of Theorem \ref{thm:3} now yields the desired result.
\end{proof}
\section{Geodesic stars and networks: open problems and conjectures}
\label{sec:stars}
In this section, we discuss applications of the main results of the paper to the study of geodesic stars and networks. We also list several open problems and conjectures.

\subsection{Geodesic stars}
\label{sec:geodesic-stars}

As mentioned in the introduction, while planar models of random geometry exhibit \emph{geodesic confluence}, there do exist exceptional points around which confluence fails. Indeed, one can define a point $z$ to be a $k$-star if there exist $k$ distinct geodesics $\{P_i\}_{i=1}^k$ emanating from $z$ such that $P_i\cap P_j=\{z\}$ for all $i\neq j$. While one expects most points in the space to be $1$-stars, there could exist a fractal set of $k$-stars for each $k\geq 2$.

The study of geodesic stars was first initiated in Brownian geometry \cite{Mie13}. More recently, the works \cite{LeGal22,MQ20} compute the precise Hausdorff dimension of $k$-stars in Brownian geometry. Phrased in terms of $\sqrt{8/3}$-LQG, they together establish that, almost surely, $\dim_{D_h}($k$\textrm{-star})=5-k$ for $k\leq 5$, with the set being empty for all $k\geq 6$. In parallel, for the directed landscape, \cite{BGH21,GZ22} computes the dimension of the set of points admitting disjoint geodesics to two fixed points. Bringing over the terminology of $k$-stars to the directed landscape, the works \cite{Bha22,Dau23+} analyse geodesic stars in the directed landscape-- the former computes the dimension of $2$-stars which are additionally located on \emph{geodesics}, while the latter computes the precise dimension of $k$-stars across the whole space for all values of $k$. The recent work \cite{BCK25} studies geodesic stars for Kendall's Poisson roads metric, a yet another interesting model of planar random geometry.

For $\gamma$-LQG with general values of $\gamma$, the precise Hausdorff dimensions of the set of $k$-stars remain a mystery; with $h$ being a whole plane GFF, we shall use $\mathfrak{S}_k\subseteq \CC$ to denote the set of $k$-stars. Unlike phenomena like strong confluence, the precise values of the $k$-star dimensions appear to be highly model specific. For example, for Brownian geometry, the proof technique in \cite{MQ20} exploits the presence of independent continuous state branching processes ($3/2$-CSBP) for boundary lengths of filled metric balls when exploring a Brownian sphere via a peeling by layers algorithm, while the computation in \cite{Dau23+} for the directed landscape uses fine probability estimates for pinches in the Airy line ensemble, building on the earlier work \cite{Ham20}. However, as conjectured in \cite{Dau23+} by harnessing a topological result from \cite{Gwy21}, it is expected that $\dim_{D_h}(\mathfrak{S}_3)$ is in fact universal!
\begin{conjecture}[{\cite[Section 1.2]{Dau23+}}]
  \label{conj:dau}
  Fix $\gamma\in (0,2)$. Almost surely, one has $\dim_{D_h}(\mathfrak{S}_3)=2$.
\end{conjecture}
In an upcoming work of the first author with Ewain Gwynne and Brin Harper computing dimension lower bounds for various random fractals associated to LQG, the lower bound in the above conjecture will be proved.  As for values other than $k=3$, there are no other conjectures for the precise values of $\dim_{D_h}(\mathfrak{S}_k)$. A priori, it is possible that $\mathfrak{S}_k\neq \emptyset$ for all values of $k$ as occurs in Euclidean geometry, but this can be ruled out by using the main result of this paper as we now show sketch a proof of.
\begin{proposition}
  \label{prop:13}
  Fix $\gamma\in (0,2)$. There is a deterministic constant $k_\gamma\in \NN$ such that $\mathfrak{S}_{k_\gamma}=\emptyset$ almost surely.
\end{proposition}
\begin{proof}[Proof sketch]
  The broad strategy is to argue that if there are many disjoint geodesics crossing an Euclidean annulus, then two of these geodesics must be close to each other and this is unlikely due to strong confluence; such an approach was also used for Brownian geometry in \cite[Lemma 4.2]{MQ20}.

  Let $d_H$ denote the Hausdorff distance with respect to the Euclidean metric. Recall that Theorem \ref{thm:3} states strong confluence for $D_h$. For this proof, we shall need a version of strong confluence for the metric $D_{h\lvert_W}$ for fixed Jordan domains $W$ which we now briefly discuss. We fix Jordan domains $U\subseteq V\subseteq W$ such that $\mathrm{dist}(\partial U,\partial V),\mathrm{dist}(\partial V,\partial W)>0$ and we will prove \eqref{eq:117} for any $D_{h\lvert_W}$-geodesic $P$ between points $u,v\in U$ which additionally satisfies $P\subseteq V$.

  We claim that, almost surely, for any points $u,v\in U$ and any $D_{h\lvert_W}$-geodesic $P$ from $u$ to $v$ which additionally satisfies $P\subseteq V$ and any $\varepsilon>0$, there must exist a $\delta>0$ such that for any $D_{h\lvert_W}$-geodesic $P'$ with $d_H(P,P')\leq \delta$, we must have
  \begin{equation}
    \label{eq:117}
          (P\setminus P')\cup (P'\setminus P)\subseteq \cB_\varepsilon(u)\cup \cB_\varepsilon(v).
        \end{equation}

For convenience we now fix a Jordan domain $V'$ sandwiched between $V$ and $W$-- that is, we choose $V'$ to have $V\subseteq V'\subseteq W$ with $\dist(\partial V,\partial V')>0, \dist(\partial V',\partial W)>0$. Now, fix $p \in (0,1)$. Since we have that almost surely,
\begin{align*}
\sup_{u,v \in U} D_h(u,v ; V) < \infty, D_h(\partial V',\partial W) > 0,  
\end{align*}
we obtain that there exists $C>1$ large enough depending on $p$ such that the event $E$ defined by%
\begin{align*}
  E=\{ \sup_{u,v \in U} D_h(u,v ; V) \leq C, D_h(\partial V',\partial W) > C^{-1}\}
\end{align*}
satisfies $\PP(E)\geq p$. Let $\phi \in C_0^{\infty}(\mathbb{C})$ be a bump function such that $\phi = 2\log(C) / \xi$ on $V$, $\phi = 0$ on $\mathbb{C} \setminus V'$ and $\phi(x) \in [0,2\log(C) / \xi]$ for all $x\in \CC$. Recall that the fields $h,h-\phi$ are mutually absolutely continuous when viewed modulo additive constants (see \cite[Lemma A.2]{Gwy19}) and hence Theorem~\ref{thm:3} applies with $h-\phi$ in place of $h$.

Suppose that the event $E$ occurs. Let $P$ be a $D_{h|_W}$-geodesic between two points $u,v \in U$ such that $P \subseteq V$. We claim that $P$ is a $D_{h-\phi}$-geodesic as well. Indeed, first we note that any $D_{h-\phi}$-geodesic between $u$ and $v$ stays in $W$. To see this, note that, since $\phi=0$ on $\CC\setminus V'$,
\begin{align*}
    D_{h-\phi}(\partial V',\partial W) = D_h(\partial V',\partial W) > C^{-1}.
\end{align*}
Further, since $\phi\lvert_V$ is identically equal to its maximum value $2\log (C)/\xi$ and since $P\subseteq V$, we have
\begin{align*}
    D_{h-\phi}(u,v) \leq D_{h-\phi}(u,v;V) = D_h(u,v;V) C^{-2} \leq C^{-1}.
\end{align*}
As a result, we obtain that
\begin{align*}
    D_{h-\phi}(u,v) \geq D_{h|_W}(u,v) C^{-2} = \ell(P ; D_h) C^{-2}.
\end{align*}
Moreover, we have that
\begin{align*}
    D_{h-\phi}(u,v) \leq \ell(P;D_{h-\phi})=C^{-2} \ell(P;D_h),
\end{align*}
which implies that $D_{h-\phi}(u,v) = \ell(P;D_{h-\phi})$ and so, on the event $E$, $P$ is also a $D_{h-\phi}$-geodesic instead of only being a $D_{h\lvert_W}$-geodesic. Similarly, on the event $E$, if $\delta>0$ is sufficiently small and $P'$ is a $D_{h|_W}$-geodesic such that $d_H(P,P') \leq \delta$, we have that $P'$ is a $D_{h-\phi}$-geodesic as well. Therefore, applying Theorem~\ref{thm:3} for the field $h-\phi$ and since $\mathbb{P}(E) \geq p$, we obtain that \eqref{eq:117} holds with probability at least $p$. Since $p \in (0,1)$ was arbitrary, it follows that \eqref{eq:117} holds almost surely.

By using \eqref{eq:117} along with an equicontinuity argument, we know that for any $U,V,W$ as in the above and any $\alpha>0$, there exists a random $\delta>0$ such that for all points $u,v\in U$ satisfying
\begin{enumerate}
\item $|u-v|>\alpha$,
\item There exists a $D_{h\lvert_W}$-geodesic $P$ from $u$ to $v$ additionally satisfying $P\subseteq V$,
\end{enumerate}
if we have a $D_{h\lvert_W}$-geodesic $P'$ with $d_H(P,P')\leq \delta$, we must have $P\cap P'\neq \emptyset$.
 Now, recall that if $X$ is a compact metric space, then the space $\mathcal{K}(X)$ of non-empty compact subsets of $X$ equipped with the Hausdorff metric is compact as well. Thus, if we consider the annulus $\mathbb{A}_{1/2,1}(0)$, then for any $\delta>0$, there is a corresponding constant $k'$ such that for any collection of paths $\{\eta_i\}_{i=1}^{i=k'}\subseteq \overline{\mathbb{A}_{1/2,1}(0)}$ from $\partial B_{1/2}(0)$ to $\partial B_{1}(0)$, there must exist two paths $\eta_i,\eta_j$ with $i\neq j$ such that $d_H(\eta_i,\eta_j)\leq \delta$. On combining this with the previous paragraph, we obtain the following. Given any $\mathbbm{p}\in (0,1)$, there exists a $k=k(\mathbbm{p})$ large enough such that
 \begin{equation}
   \label{eq:115}
   \PP\left(\exists \{z_i\}_{i=1}^k\subseteq \partial B_{1/2}(0), \{w_i\}_{i=1}^k\subseteq \partial B_{1}(0) \textrm{ and disjoint } D_{h\lvert_{\mathbb{A}_{1/4,2}(0)}}\textrm{--geodesics } \Gamma_{z_i,w_i}\subseteq \overline{\mathbb{A}_{1/2,1}(0)}\right)\leq \mathbbm{p}.
 \end{equation}
 Crucially, note that $\mathbbm{p}$ above is arbitrary. By using the scaling and translational invariance of $h$, one can define an analogue of the above event for any annulus $\mathbb{A}_{r/4,2r}(z)$ instead of $\mathbb{A}_{1/4,2}(0)$ and we denote this event by $E_{r,z}$. Note that $E_{r,z}$ is measurable with respect to $h\lvert_{\mathbb{A}_{r/4,2r}(z)}$ viewed modulo an additive constant. By a standard iteration argument for the GFF (see \cite[Lemma 3.1]{GM19}), we obtain that given any $\beta>0$, we can choose $\mathbbm{p}$ small enough and thereby $k=k(\mathbbm{p})$ large enough such that for all $z\in \CC$, we have
 \begin{equation}
   \label{eq:120}
   \PP\left(\bigcap_{i=1}^{\log \varepsilon^{-1}}  E_{2^i\varepsilon,z}\right)\leq \varepsilon^\beta.
 \end{equation}

 Finally, it can be checked that if an Euclidean ball $B_{\varepsilon}(z)$ contains a $k$-star $z'$ admitting geodesics $\{P_{i}\}_{i=1}^k$ from points of $\partial B_4(z')$ to $z'$, then the event $\bigcap_{i=1}^{\log \varepsilon^{-1}}  E_{2^i\varepsilon,z}$ must necessarily occur. Thus, by choosing $\beta>2$, we can do a union bound over $O(\varepsilon^{-2})$ many $\varepsilon$ sized balls covering a bounded region, thereby proving that, almost surely, $k$-stars $z'$ as above whose $k$ geodesic arms each have one end point in $\partial B_4(z')$ do not exist. A simple scale invariance argument for the GFF now yields the above-mentioned non-existence when $\partial B_{4}(z')$ is replaced by $\partial B_{4\delta}(z')$ for any fixed $\delta>0$. This proves that $k$-stars almost surely do not exist.

\end{proof}

While the above proposition states that $\mathfrak{S}_{k_\gamma}=\emptyset$, we do not expect the $k_\gamma$ above to be uniformly bounded in $\gamma$. Indeed, we have the following conjecture.
\begin{conjecture}
  \label{conj:starzero}
  Fix $k\in \NN$. Then $\dim_{D_h}(\mathfrak{S}_k)$ is almost surely deterministic and furthermore, we have $\lim_{\gamma \rightarrow 0}\dim_{D_h}(\mathfrak{S}_k)=2$.
\end{conjecture}
Heuristically, the source of the above conjecture is the fact that the $\gamma$-LQG metric converges to the Euclidean metric as $\gamma\rightarrow 0$ (see \cite[Theorem 1.7]{Kav25}) and for the Euclidean metric, every point is an $\infty$-star. Finally, before moving on, we note that a useful consequence of strong confluence (specifically Proposition \ref{prop:15}) is that for any $k\in \NN$ and $z\in \mathfrak{S}_k$, we can find rational points $u_1,\dots u_k$ and $D_h$-geodesics $\{\Gamma_{z,u_i}\}_{i=1}^k$ intersecting each other only at $z$. In particular, this reduces the question of determining $\dim_{D_h}(\mathfrak{S}_k)$ to that of determining the corresponding dimension for the set of points $z$ admitting disjoint (except at $z$) geodesics to \emph{fixed} points $u_1,\dots, u_k$.
\subsection{Geodesic networks}
\label{sec:geodesic-networks}
Recall that almost surely, for fixed $z,w\in \CC$, there is a unique \cite[Theorem 1.2]{MQ18} $D_h$-geodesic $\Gamma_{z,w}$. However, there indeed exist exceptional points $z,w$ admitting multiple such geodesics. The graph formed by the set of all possible geodesics $\Gamma_{z,w}$, known as a ``geodesic network'', can have many different topologies, and it is of interest to classify the set of possible geodesic networks, and for any fixed network, compute the Hausdorff dimension of points $(z,w)\subseteq \CC\times \CC$ for which the geodesics $\Gamma_{z,w}$ give rise to the above network.

The study of geodesic networks in random geometry was initiated in \cite{AKM17} in the context of Brownian geometry, wherein they classified precisely which geodesic networks are a.s.\ dense and which are nowhere dense. Subsequently, the work \cite{MQ20} obtained the conjecturally correct $\sqrt{8/3}$-LQG Hausdorff dimensions of the set of pairs $(z,w)$ for which many of the above networks occur. Thereafter, the work \cite{Gwy21} analysed geodesic networks in $\gamma$-LQG for general values of $\gamma$, completely classifying the dense networks which occur. Finally, the work \cite{Dau23+} analyses geodesic networks in the directed landscape and, in particular, completely classifies and computes the precise Hausdorff dimensions for all possible geodesic networks in the directed landscape.

Both in \cite{MQ20} and \cite{Dau23+}, the corresponding network dimensions are obtained by considering a geodesic network as a concatenation of three portions: the two geodesic stars at its endpoints and an interior network (see Figure \ref{fig:fig14}). Guided by this intuition, \cite[Conjecture 1.6]{Dau23+} states a conjecture regarding which geodesic networks indeed exist in $\gamma$-LQG for general values of $\gamma$ and the corresponding Hausdorff dimension (with respect to $D_h\times D_h$ and in terms of the star dimensions $\dim_{D_h}(\mathfrak{S}_k)$ which are unknown) of the set of pairs $(z,w)\in \CC\times \CC$ for which these networks occur. In particular, they conjecture that $\gamma$-LQG a.s.\ has either $27,28$ or $29$ possible geodesic networks depending on whether $\dim_{D_h}(\mathfrak{S}_4)$ lies in $[0,1), [1,3/2)$ or $[3/2,d_\gamma]$.

\begin{figure}
  \centering
  \includegraphics[width=0.25\linewidth]{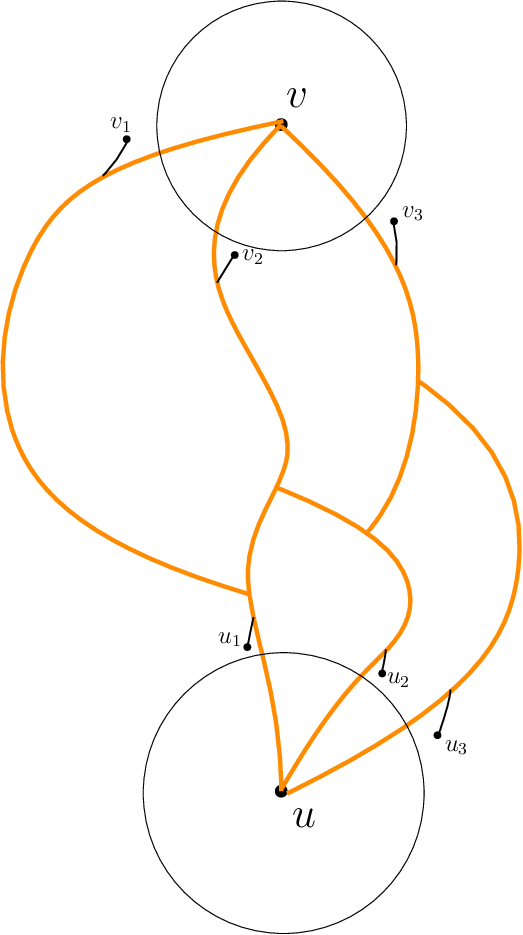}
  \caption{The decomposition of a geodesic network into star events and an interior network from \cite{Dau23+}: Here, the points $u,v$ are both $3$-stars and the points $\{u_i\}_{i=1}^{3}, \{v_j\}_{j=1}^{3}$ are rational points chosen such that the geodesics $\{\Gamma_{u_i,v_j}\}_{i,j}$ approximating the interior geodesic network corresponding to $u$ and $v$. \cite{Dau23+} conjectures that the network as in the above almost surely does exist in $\gamma$-LQG and the Hausdorff dimension with respect to $D_h\times D_h$ of the set of pairs $(u,v)$ as above is almost surely equal to $\dim_{D_h}(\mathfrak{S}_3)+\dim_{D_h}(\mathfrak{S}_3)-4=2+2-4=0$ for all values of $\gamma$. Here, $4$ comes from the number of bounded ``faces'' appearing in the geodesic network (the graph formed by the orange curves) and $2$ is the conjectured a.s.\ value of $\dim_{D_h}(\mathfrak{S}_3)$ (see Conjecture \ref{conj:dau}).} 
  \label{fig:fig14}
\end{figure}
In order to classify geodesic networks in the directed landscape, \cite{Dau23+} crucially uses strong confluence in multiple important ways, and we now state analogues of these for $\gamma$-LQG.
\begin{enumerate}
\item \label{it:dau1} Geodesic networks containing ``interior bubbles'' do not occur.
\item \label{it:dau2} While splitting a geodesic network into two star points and an interior network, the interior network is ``generic'' in the sense that it can be approximated by an interior network occurring between rational points.
\item \label{it:dau3} For disjoint compact sets $K_1,K_2$ with $\mathrm{dist}(K_1,K_2)=d>0$ and for any $0<\alpha<\beta<d$, there exist finitely many pairs of points $(p_i,q_i)\in \{z:\dist(K_1,z)=\alpha\}\times  \{z:\dist(K_1,z)=\beta\}$ and geodesics $\Gamma_{p_i,q_i}\subseteq \{z:\dist(K_1,z)\in [\alpha,\beta]\}$ such that for some point $u\in K_1,v\in K_2$ and geodesic $\Gamma_{u,v}$, we have $\Gamma_{p_i,q_i}\subseteq \Gamma_{u,v}$.
\end{enumerate}
The analogue of item \ref{it:dau3} above in \cite{Dau23+} appears as \cite[Proposition 3.4 (i)]{Dau23+}. We note that all the three items above are indeed consequences of the main result of this paper. Indeed, item \ref{it:dau1} appears as Proposition \ref{prop:11}, item \ref{it:dau2} is a consequence of Proposition \ref{prop:15} and item \ref{it:dau3} can be obtained by first using strong confluence to show that for any $\alpha\in (0,d)$, there are a.s.\ finitely many points in $\{z:\dist(K_1,z)=\alpha\}$ which are last intersection points of geodesics from points in $K_1$ to $K_2$, and by then using the non-existence of geodesic bubbles.

Since the above items \ref{it:dau1}, \ref{it:dau2}, \ref{it:dau3} indeed hold true in $\gamma$-LQG for general values of $\gamma$, one could attempt to classify geodesic networks in $\gamma$-LQG as done in \cite{Dau23+} for the directed landscape. The conjectured formulae involve the star dimensions $\dim_{D_h}(\mathfrak{S}_k)$ which are difficult to compute, but as a first step, one might try to compute the network dimensions in terms of the yet unknown star dimensions.

\begin{problem}
Fix a geodesic network $G$. With respect to the metric $D_h\times D_h$ on $\CC\times \CC$, compute the Hausdorff dimension of the set of pairs $(z,w)\in\CC\times \CC$ such that the possible geodesics $\Gamma_{z,w}$ are arranged according to $G$ in terms of the unknown star dimensions $\{\dim_{D_h}(\mathfrak{S}_k)\}_{k=1}^\infty$\footnote{We note that this would first involve showing that $\dim_{D_h}(\mathfrak{S}_k)$ is deterministic (see Conjecture \ref{conj:starzero}).}.
\end{problem}

For this, one approach is to attempt to carry out the strategy in \cite{Dau23+} by fixing $k$ distinct points $z_1,\dots, z_k\in \CC$, an affine set $A\subseteq \RR^k$ and computing the Hausdorff dimensions
\begin{equation}
  \label{eq:119}
  \dim_{D_h}(\{z: \exists \textrm{ almost}\footnotemark \textrm{ disjoint geodesics }\{\Gamma_{z,z_i}\}_{i=1}^k: (D_h(z,z_i))_{i=1}^k\in A\}).
\end{equation}
\footnotetext{Here ``almost disjoint'' refers to being pairwise disjoint except at $z$.}
It appears plausible that by using bump function arguments, one might be able to compute the optimal upper bounds for \eqref{eq:119} and the network dimensions in terms of the star dimensions (see \cite[Proposition 4.4]{Dau23+} for the directed landscape version). However, obtaining lower bounds on \eqref{eq:119} in terms of $\dim_{D_h}(\mathfrak{S}_k)$, particularly simultaneously over multiple $A\subseteq \RR^k$, appears to be harder (see \cite[Proposition 4.1]{Dau23+} for the directed landscape version).

In \cite{Dau23+}, to rule out the existence of certain networks apart from those with interior bubbles, it is proved (see \cite[Section 5]{Dau23+}) that ``generic'' interior networks as in item \ref{it:dau2} have points of ``multiplicity'' at most $3$. For this, in particular, one has to rule out two geodesics which intersect only at a single point (\cite[Lemma 5.1]{Dau23+}). This is not known for $\gamma$-LQG and, in fact, appears to be unknown even for Brownian geometry, and we now state it as a conjecture.

\begin{conjecture}
  Fix $\gamma\in (0,2)$, a whole plane GFF $h$ and distinct points $z_1,z_2,z_3,z_4\in \CC$. Then almost surely, the intersection $\Gamma_{z_1,z_2}\cap \Gamma_{z_3,z_4}$ cannot be a singleton.
\end{conjecture}

The above question is related to computing the dimension of $k$-stars additionally lying on geodesics since if $\Gamma_{z_1,z_2}\cap \Gamma_{z_3,z_4}$ is a singleton, it must necessarily be a $4$-star which lies in the interior of a geodesic. Star points which lie on geodesics were first investigated in the directed landscape in \cite[Theorem 4]{Bha22}, where the Hausdorff dimension of points on a directed landscape geodesic with one additional geodesic arm emanating out was precisely computed. In \cite[Lemma 5.3]{Dau23+}, it was established that there a.s.\ exist no points on a directed landscape geodesics with two additional geodesic arms emanating out. In view of these developments for the directed landscape, we state the following open question for $\gamma$-LQG.
\begin{problem}
  Fix $\gamma\in (0,2)$ and a whole plane GFF $h$. Fix distinct points $z_1,z_2\in \CC$. Compute the Hausdorff dimensions $\dim_{D_h}(\mathfrak{S}_3\cap \Gamma_{z_1,z_2}),\dim_{\mathrm{Euc}}(\mathfrak{S}_3\cap \Gamma_{z_1,z_2})$. Is the set $\mathfrak{S}_4\cap \Gamma_{z_1,z_2}$ necessarily empty?
\end{problem}
We note that it is easy to see that the set $\mathfrak{S}_3\cap \Gamma_{z_1,z_2}$ is a.s.\ at least countably infinite since any point where $\Gamma_{z_1,z_2}$ first coalesces with another geodesic automatically belongs to the above set.

\newpage
{\footnotesize
{\renewcommand{\arraystretch}{1.10}
  \begin{longtable}{|l|l|l|}
 \hline {\bf Notation}
  & {\bf  Defined in}
  & {\bf Informal Description}\\
\hline  
  $\gamma\in (0,2)$
  &   Section \ref{sec:setup}
  &   LQG parameter.
  \\
  \hline

$\nu>0$ 
&   Proposition \ref{thm:theorem_uniqueness_lqg_metric}  
  &  $\mathfrak{C}_{r}^{\bz,\bw}(z)$ is produced with an $r\in [\varepsilon^{1+\nu},\varepsilon]$.\\ \hline

  $u_+ ,  v_+ $
  &   Definition \ref{def:1*} 
  &   starting and ending points of the one of the geodesics making the $\scX$.\\\hline
  $ u_- ,  v_- $
  &   Definition \ref{def:1*}
  &   starting and ending points of the other geodesic making the $\scX$.\\ \hline

  $P_+$ (resp.\ $P_-$)
  &   Definition \ref{def:1*} 
  &  The unique $D_h$-geodesic from $u_+$ to $v_+$ (resp.\ from $u_-$ to $v_-$).\\\hline

    $u ,  v \in \mathbb{C}$
  &   Definition \ref{def:1*}  
  & the \emph{random} initial and final points of $P_+\cap P_-$ \\  \hline

$\tau_u^+,\tau_u^-,\tau_v^+,\tau_v^-$
  &   Proposition~\ref{prop:definition_of_F_r(z)}
  & times such that $P_+(\tau_u^+)=u,P_-(\tau_u^-)=u,P_+(\tau_v^+)=v,P_-(\tau_v^-)=v$\\ \hline
  
   $\Phi \in (0,1/10)$ 
&   Lemma~\ref{lem:1*} 
 &  Parameter governing the length of $P_+\cap P_-$ compared to $P_+$ and $P_-$.\\ 
 \hline
 
  $b>0$
  &   Lemma \ref{lem:1*}  
  &   lower bound on the distance between any pair of points in $\{u_+,u_-,v_+,v_-,u,v\}$ \\  &  
  &   and between $P_+\cap P_-$ and $\{u_+ ,  v_+ ,  u_- ,  v_-\}$.\\ \hline

 $\mathcal{A}_{u_+}$ (resp.\ $\mathcal{A}_{u_-}$)
  &   Proposition~\ref{prop:definition_of_F_r(z)}
  &  the \emph{deterministic} collection of squares containing, on the event $F_r(z)$,\\ &  
  & the arm of the $\scX$ connecting $u_+$ and $u$     (resp.\ $u_-$ and $u$ ).\\ \hline

$\mathcal{A}_{v_+}$ (resp.\ $\mathcal{A}_{v_-}$)
&   Proposition~\ref{prop:definition_of_F_r(z)}
  &  the \emph{deterministic} collection of squares containing, on the event $F_r(z)$,\\ &  
  & the arm of the $\scX$ connecting $v_+$ and $v$     (resp.\ $v_-$ and $v$ ).\\ \hline

  $\mathcal{M}$ 
  &   Proposition~\ref{prop:definition_of_F_r(z)}
  & the \emph{deterministic} collection of squares containing, on the event $F_r(z)$, $P_+\cap P_-$.\\ \hline

 $Q_u$ (resp.\ $Q_v$)
&   Proposition~\ref{prop:definition_of_F_r(z)}
  & \emph{deterministic} path in $B_{A^2 r}(z)$ connecting $z-A^2 r$ (resp.\ $z+A^2 r$)\\
 $I_u$ (resp.\ $I_v$) & & to $S_u$ (resp.\ $S_v$) via $x_u\in I_u\subseteq \partial S_u$ (resp.\ $x_v\in I_v\subseteq \partial S_v$).\\ \hline

  $\delta_1>0$ 
  &  Proposition~\ref{prop:definition_of_F_r(z)}
  &   size of the deterministic squares $S_u,S_v$ containing $u$ and $v$ on the event $F_r(z)$.\\ \hline

$\delta_2\ll\delta_1$ 
  &  Proposition~\ref{prop:definition_of_F_r(z)}
  & Length of $I_u,I_v$ and a lower bound \\ 
  & &on the distance between certain parts of the coarse-grained $\scX$. %
  
  \\ \hline

$\delta_3 \ll\delta_2 $ 
&   Proposition~\ref{prop:definition_of_F_r(z)}
  &   size of the squares intersecting the four arms and \\ 
  &  
 &  the common segment $P_+\cap P_-$ of the $\scX$.  \\   \hline

$\mathcal{D}_u$ (resp.\ $\cD_v$)
&   Proposition~\ref{prop:definition_of_F_r(z)}
  &   union of the $\delta_3 r$ squares covering $Q_u$ (resp.\ $Q_v$).\\ \hline
  
$V_r(z)$ 
&   \eqref{eq:26} 
  &  Interior of the set $\cA_{u_+}\cup\cA_{u_-}\cup\cM\cup\cA_{v_+}\cup\cA_{v_-} \cup S_u\cup S_v$.  \\
    &  
  & On $F_r(z)$, the entire $\scX$ resides in the deterministic set $V_r(z)$ \\ \hline

 $O_r(z)$ 
&   \eqref{eq:95}
  &  $\cD_u\cup V_r(z)\cup \cD_v$\\ \hline

$A>0$ 
&   Proposition~\ref{prop:definition_of_F_r(z)}
  &   $V_r(z)$ lies inside $B_{Ar}(z)$. \\
  & & On $F_r(z)$, geodesics inside $B_{Ar}(z)$ are determined by the field in $B_{A^2 r}(z)$\\ \hline

$\rho>0$ 
&   Section \ref{sec:tubes}  
  &   The ball $B_{2r}(0)$ contains disjoint balls $B_{\rho r}(z)$ for $z\in \cZ$. \\ 
  &  
  &  For each $z\in \cZ$, we consider the event $F_{\rho r/A^2}(z)$. \\ \hline

$\delta \in (0,1)$ 
&   Section \ref{sec:tubes}  
    &   lower bound on the Euclidean distance between the locations\\
    & Lemma \ref{lem:Er-conds}& where metric balls around $\bz,\bw$ intersect $\partial B_{3r}(0)$.\\ \hline

$\zeta < \delta_3\rho/(4A^2)$ 
&   Section \ref{sec:bump_functions}  
  &  The bump function $f_r^{x,y}$ is supported on $B_{\zeta r}(U_{r}^{x,y})$ \\\hline

    $\theta<\zeta/100$ 
&    \eqref{eq:1}
    & The bump function $g_r^x$ is supported on $B_{\theta^2 r}(W_r^x)$.\\ \hline

$K_f\ll K_g, \phi_r^{x,y}$ 
&    \eqref{eq:1}  
  & $\phi_r^{x,y}:= K_f f_r^{x,y} + K_g (g_r^x + g_r^y)$\\ \hline %

$\Delta \in (0,1)$ 
&    \eqref{it:lqg_euclidean_distance_small} in $E_r(0)$  
  &   upper bound on LQG distances between points on $\partial B_{3r}(0)$ \\ 
  &  Lemma \ref{lem:Er-conds}
  &  with Euclidean distance at most $\delta r$.\\  \hline

$B>0$ 
&    \eqref{it:lqg_distance_upper_bound}
and  \eqref{it:lqg_distance_thin_ngbd}
 
  &   upper bound on the LQG diameters of the domains $U_r^{x,y}$\\
    &  in $E_r(0)$
  &  and lower bound on the LQG length of any\\
    &  
   &  Euclidean path of Euclidean diameter at least $\delta_3 \rho r/ (10^4 A^2)$ \\   
     &  
  &  contained in a small Euclidean neighborhood of $\partial U_r^{x,y}$.\\ \hline

 $\alpha \in (0,1)$ 
&    \eqref{it:lqg_distance_lower_bound} in $E_r(0)$
  &   lower bound on LQG distances between any two points \\ 
  &  
  &  on $\mathbbm{A}_{r/4 ,  r}(0)$ with Euclidean distance at least $\zeta r$.\\  \hline

$M>0$ 
&    \eqref{it:lqg_distance_thin_tube}  in $E_r(0)$
  &   upper bound on the LQG diameter of $W_r^x$.\\ \hline

$\Lambda_0>0$ 
&    \eqref{it:bump_function_dirichlet_energy} in $E_r(0)$ 
  &   upper bound on $|(h,\phi)_{\nabla}| +\frac{1}{2}|(\phi,\phi)_{\nabla}|$ for the bump functions $\phi \in \mathcal{G}_r$.\\ \hline

$\Lambda>0$ 
&   Section \ref{sec:geodesic_goes_through_chi}  
  &   upper bound on the Radon-Nikodym derivative of the \\
   &  
  &   
  field $h-\phi$ with respect to the field $h$ for all $\phi \in \mathcal{G}_r$.\\ \hline

\hline
    \caption{Parameters and sets from Section \ref{sec:construction_of_chi}}
    \label{table}
\end{longtable}
}
}

\printbibliography
 \end{document}